\documentclass[]{amsart}

\usepackage{amsmath}
\usepackage{amsthm}

\usepackage[vcentermath, enableskew]{youngtab}
\usepackage{tikz}
\usepackage{mathtools}

\usepackage[colorlinks=true, pdfstartview=FitV, linkcolor=blue, citecolor=blue, urlcolor=blue]{hyperref}

\usepackage{enumerate}

\usepackage{tikz}

\usepackage{wasysym}

\usepackage{amssymb}

\usepackage{multicol}

\usepackage{amsmath,amscd}

\usepackage{young}

\usepackage{ytableau}
\usepackage{multicol}
\usepackage{listings}
\usepackage{color}
\definecolor{lightgray}{gray}{0.75}

\newcommand\greybox[1]{%
  \vskip\baselineskip%
  \par\noindent\colorbox{lightgray}{%
    \begin{minipage}{\textwidth}#1\end{minipage}%
  }%
  \vskip\baselineskip%
}

\newtheorem{claim}{Claim}
\newtheorem{conjecture}{Conjecture}
\newtheorem{conj}{Conjecture}

\newtheorem{theorem}{Theorem}
\newtheorem{corollary}{Corollary}
\newtheorem{lemma}{Lemma}

\newtheorem{result}{Result}

\theoremstyle{definition}
\newtheorem{procedure}{Procedure}
\newtheorem{remark}{Remark}
\newtheorem{definition}{Definition}
\newtheorem{example}{Example}
\begin{document}

\title[A conjectured formula for the rational $q,t$-Catalan polynomial]{A conjectured formula for the rational $q,t$-Catalan polynomial}

\author[G.~Hawkes]{Graham Hawkes}



\begin{abstract}
We conjecture a formula  for the rational $q,t$-Catalan polynomial $\mathcal{C}_{r/s}$ that is symmetric in $q$ and $t$ by definition.  The conjecture posits that  $\mathcal{C}_{r/s}$ can be written in terms of symmetric monomial strings indexed by maximal Dyck paths.  We show that for any finite $d^*$, giving a combinatorial proof of our conjecture on the infinite set of functions $\{ \mathcal{C}_{r/s}^d: r\equiv 1 \mod s, \,\,\, d \leq d^*\}$ is equivalent to a finite counting problem. 
\end{abstract}

\maketitle

\section{Introduction}

The rational $q,t$-Catalan polynomial is defined combinatorially in \cite{LW09} (definition 21 of section 7).  This is a generalization of the classical $q,t$-Catalan polynomial, the combinatorics of which had been  studied earlier in places such as \cite{GH02}, \cite{HA03}, \cite{HHLRU}, and \cite {HL05} (see also \cite{HA08}).  A  famous open problem in algebraic combinatorics is to give a combinatorial proof of the  symmetry in $q$ and $t$ of these polynomials. In this paper, we  conjecture and provide theoretical and computational evidence for  a formula (which is $q,t$-symmetric by definition) for the rational $q,t$-Catalan polynomial, and give combinatorial proofs of its correctness in certain cases (see Result \ref{res}).

More precisely, let $\mathbf{D}_{r/s}$ be the set of rational Dyck paths from $(0,0)$ to $(s,r)$ (that is, the set of integral lattice paths composed only of North and East steps from $(0,0)$ to $(s,r)$ that stay below the line between these two points). The rational $q,t$-Catalan polynomial is defined in \cite{LW09} as:
\begin{eqnarray*}
\mathcal{C}_{r/s}:=  \sum_{x \in \mathbf{D}_{r/s}} q^{area(x)} t^{dinv(x)}
\end{eqnarray*}
where $area$ and $dinv$\footnote{
The statistics of $area$ and $dinv$ appear in \cite{HA03} and are generalized to the rational case in \cite{LW09}.
In section 9 we give  a self-contained formulation of this conjecture  (see Conjecture \ref{con} of section 9).  The statistic of $dinv$ is replaced with the use of the statistic $degr$ in that section.  They are related by $degr=M-area-dinv$ where $M$ is the number of boxes fully contained in the triangle $(0,0), (s,0), (s,r)$.} are certain statistics on rational Dyck paths.  Then, writing $[a,b]_{q,t}=q^{a}t^{b}+q^{a+1}t^{b-1}+\cdots+q^{b-1}t^{a+1}+q^{b}t^{a}$ and  $(a,b)_{q,t}=[a+1,b-1]_{q,t}$
\greybox{
\begin{conj}
Let $\mathbf{T}_{r/s} \subseteq \mathbf{D}_{r/s}$ denote the subset of paths that are maximal, where a path is defined as maximal if it passes as close as possible to the bounding diagonal.  Then
we have
\begin{eqnarray}\label{mon}
\mathcal{C}_{r/s}=\sum_{x \in \mathbf{T}_{r/s}^+} \big[area(x),dinv(x)\big]_{q,t} - \sum_{x \in \mathbf{T}_{r/s}^-} \big(dinv(x),area(x)\big)_{q,t} 
\end{eqnarray}
where if $x \in \mathbf{T}_{r/s}$ then $x \in \mathbf{T}_{r/s}^+$ if  $area(x)\leq dinv(x)$ and $x \in \mathbf{T}_{r/s}^-$ otherwise.
\end{conj}
}

\begin{example}                                                    
There are exactly $7$ rational Dyck paths of length 3 and height 5, which are shown below.  Of these, precisely 2 are maximal, which are indicated by the black circles.
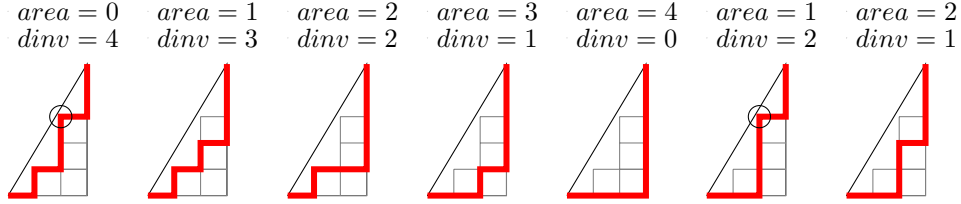
\begin{multicols}{7}
\noindent
\begin{tikzpicture}[scale=0.35] 
 \draw (0, 0)--(3, 5)--(3,0)--(0,0); 
 \draw[step=1cm,gray] (1,0) grid (3,1); 
 \draw[step=1cm,gray] (2,1) grid (3,3); 
 \draw[red, line width=0.8mm](0,0)--(1,0)--(1,1)--(2,1)--(2,3)--(3,3)--(3,5);
 \filldraw[black] (0,7) circle (0.01pt) node[anchor=west]{$area=0$}; 
 \filldraw[black] (0,6) circle (0.01pt) node[anchor=west]{$dinv=4$}; 
  \draw[black] (2,3) circle (12  pt); 
\end{tikzpicture}
\vfill\null
\columnbreak
\noindent
\begin{tikzpicture}[scale=0.35] 
 \draw (0, 0)--(3, 5)--(3,0)--(0,0); 
 \draw[step=1cm,gray] (1,0) grid (3,1); 
 \draw[step=1cm,gray] (2,1) grid (3,3); 
 \draw[red, line width=0.8mm](0,0)--(1,0)--(1,1)--(2,1)--(2,2)--(3,2)--(3,5);
 \filldraw[black] (0,7) circle (0.01pt) node[anchor=west]{$area=1$}; 
 \filldraw[black] (0,6) circle (0.01pt) node[anchor=west]{$dinv=3$}; 
\end{tikzpicture}
\vfill\null
\columnbreak
\noindent
\begin{tikzpicture}[scale=0.35] 
 \draw (0, 0)--(3, 5)--(3,0)--(0,0); 
 \draw[step=1cm,gray] (1,0) grid (3,1); 
 \draw[step=1cm,gray] (2,1) grid (3,3); 
 \draw[red, line width=0.8mm](0,0)--(1,0)--(1,1)--(2,1)--(2,1)--(3,1)--(3,5);
 \filldraw[black] (0,7) circle (0.01pt) node[anchor=west]{$area=2$}; 
 \filldraw[black] (0,6) circle (0.01pt) node[anchor=west]{$dinv=2$}; 
\end{tikzpicture}
\vfill\null
\columnbreak
\noindent
\begin{tikzpicture}[scale=0.35] 
 \draw (0, 0)--(3, 5)--(3,0)--(0,0); 
 \draw[step=1cm,gray] (1,0) grid (3,1); 
 \draw[step=1cm,gray] (2,1) grid (3,3); 
 \draw[red,  line width=0.8mm](0,0)--(1,0)--(1,0)--(2,0)--(2,1)--(3,1)--(3,5);
 \filldraw[black] (0,7) circle (0.01pt) node[anchor=west]{$area=3$}; 
 \filldraw[black] (0,6) circle (0.01pt) node[anchor=west]{$dinv=1$}; 
\end{tikzpicture}
\vfill\null
\columnbreak
\noindent
\begin{tikzpicture}[scale=0.35] 
 \draw (0, 0)--(3, 5)--(3,0)--(0,0); 
 \draw[step=1cm,gray] (1,0) grid (3,1); 
 \draw[step=1cm,gray] (2,1) grid (3,3); 
 \draw[red,  line width=0.8mm](0,0)--(1,0)--(1,0)--(2,0)--(2,0)--(3,0)--(3,5);
 \filldraw[black] (0,7) circle (0.01pt) node[anchor=west]{$area=4$}; 
 \filldraw[black] (0,6) circle (0.01pt) node[anchor=west]{$dinv=0$}; 
\end{tikzpicture}
\vfill\null
\columnbreak
\noindent
\begin{tikzpicture}[scale=0.35] 
 \draw (0, 0)--(3, 5)--(3,0)--(0,0); 
 \draw[step=1cm,gray] (1,0) grid (3,1); 
 \draw[step=1cm,gray] (2,1) grid (3,3); 
 \draw[red,  line width=0.8mm](0,0)--(1,0)--(1,0)--(2,0)--(2,3)--(3,3)--(3,5);
 \filldraw[black] (0,7) circle (0.01pt) node[anchor=west]{$area=1$}; 
 \filldraw[black] (0,6) circle (0.01pt) node[anchor=west]{$dinv=2$};
  \draw[black] (2,3) circle (12pt); 
\end{tikzpicture}
\vfill\null
\columnbreak
\noindent
\begin{tikzpicture}[scale=0.35] 
 \draw (0, 0)--(3, 5)--(3,0)--(0,0); 
 \draw[step=1cm,gray] (1,0) grid (3,1); 
 \draw[step=1cm,gray] (2,1) grid (3,3); 
 \draw[red,  line width=0.8mm](0,0)--(1,0)--(1,0)--(2,0)--(2,2)--(3,2)--(3,5);
 \filldraw[black] (0,7) circle (0.01pt) node[anchor=west]{$area=2$}; 
 \filldraw[black] (0,6) circle (0.01pt) node[anchor=west]{$dinv=1$}; 
\end{tikzpicture}
\end{multicols}

Therefore the conjecture states (correctly) that:
\begin{eqnarray*}
\mathcal{C}_{5/3}= [0,4]_{q,t}+[1,2]_{q,t}=q^0t^4+q^1t^3+q^2t^2+q^3t^1+q^4t^0+q^1t^2+q^2t^1
\end{eqnarray*}
For a larger example (where $\mathbf{T}_{r/s}^- \neq \emptyset$) see example \ref{conex}.
\end{example}

Let us now outline a combinatorial procedure that  would (we will carry out this procedure in some but not all cases) prove our conjecture:  As noted in  \cite{LLL18} and \cite{HLLL20} it is convenient to consider the homogeneous parts of $\mathcal{C}_{r/s}$ separately.   Therefore, let  $\mathcal{C}_{r/s}^d$ be the part of $\mathcal{C}_{r/s}$ of total  degree $M-d$ in $q$ and $t$ where $M$ is the number of boxes fully contained in the triangle $(0,0), (s,0), (s,r)$. Further, define   $\mathbf{D}_{r/s}^d=\{x: x \in \mathbf{D}_{r/s}, \,\,\, area(x)+dinv(x)=M-d\}$ (this is equal to  the subset of $\mathbf{D}_{r/s}$ that is counted by $\mathcal{C}_{r/s}^d$). Similarly, let us define $\mathbf{T}_{r/s}^d = \{x: x \in \mathbf{T}_{r/s}, \,\,\, area(x)+dinv(x)=M-d\}$.   We can now describe our procedure:
\begin{procedure}\label{pro}
Perform the following steps.
\begin{enumerate}
\item {For each element $x \in \mathbf{T}_{r/s}^d$, construct a string of elements $\{x=x_0,x_1,\ldots,x_b\}$ where $area(x_{i+1})=area(x_i)+1$ and $dinv(x_{i+1})=dinv(x_i)-1$ for $i \in [0,b)$ and show that these strings partition $\mathbf{D}_{r/s}^d$.}
\item {Let $\mathbf{B}_{r/s}^d$ denote the set composed of all elements of $\mathbf{D}_{r/s}^d$ that are the terminal (rightmost) element of some string constructed in the previous step. Construct a bijection $\mathbf{T}_{r/s}^d \rightarrow \mathbf{B}_{r/s}^d$ that interchanges the statistics of $area$ and $dinv$. }
\end{enumerate}
\end{procedure}

The first step of Procedure \ref{pro} is inspired by previous attempts to create a symmetric string  decomposition of Dyck paths such as were made \cite{LLL18} and \cite{HLLL20}.  The second step is inspired by previous attempts to solve the problem by giving an explicit bijection that interchanges $area$ and $dinv$. The approach that we use combines these two methods.  In particular, if our strings happened to be symmetric  (our strings are not  symmetric in general) our approach would be equivalent to the first one.  If our strings happened to all be singletons, our approach would be equivalent to the second one.

The main result of the paper is as follows: In the case that $r\equiv  1 \mod s$, for each positive integer, $d^*$, we give a set of instructions $I(d^*)$ to carry out Procedure \ref{pro} for all $\mathbf{D}_{r/s}^d$ such that $d \leq d^*$.   Moreover, we can prove that these instructions correctly perform Procedure \ref{pro} if we are allowed to perform a finite base case check $\beta(d^*)$ whose complexity depends on $d^*$ (but, importantly, is independent of $r$, $s$, and $d$).

 More precisely, the results of this paper show that:

\greybox{
\begin{result}\label{res}
For any integer, $d^*$, providing a  set of instructions, $I(d^*)$ and proving that they perform Procedure \ref{pro} correctly for all values of $(r,s,d)$ such that $r\equiv  1 \mod s$ and $d \leq d^*$ is equivalent to carrying out a finite computation, $\beta(d^*)$, that depends only on $d^*$.
\end{result}
}

 This means that, in the case $r \equiv  1 \mod s$, the problem of giving a combinatorial proof of the symmetry of the rational $q,t$-Catalan polynomial is reduced to the problem of finding a way to check $\beta(d^*)$ (combinatorially) for all $d^*$ simultaneously.

For concreteness, we will use the value of $d^*=20$ in the proofs, but it will be easy to see we could equally well use another number.  We choose $d^*=20$ because the time needed to \emph{actually} complete the base case  check is still reasonably small. We note that  the approach of bounding $d$ was used to combinatorially prove the symmetry of the  classical  (when $r=s+1$)  $q,t$-Catalan polynomial in  \cite{LLL18} and \cite{HLLL20}. The former of these restricted to $d \in\{0,1,2,3,4,5,6,7,8,9\}$ and the latter extended this to include $d=10$ and $d=11$.  \textbf{The contents of our paper, therefore, restricted to the classical case, extend the results of  \cite{LLL18} and \cite{HLLL20} from 11 up to any number $d^*$ for which one has time to perform the finite computation, $\beta(d^*)$.} It is important to understand that completing the finite computation $\beta(d^*)$ does more than just prove Conjecture \ref{con}--it allows Procedure \ref{pro} itself to be completed.

Since  Procedure \ref{pro} is composed of two steps, both of which are entirely combinatorial, one may be surprised that the formula (equation \ref{mon}) that  follows from it  includes negative terms. The fact that the negative terms appear is due to a certain form of overcounting that occurs when we combine the first and second steps of Procedure \ref{pro}.  We now explain how this works in detail (the argument below applies when $M-d$ is odd, a nearly identical one applies when $M-d$ is even).

Suppose we have completed the first step of Procedure \ref{pro} and have a list of strings that partitions $\mathbf{D}_{r/s}^d$.  Any string that has  at least one element, $x$, with $area(x)-dinv(x) \in \{-1,1\}$ we will call \emph{medial}. A string that only includes elements $x$ with $area(x)<dinv(x)-1$ we will call \emph{upper} and any string that only includes elements $x$ with $dinv(x)<area(x)-1$ we will call \emph{lower}.    First we will count the ``missing" elements, that is, for each upper or lower string we will count the minimal number of elements that would need to be added to that string to make it medial. (For instance, in order to make the leftmost string in Figure \ref{lines} medial we would need to add two elements, corresponding to the monomials $q^6t^9$ and $q^7t^8$.)  Next, because of the bijection $\mathbf{T}_{r/s}^d \rightarrow \mathbf{B}_{r/s}^d$, each lower string can be paired with an upper string that is missing the equal and opposite elements.  Since for each $x \in (\mathbf{T}_{r/s}^d)^-$ there is one lower string (and one upper string paired with it) we have one full symmetric string of missing elements which corresponds to the monomials appearing in $(dinv(x),area(x))_{q,t}$.  From all this it follows that the missing elements are precisely counted by the negative terms in equation \ref{mon}.

We now assume that all strings have been (minimally) lengthened to become medial.  The bijection $\mathbf{T}_{r/s}^d \rightarrow \mathbf{B}_{r/s}^d$ can now be used again to match the top part of each medial string with the bottom part of another (possibly the same) medial string.  Since each $x \in (\mathbf{T}_{r/s}^d)^+$ corresponds to the top part of a medial string (each of which is paired with the bottom part of a medial string) we have one full symmetric string of elements which corresponds to the monomials appearing in $[area(x),dinv(x)]_{q,t}$ for each  $x \in (\mathbf{T}_{r/s}^d)^+$.   Since all elements appear in such a fashion,  the elements in the  set of lengthened strings are precisely counted  by the positive terms in equation \ref{mon}.

Since the overcounting committed by the lengthening of strings can be corrected by subtracting the count associated to the missing elements, it follows that together the positive and negative  terms of equation \ref{mon} correctly count all elements of $\mathbf{D}_{r/s}^d$.  Hence the completion of Procedure \ref{pro} is sufficient to give a combinatorial proof of Conjecture \ref{con}.  See Figure \ref{lines}\footnote{Figure \ref{lines} is a hypothetical example designed for illustrative purposes only.}  for a visualization
 of this discussion.

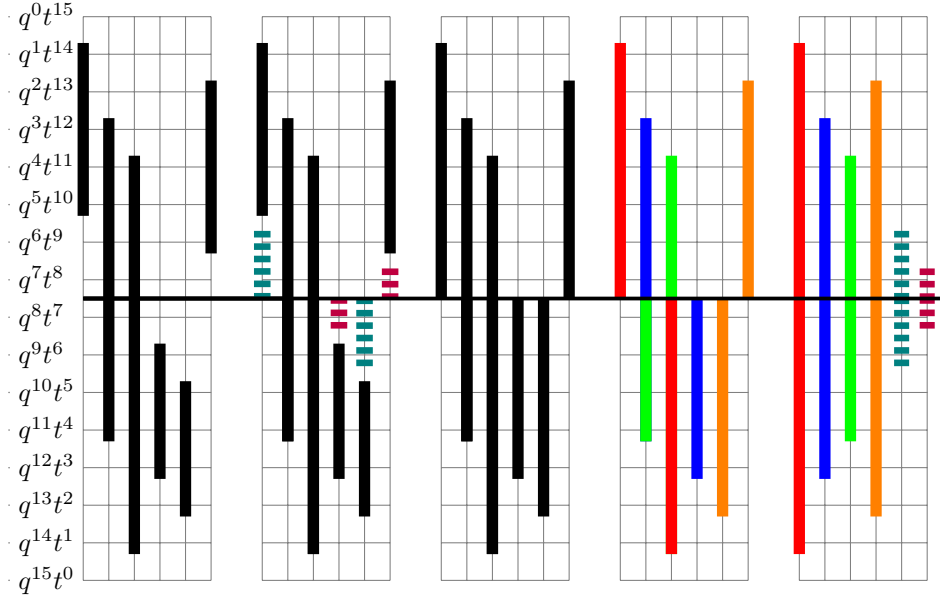
\begin{figure}

\begin{tikzpicture}[xscale=0.34, yscale=0.5] 
 \draw[step=1cm,gray] (-1,-7) grid (4,8); 

\draw[black, line width=1.5mm](-1,7.3)--(-1,2.7);
\draw[black, line width=1.5mm](0,5.3)--(0,-3.3);
\draw[black, line width=1.5mm](1,4.3)--(1,-6.3);
\draw[black, line width=1.5mm](2,-0.7)--(2,-4.3);
\draw[black, line width=1.5mm](3,-1.7)--(3,-5.3);
\draw[black, line width=1.5mm](4,6.3)--(4,1.7);

\filldraw[black] (-3.9,8) circle (0.1pt) node[anchor=west]{$q^0t^{15}$}; 
\filldraw[black] (-3.9,7) circle (0.1pt) node[anchor=west]{$q^1t^{14}$}; 
\filldraw[black] (-3.9,6) circle (0.1pt) node[anchor=west]{$q^2t^{13}$}; 
\filldraw[black] (-3.9,5) circle (0.1pt) node[anchor=west]{$q^3t^{12}$}; 
\filldraw[black] (-3.9,4) circle (0.1pt) node[anchor=west]{$q^4t^{11}$}; 
\filldraw[black] (-3.9,3) circle (0.1pt) node[anchor=west]{$q^5t^{10}$}; 
\filldraw[black] (-3.9,2) circle (0.1pt) node[anchor=west]{$q^6t^9$}; 
\filldraw[black] (-3.9,1) circle (0.1pt) node[anchor=west]{$q^7t^8$}; 
\filldraw[black] (-3.9,0) circle (0.1pt) node[anchor=west]{$q^8t^7$}; 
\filldraw[black] (-3.9,-1) circle (0.1pt) node[anchor=west]{$q^9t^6$}; 
\filldraw[black] (-3.9,-2) circle (0.1pt) node[anchor=west]{$q^{10}t^5$}; 
\filldraw[black] (-3.9,-3) circle (0.1pt) node[anchor=west]{$q^{11}t^4$}; 
\filldraw[black] (-3.9,-4) circle (0.1pt) node[anchor=west]{$q^{12}t^3$}; 
\filldraw[black] (-3.9,-5) circle (0.1pt) node[anchor=west]{$q^{13}t^2$}; 
\filldraw[black] (-3.9,-6) circle (0.1pt) node[anchor=west]{$q^{14}t^1$}; 
\filldraw[black] (-3.9,-7) circle (0.1pt) node[anchor=west]{$q^{15}t^0$};

 \draw[step=1cm,gray] (6,-7) grid (11,8); 
\draw[line width=0.5mm](-1,0.5)--(7,0.5);
\draw[black, line width=1.5mm](6,7.3)--(6,2.7);
\draw[teal, line width=2.2mm, dash pattern=on 2.5pt off 2.2pt ](6,2.3)--(6,0.5);
\draw[black, line width=1.5mm](7,5.3)--(7,-3.3);
\draw[black, line width=1.5mm](7,0.5)--(7,-3.3);
\draw[black, line width=1.5mm](8,4.3)--(8,-6.3);
\draw[purple, line width=2.2mm, dash pattern=on 2.5pt off 2.2pt ](9,-0.3)--(9,0.5);
\draw[black, line width=1.5mm](8,0.5)--(8,-6.3);
\draw[black, line width=1.5mm](9,-0.7)--(9,-4.3);
\draw[black, line width=1.5mm](10,-1.7)--(10,-5.3);
\draw[teal,line width=2.2mm, dash pattern=on 2.5pt off 2.2pt ](10,-1.3)--(010,0.5);
\draw[black, line width=1.5mm](11,6.3)--(11,1.7);
\draw[purple, line width=2.2mm, dash pattern=on 2.5pt off 2.2pt ](11,1.3)--(11,0.5);

 \draw[step=1cm,gray] (13,-7) grid (18,8); 
\draw[line width=0.5mm](-1,0.5)--(7,0.5);
\draw[black, line width=1.5mm](13,7.3)--(13,0.5);

\draw[black, line width=1.5mm](14,5.3)--(14,-3.3);
\draw[black, line width=1.5mm](14,0.5)--(14,-3.3);
\draw[black, line width=1.5mm](15,4.3)--(15,-6.3);
\draw[black, line width=1.5mm](15,0.5)--(15,-6.3);
\draw[black, line width=1.5mm](16,0.5)--(16,-4.3);

\draw[black, line width=1.5mm](17,0.5)--(17,-5.3);

\draw[black, line width=1.5mm](18,6.3)--(18,0.5);

 \draw[step=1cm,gray] (20,-7) grid (25,8); 

\draw[red, line width=1.5mm](20,7.3)--(20,0.5);

\draw[blue, line width=1.5mm](21,5.3)--(21,-3.3);
\draw[green, line width=1.5mm](21,0.5)--(21,-3.3);
\draw[green, line width=1.5mm](22,4.3)--(22,-6.3);
\draw[red, line width=1.5mm](22,0.5)--(22,-6.3);
\draw[blue, line width=1.5mm](23,0.5)--(23,-4.3);

\draw[orange, line width=1.5mm](24,0.5)--(24,-5.3);

\draw[orange, line width=1.5mm](25,6.3)--(25,0.5);

 \draw[step=1cm,gray] (27,-7) grid (32,8); 

\draw[red, line width=1.5mm](27,7.3)--(27,-6.3);
\draw[blue, line width=1.5mm](28,5.3)--(28,-4.3);
\draw[green, line width=1.5mm](29,4.3)--(29,-3.3);
\draw[orange, line width=1.5mm](30,6.3)--(30,-5.3);

\draw[teal, line width=2mm, dash pattern=on 2.5pt off 2.2pt](31,2.3)--(31,0.5);
\draw[teal, line width=2mm, dash pattern=on 2.5pt off 2.2pt](31,-1.3)--(31,0.5);
\draw[purple, line width=2mm, dash pattern=on 2.5pt off 2.2pt](32,1.3)--(32,0.5);
\draw[purple, line width=2mm, dash pattern=on 2.5pt off 2.2pt](32,-0.3)--(32,0.5);

\draw[line width=0.5mm](-1,0.5)--(32.5,0.5);

\end{tikzpicture}
\caption{Visualization of Procedure \ref{pro}} \label{lines}

\medskip
\small
\begin{enumerate}
\item The strings are constructed.
\item The additions are shown along with their pairing (via coloration).
\item The lengthened strings are shown.
\item The tops and bottoms of the lengthened strings are paired (via coloration).
\item The total elements counted (solid) and elements overcounted (dotted) are shown.
\end{enumerate}

\end{figure}

\subsection{Contents}\label{contents}

The paper is organized as follows:

\begin{itemize}
\item   In section 2 we introduce the definitions and establish a few basic facts. 
\item  In section 3 we include some technical results about the form of degree bounded Dyck paths. 
\item In section 4 we define the cyclic maps on Dyck paths and describe their basic properties--this will be helpful for the construction of strings.  
\item  In section 5 we introduce bounded partitions and $0-1$ matrices and establish their relationship to Dyck paths. 
\item  In section 6 we state and prove Theorem \ref{project}.
\item In section 7 we describe how to construct strings of Dyck paths and show how to extend them.
\item In section 8 we consider the base case of $\ell=2$. 
\item In section 9 we restate the conjecture and complete the proofs.
\item Appendix A gives the code needed to check the base case for the $d^*=20$ case.  
\item Appendix B gives the code needed to test the conjecture for any pair $(s,r)$ of relatively prime integers. 
\end{itemize}

Taken as a whole the contents of section 9 establish Result \ref{res} (once the reader has been convinced that the value of $d^*=20$ used in that section is arbitrary).  We should also mention that strictly speaking, the statement before Result \ref{res}   requires an application of a finite version of the axiom of choice.  Indeed, there are certain base cases where we show that the completion of Procedure \ref{pro} is possible but do not explicitly carry it out.  Instead of making an (arbitrary) set of rules to do this we simply trust that the reader accepts the axiom of choice for finite sets.

\section{Dyck paths, degree, and area}

\emph{
Sections 2 to 8 contain material only pertaining to the case  $s=(\ell+1)$ and $r=(\ell+1)m+1$.  In this case it is irrelevant whether we consider paths from $(0,0)$ to $(s,r)$ that stay below the line between these two points or the paths from $(0,0)$ to $(s,r-1)$ that stay weakly below the line between those two points because there is a bijection from the former set of paths to the latter given by deleting the last vertical step in the path.  It is somewhat notationally simpler to consider the latter so we do so in these sections.  However, the reader should be aware that the maximal paths are still defined in terms of the diagonal from $(0,0)$ to $(s,r)$.  Therefore, maximal paths in this setting are those whose first vertical step is as large as possible.  }

An $(\ell,m)$-Dyck path is a integral lattice path from $(0,0)$ to $(\ell+1,m(\ell+1))$ that stays (weakly) below the diagonal line between these two points. We can specify such a path $x$ using \emph{step coordinates} via the expression $x=(x_0,x_1,\ldots,x_{\ell})$ where  $x_0+\ldots+x_i\leq m(i+1)$ for $0 \leq i \leq \ell -1$ and $x_0+\ldots+x_{\ell}=m(\ell+1)$.

 Since the last step coordinate of an $(\ell,m)$-Dyck path is determined by the previous step coordinates we will sometimes write $x=(x_0,\ldots,x_{\ell-1},-)$.   In terms of lattice paths, the $x_i$ appearing in an $(\ell,m)$-Dyck path represent the number of vertical steps taken between each horizontal step in the physical path.

\begin{multicols}{2}
\begin{example}
A $(5,2)$-Dyck path $x$ is shown on the right.  
\begin{itemize}
\item In step coordinates we have $x=(1,3,0,2,2,4)$.
\item In position coordinates (defined in section 3) we have $x=[0,1,0,2,2,2]$.
\item We have $area(x)=7$ (see below).
\item We have $degr(x)=9$ (see below).
\end{itemize}
\end{example}
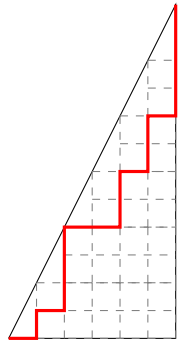
\columnbreak
\begin{eqnarray*}
 \begin{tikzpicture}[scale=0.37] 
 \draw (0, 0)--(6, 12)--(6,0)--(0,0); 
 \draw[step=1cm,gray,dashed] (1,0) grid (6,2); 
 \draw[step=1cm,gray,dashed] (2,2) grid (6,4); 
 \draw[step=1cm,gray,dashed] (3,4) grid (6,6); 
 \draw[step=1cm,gray,dashed] (4,6) grid (6,8); 
 \draw[step=1cm,gray,dashed] (5,8) grid (6,10); 
\draw[red, very thick](0,0)--(1,0)--(1,1)--(2,1)--(2,4)--(4,4)--(4,6)--(5,6)--(5,8)--(6,8)--(6,12);
\end{tikzpicture}
\end{eqnarray*}

\end{multicols}

The area of an $(\ell,m)$-Dyck path is defined to be $area(x)=M-(\ell x_0 + \cdots + 1 x_{\ell-1})$ where $M=m {{\ell+1}\choose 2}$, and the degree is defined to be $degr^+(x)+degr^-(x)$ where 

\begin{eqnarray*}
degr^+(x)=\sum_{1 \leq i \leq j < \ell}\delta_{ij}^-(x)\\
degr^-(x)=\sum_{1 \leq i \leq j < \ell}\delta_{ij}^+(x)
\end{eqnarray*}
where
\begin{eqnarray*}
 \delta_{ij}^+(x)=min(x_i,max(0,(x_i-m)+\cdots+(x_j-m)-1))\\
\delta_{ij}^-(x)=min(x_{i-1},max(0,(m-x_i)+\cdots+(m-x_j)))
\end{eqnarray*}

\begin{claim} \label{fitsout}
Suppose that $x=(x_1,\ldots,x_{\ell})$ is an $(\ell-1,m)$-Dyck path and that $x_0 \leq m$, then  then $x'=(x_0,x_1,\ldots,x_{\ell-1},-)$ is an $(\ell,m)$-Dyck path and  $degr(x') \leq x_0(\ell-1)+degr(x)$.  
\end{claim}
\begin{proof}
We have $(x_1-m)+\cdots + (x_j-m) \leq 0$ and $x_0 \leq m$ whence $(x_0-m)+\cdots + (x_j-m) \leq 0 $ so that $x'$ is an $(\ell,m)$-Dyck path.
We also have 
\begin{eqnarray*}
 degr(x')-degr(x)=\sum_{1 \leq j < \ell} \delta_{1j}^+(x')+ \delta_{1j}^-(x')
\end{eqnarray*}
 Now $(x_1-m)+\cdots+(x_j-m) \leq 0 $ implies that $\delta_{1j}^+(x')=0$ and  this along with $\delta_{1j}^-(x') \leq x_0$ implies each term in the sum is at most $x_0$ proving the claim.
\end{proof}

\begin{claim} \label{fitsin}
Suppose that $x=(x_0,x_1,\ldots,x_{\ell})$ is an $(\ell,m)$-Dyck path and that for all $1 \leq j <\ell $ we have $(x_0-m)+\cdots(x_j-m) \leq -m $ then $x'=(x_1,\ldots,x_{\ell-1},-)$ is an $(\ell-1,m)$-Dyck path and  $degr(x)=x_0(\ell-1)+degr(x')$.  
\end{claim}
\begin{proof}
We have $(x_0-m)+\cdots + (x_j-m) \leq -m $ whence $(x_1-m)+\cdots + (x_j-m) \leq 0 $ so that $x'$ is an $(\ell-1,m)$-Dyck path.
We also have 
\begin{eqnarray*}
 degr(x)-degr(x')=\sum_{1 \leq j < \ell} \delta_{1j}^+(x)+ \delta_{1j}^-(x)
\end{eqnarray*}
 Now $(x_1-m)+\cdots+(x_j-m) \leq 0 $ implies that $\delta_{1j}^+(x)=0$ and $(x_0-m)+\cdots + (x_j-m) \leq -m $ implies $\delta_{1j}^-(x)=x_0$ so each term in the sum is equal to $x_0$ proving the claim.
\end{proof}

Sometimes it will be convenient to compute $degr(x)$ in a different way.  To do this we define

\begin{eqnarray*}
\epsilon_{ij}^+(x)=min(m,max(0,(x_i-m)+\cdots+(x_j-m)-1))\\
\epsilon_{ij}^-(x)=min(m,max(0,(m-x_i)+\cdots+(m-x_j)))\\
\epsilon_{ij}^0(x)=max(0,(m-x_i)+\cdots+(m-x_j)-m)
\end{eqnarray*}

\begin{claim}\label{epsi}
Let $1 \leq j<\ell$
\begin{eqnarray}
\epsilon_{1j}^+(x)+\cdots+\epsilon_{jj}^+(x)=\delta_{1j}^+(x)+\cdots+\delta_{jj}^+(x)\label{equalplus}\\
\epsilon_{1j}^-(x)+\cdots+\epsilon_{jj}^-(x)-\epsilon_{0j}^0(x)=\delta_{1j}^-(x)+\cdots+\delta_{jj}^-(x)\label{equalminus}
\end{eqnarray}
\end{claim}

\begin{proof}
By considering the four possible cases:
\begin{itemize}
\item $x_i,m \leq (x_{i}-m)+\cdots+(x_j-m)-1$
\item $x_i,m \geq (x_{i}-m)+\cdots+(x_j-m)-1$
\item $x_i < (x_{i}-m)+\cdots+(x_j-m)-1<m$
\item $m <  (x_{i}-m)+\cdots+(x_j-m)-1<x_i$
\end{itemize}
it is easy to verify that
\begin{eqnarray*}
 \epsilon_{ij}^+(x)-\delta_{ij}^+(x)=max(0,(x_{i}-m)+\cdots+(x_j-m)-1-x_i)\nonumber\\-max(0,(x_i-m)+\cdots+(x_j-m)-1-m)\nonumber
\end{eqnarray*}
Thus if the difference between the left hand side and the right hand side of \ref{equalplus} is rewritten using this equation, all middle terms cancel and we are left with
\begin{eqnarray*}
max(0,-1-m)-max(0,(x_1-m)+\cdots+(x_j-m)-1-m)=0
\end{eqnarray*}
since $(x_1-m)+\cdots+(x_j-m) \leq m+1$ follows from the fact $x$ is an $(\ell,m)$-Dyck path. Similarly by considering the four possible cases:
\begin{itemize}
\item $x_{i-1},m \leq (m-x_{i})+\cdots+(m-x_j)$
\item $x_{i-1},m \geq (m-x_{i})+\cdots+(m-x_j)$
\item $x_{i-1} < (m-x_{i})+\cdots+(m-x_j)<m$
\item $m <  (m-x_{i})+\cdots+(m-x_j)<x_{i-1}$
\end{itemize}
one can see that
\begin{eqnarray*}
 \epsilon_{ij}^-(x)-\delta_{ij}^-(x)=max(0,(m-x_{i})+\cdots+(m-x_j)-x_{i-1})\nonumber\\-max(0,(m-x_i)+\cdots+(m-x_j)-m)\nonumber
\end{eqnarray*}
Thus if the difference between the left hand side and the right hand side of \ref{equalminus} is rewritten using this equation, all middle terms cancel and we are left with
\begin{eqnarray*}
-\epsilon_{(0)j}^0(x)+max(0,(m-x_{1})+\cdots+(m-x_j)-x_0)\nonumber\\-max(0,(m-x_j)-m)=0
\end{eqnarray*}
since the first two terms above are opposites and the third term is $0$.
\end{proof}
It follows from the previous claim that
\begin{claim}\label{split}
We have:
\begin{eqnarray*}
degr^+(x)=\sum_{1 \leq i\leq j<\ell} \epsilon_{ij}^+(x) \\
degr^-(x)=\sum_{1 \leq i\leq j<\ell} \epsilon_{ij}^-(x) - \sum_{1\leq j <\ell} \epsilon_{0j}^0(x) 
\end{eqnarray*}
\end{claim}

\begin{example}\label{comdeg}
Suppose that $\ell=5$ and that $m=5$ and that $x=(3,0,12,1,2,12)$ then we have \\

\noindent \mbox{$\delta_{11}^+=0$, $\delta_{12}^+=0$, $\delta_{13}^+=0$, $\delta_{14}^+=0$, $\delta_{22}^+=6$, $\delta_{23}^+=2$, $\delta_{24}^+=0$, $\delta_{33}^+=0$, $\delta_{34}^+=0$, $\delta_{44}^+=0$}  \\

\noindent \mbox{$\epsilon_{11}^+=0$, $\epsilon_{12}^+=1$, $\epsilon_{13}^+=0$, $\epsilon_{14}^+=0$, $\epsilon_{22}^+=5$, $\epsilon_{23}^+=2$, $\epsilon_{24}^+=0$, $\epsilon_{33}^+=0$, $\epsilon_{34}^+=0$, $\epsilon_{44}^+=0$}  \\

\noindent so that by either method $degr^+(x)=8$ whereas\\

\noindent \mbox{$\delta_{11}^-=3$, $\delta_{12}^-=0$, $\delta_{13}^-=2$, $\delta_{14}^-=3$, $\delta_{22}^-=0$, $\delta_{23}^-=0$, $\delta_{24}^-=0$, $\delta_{33}^-=4$, $\delta_{34}^-=7$, $\delta_{44}^-=1$ } \\

\noindent \mbox{$\epsilon_{11}^-=5$, $\epsilon_{12}^-=0$, $\epsilon_{13}^-=2$, $\epsilon_{14}^-=5$, $\epsilon_{22}^-=0$, $\epsilon_{23}^-=0$, $\epsilon_{24}^-=0$, $\epsilon_{33}^-=4$, $\epsilon_{34}^-=5$, $\epsilon_{44}^-=3$ } \\

\noindent and\\

\noindent \mbox{$\epsilon_{01}^0=2$, $\epsilon_{02}^0=0$, $\epsilon_{03}^0=0$, $\epsilon_{04}^0=2$  }\\

\noindent so that by either method $degr^-(x)=20$.\\

\end{example}

\begin{definition}
We define the set $\mathbf{D}_{\ell,m}$ to be the set of all $(\ell,m)$-Dyck paths.  We define $\mathbf{T}_{\ell,m}$ to be the subset of $\mathbf{D}_{\ell,m}$  such that if $x=(x_0,x_1,\ldots,x_{\ell})$ then $x_0=m$.   We define $\mathbf{D}_{\ell,m}^{<d}= \{x \in D_{\ell,m}: degr(x)<d\}$ and $\mathbf{T}_{\ell,m}^{<d}= \{x \in T_{\ell,m}: degr(x)<d\}$.
\end{definition}

\section{Properties of degree bounded Dyck paths}
In this section we establish some properties of $(\ell,m)$-Dyck paths with degree less than $(\ell-2)m$.  
\begin{claim}\label{bound1}
Let $x=(x_0,x_1,\ldots,x_{\ell})$.   Suppose $x_1+\cdots+x_{\ell-1}\leq (\ell-2)m$ and $x_0+x_{1} \geq m$ then $degr^-(x) \geq (\ell-2)m$.   
\end{claim}
\begin{proof}
First suppose that $x_0+x_1=m$.  Define
\begin{eqnarray*}
neg(x)= \left(\sum_{1 \leq i \leq j \leq \ell-1} \delta_{ij}^-(x) \right)-\delta_{11}^-(x) 
\end{eqnarray*}
For $1<k\leq \ell$ let 
\begin{eqnarray*}
x^k=(x_0,x_1,0,\ldots,0,x_{2}+\cdots+x_k,x_{k+1},\ldots,x_{\ell})
\end{eqnarray*}
It is easy to see that $\delta_{1j}^-(x^{\ell})+\delta_{2j}^-(x^{\ell})=m$ for $2 \leq j \leq \ell-1$ so that  $neg(x^{\ell}) \geq (\ell-2)m$. Now suppose $neg(x^{k}) \geq (\ell-2)m$  for some $k \in (2,\ell]$ and let $\mathtt{shift_k}(z_0,\ldots,z_{k-1},z_k,\ldots,z_{\ell})=(z_0,\ldots,z_{k-1}+1,z_k-1,\ldots,z_{\ell})$ for any $z$.  Then $x^{k-1}$ is obtained from $x^k$ by applying $\mathtt{shift_k}$ $(x_{2}+\cdots+x_{k-1})$ times. 

 The only terms that appear in $neg()$ that can decrease in value when one of these instances of $\mathtt{shift_k}$ is applied are $\delta_{1(k-1)}$, $\delta_{2(k-1)}$ and $\delta_{(k+1)j}$ for $j \in [k+1,\ell-1]$ and each of these values may decrease by at most $1$.  The only terms that appear in $neg()$ that can increase in value when one of these instances of $\mathtt{shift_k}$ is applied are $\delta_{kj}$ for $j \in [k,\ell-1]$ and each of these values may increase by at most $1$.  

Moreover, if the $t^{th}$ application of $\mathtt{shift_k}$ results in a decrease in $\delta_{1(k-1)}$ then $x_0+x_1+t >(k-1)m$ so that $t> (k-2)m$ which implies that  $\delta_{2(k-1)}$ is $0$ before and after the $t^{th}$ application of $\mathtt{shift_k}$.  Thus, at most one of $\delta_{1(k-1)}$ and $\delta_{2(k-1)}$ can decrease with each application of $\mathtt{shift_k}$.  Further, if  $t^{th}$ application of $\mathtt{shift_k}$ results in a decrease in either $\delta_{1(k-1)}$ or  $\delta_{2(k-1)}$ we have $x_1+t>(k-2)m$.  This in combination with the assumption that $x_1+\cdots+x_{\ell-1}\leq (\ell-2)m$ implies that $(x_2+\cdots+x_k-t)+x_{k+1}+\cdots+x_{\ell-1}<(\ell-k)m$ which implies that $\delta_{k(\ell-1)}$ increases with the $t^{th}$ application of $\mathtt{shift_k}$.

Finally, if  for some $j \in [k+1,\ell-1]$ the $t^{th}$ application of $\mathtt{shift_k}$ results in a decrease in $\delta_{(k+1)j}$ then $(x_2+\cdots+x_k-t)+x_{k+1}+\cdots+x_{j}<(j-k)m$ which implies that $(x_2+\cdots+x_k-t)+x_{k+1}+\cdots+x_{j-1}<(j-k)m$ so that  $\delta_{k(j-1)}$ must increase with the $t^{th}$ application of $\mathtt{shift_k}$.  It follows that $neg(x^{k-1}) \geq neg(x^k)$ so   by induction we see that $neg(x) \geq (\ell-2)m$ since $x^2=x$.

Now suppose that $x_0+x_1 \geq m$.  If $x_1>m$ then the argument above implies that $neg(0,m,x_2,\ldots,x_{\ell-1},-) \geq m(\ell-2)$ which then  implies that $neg(0,x_1,x_2,\ldots,x_{\ell-1},-) \geq (\ell-2)m$ which implies $neg(x_0,x_1,x_2,\ldots,x_{\ell-1},-) \geq (\ell-2)m$.  On the other hand, if $x_1 \leq m$ then the argument above implies that $neg(m-x_1,x_1,x_2,\ldots,x_{\ell-1},-) \geq (\ell-2)m$ so that $neg(x_0,x_1,x_2,\ldots,x_{\ell-1},-) \geq (\ell-2)m$.  But  $degr^-(x) \geq neg(x) \geq (\ell-2)m$.

\end{proof}

\begin{claim} \label{bound2}
Let $x=(x_0,\ldots,x_{\ell})$ and suppose that $i$ is minimal such that $i>0$ and $x_i+x_{i+1} \geq m$.  Suppose that $x_{i+1}+\cdots+x_{\ell-1}>(\ell-i)m+1$ and that either there exists some $2 \leq j \leq \ell-1$ such that $x_{0}+\cdots+x_{j}>jm$ or else $x_0+x_1 \geq m$.  Then $degr(x) \geq (\ell-2)m$. 
\end{claim}

\begin{proof}
 From the fact that $x_{i+1}+\cdots+x_{\ell-1}>(\ell-i)m+1$ we can deduce that  for $k \in [i+1,\ell-1]$  we have $\epsilon_{(i+1)k}^+(x)+\epsilon_{(k+1)(\ell-1)}^+(x)\geq m$ and also that $\epsilon_{(i+1)(\ell-1)}^+(x)=m$ so certainly $degr^+(x) \geq (\ell-i-1)m$ so we may from now on assume $i>1$.

Next, suppose that $x_0+x_1<m$. Since we also have $x_1+x_2<m, \ldots, x_{i-1}+x_i <m$ if we were to have $j \leq i$ then $x_0+\cdots+x_j<jm$ which contradicts an assumption so $j>i$.   Now for $k \in [2,i]$, since   $(x_0-m)+\cdots+(x_j-m)>-m$  while $x_2,\ldots,x_{k-1} <m$ we have $(x_k-m)+\cdots+(x_j-m) >m-x_0-x_1$ implying that $\epsilon_{kj}^+(x)\geq m-x_0-x_1$. So far we have that $degr^+(x) \geq (\ell-i-1)m+(i-1)(m-x_0-x_1)$.  Furthermore, for $k \in [2,i]$, we have $\delta_{1k}^-(x)=x_0$ and  $\delta_{2k}^-(x)=x_1$ so $degr^-(x) \geq (x_0+x_1)(i-1)$.  Thus $degr(x) \geq (\ell-2)m$.

On the other hand, if $x_0+x_1 \geq m$ we still have $degr^+(x) \geq (\ell-i-1)m$ and since $x_1+x_2<m, \ldots, x_{i-1}+x_i <m$ we have $x_1+\cdots+x_{i}<(i-1)m$ so the previous claim can be applied to $x'=(x_0,\ldots,x_i,-)$ and we see that $degr^-(x) \geq degr^-(x') \geq (i-1)m$ so again $degr(x) \geq (\ell-2)m$.

\end{proof}

\section{Point, Pair, and the Cycle Maps}

If $x=(x_0,x_1,\ldots,x_{\ell}) \in \mathbf{D}_{\ell,m}$ in step coordinates  then $x$ may be represented in \emph{position coordinates} as $x=[a_0,a_1,\ldots,a_{\ell}]$ where $a_i=(x_i-m)+\cdots +(x_{\ell}-m)$ or equivalently $a_0=0$ and $a_i=(m-x_0)+\cdots (m-x_{i-1})$ for $i>0$. On the other hand, if $x=[a_0,\ldots,a_{\ell}]$ is written in position coordinates then  $x=(x_0,x_1,\ldots,x_{\ell-1},-)$ where $x_i=m+a_{i}-a_{i+1}$.  In terms of the physical path, $a_i$ represents how far the path is below the diagonal before the step represented by $x_i$ is taken.  The following is immediate from this definition.

\begin{lemma}
If $x=[a_0,\ldots,a_{\ell}]$ in position coordinates then 
\begin{eqnarray*}
area(x)=\sum_{i=1}^n a_i
\end{eqnarray*}
\end{lemma}

\begin{definition} \label{alpha}
Let  $\alpha(a,b)=min(a-b-1,m)$ if $a>b$ and $\alpha(a,b)=min(b-a,m)$ if $a\leq b$. Define $\alpha^0(a)=max(0,a-m)$.
\end{definition}
Note that if $x=[a_0,a_1,\ldots,a_{\ell}]$ we have  $\epsilon_{ij}^+(x)+\epsilon_{ij}^-(x)= \alpha(a_i,a_{j+1})$ and  $\epsilon_{0j}(x)=\alpha^0(a_{j+1})$  so that 
\begin{eqnarray*}
degr(x)=\sum_{1 \leq i<j\leq \ell}\alpha(a_i,a_j) + \sum_{2 \leq j \leq \ell} \alpha^0(a_j)
\end{eqnarray*}

\begin{example}\label{comdeg}
Recall that in example \ref{comdeg} we had that $\ell=5$ and that $m=5$ and that $x=(3,0,12,1,2,12)$. We can write $x$ in position coordinates as $x=[0,2,7,0,4,7]$ Therefore we have

\begin{eqnarray*}
degr(x)=\alpha(2,7)+\alpha(2,0)+\alpha(2,4)+\alpha(2,7)+\\
\alpha(7,0)+\alpha(7,4)+\alpha(7,7)+\\
\alpha(0,4)+\alpha(0,7)+\\
\alpha(4,7)-\\
\alpha_0(7)-\alpha_0(0)-\alpha_0(4)-\alpha_0(7)=\\
5+1+2+5+5+2+0+4+5+3-2-0-0-2=28
\end{eqnarray*}
as before.
\end{example}

We translate some earlier results into position coordinates:

\begin{corollary} \label{corfitsout}
Suppose that $x=[0,a_2,\ldots,a_{\ell}]$ is an $(\ell-1,m)$-Dyck path and choose some $a_1 \in [0,m]$, then  then $x'=[0,a_1,a_2+a_1,\ldots,a_{\ell}+a_1]$ is an $(\ell,m)$-Dyck path and  $degr(x') \leq (m-a_1)(\ell-1)+degr(x)$.  
\end{corollary}
\begin{proof}
Translate claim \ref{fitsout} from step coordinates to position coordinates.
\end{proof}

\begin{corollary} \label{corfitsin}
Suppose that $x=[0,a_1,\ldots,a_{\ell}]$ is an $(\ell,m)$-Dyck path and that for all $2 \leq j \leq \ell $ we have $a_j \geq m $ then $x'=[0,a_2-a_1,\ldots,a_{\ell}-a_1]$ is an $(\ell-1,m)$-Dyck path and  $degr(x)=(m-a_1)(\ell-1)+degr(x')$.  
\end{corollary}
\begin{proof}
Translate claim \ref{fitsin} from step coordinates to position coordinates.
\end{proof}

\begin{corollary} \label{bound3}
Let $x=[a_0,a_1,\ldots,a_{\ell}]$.  Suppose that $a_2 \leq m$ or for some $j \in (2, \ell]$ we have $a_j<m$. Let $i>0$ be minimal such that $a_i-a_{i+2} \geq -m$ and suppose $a_{i+1}-a_{\ell}>m+1$. Then $degr(x) \geq (\ell-2)m$.  
\end{corollary}
\begin{proof}
Translate claim \ref{bound2} from step coordinates to position coordinates.
\end{proof}

Now we can define the cycle maps:

\begin{definition}
If $x=[a_0,\ldots,a_{\ell}]$ in position coordinates then we define 
\begin{eqnarray*}
\mathtt{cycleright}(i,x)=[a_0,\ldots,a_i,a_{\ell}+1,a_{i+1},a_{i+2},\ldots,a_{\ell-2},a_{\ell-1}]
\end{eqnarray*}
\end{definition}

\begin{definition}
If $x=[a_0,\ldots,a_{\ell}]$ in position coordinates then we define 
\begin{eqnarray*}
\mathtt{cycleleft}(i,x)=[a_0,\ldots,a_i,a_{i+2
},a_{i+3},\ldots a_{\ell-1},a_{\ell}, a_{i+1}-1]
\end{eqnarray*}
\end{definition}

The following are easy to check:

\begin{itemize}
\item $\mathtt{cycleright}(i,\mathtt{cycleleft}(i,x))=x$
\item $\mathtt{cycleleft}(i,\mathtt{cycleright}(i,x))=x$
\item $area(\mathtt{cycleright}(i,x))=area(x)+1$
\end{itemize}

\begin{definition}
Let $x=[a_0,\ldots,a_{\ell}]$ we define $point(x)$ to be the minimal $r \leq \ell$ such that $a_r-a_{\ell}> -m$.
\end{definition}
\begin{definition}
Let $x=(x_0,\ldots,x_{\ell})$ we define $pair(x)$ to be the minimal $r < \ell$ such that $a_r-a_{r+2} \geq -m$ (where the convention is $a_s=0$ for $s>\ell$).  
\end{definition}

\begin{definition}
 Let $x=[a_0,\ldots,a_{\ell})$ and $point(x)=r$.  If $r<\ell$ and for all $0 \leq i \leq r-2$ we have $a_i-a_{i+2}<-m$  we say $x$ is \textbf{rightable} (at $r$) and define $\mathtt{right}(x)=\mathtt{cycleright}(r,x)$.
\end{definition}

\begin{definition}
Let $x=[a_0,\ldots,a_{\ell}]$ and $pair(x)=r$. If $a_{r+1}-a_{\ell} \leq m+1$ we say $x$ is \textbf{leftable} (at $r$) and define $\mathtt{left}(x)=\mathtt{cycleleft}(r,x)$.
\end{definition}

\begin{remark}
When $m=1$, the maps $\mathtt{right}$ and $\mathtt{left}$ essentially reduce to the maps $\mathtt{\nu^{-1}}$ and $\mathtt{\nu}$ of \cite{LLL18} and \cite{HLLL20}.
\end{remark}

\begin{example}
Let $m=5$ and let $x^0=[0,4,9,12,6,6,6,7,12,8]$. For $i<0$ define $x^{i}=\mathtt{left}(x^{i+1})$ if $x^{i+1}$ is leftable and for $i>0$ define $x^i=\mathtt{right}(x^{i-1})$ if $x^{i-1}$ is rightable. All such $x^{i}$ are shown below and to its right the value of the tuple $(pair(x^i),point(x^i))$.  Note that $x^{-5}$ is unleftable and $x^6$ is unrightable.  
\setcounter{MaxMatrixCols}{23}
\begin{eqnarray*}
\begin{matrix}
x^{-5}&0&4&7&12&8&11&8&5&5&5&(2,1)\\
x^{-4}&0&4&6&7&12&8&11&8&5&5&(1,1)\\
x^{-3}&0&4&6&6&7&12&8&11&8&5&(1,1)\\
x^{-2}&0&4&6&6&6&7&12&8&11&8&(1,1)\\
x^{-1}&0&4&9&6&6&6&7&12&8&11&(1,2)\\
\textcolor{red}{x^{0}}&\textcolor{red}{0}&\textcolor{red}{4}&\textcolor{red}{9}&\textcolor{red}{12}&\textcolor{red}{6}&\textcolor{red}{6}&\textcolor{red}{6}&\textcolor{red}{7}&\textcolor{red}{12}&\textcolor{red}{8}&\textcolor{red}{(2,1)}\\
x^{1}&0&4&9&9&12&6&6&6&7&12&(1,2)\\
x^{2}&0&4&9&13&9&12&6&6&6&7&(2,1)\\
x^{3}&0&4&8&9&13&9&12&6&6&6&(1,1)\\
x^{4}&0&4&7&8&9&13&9&12&6&6&(1,1)\\
x^{5}&0&4&7&7&8&9&13&9&12&6&(1,1)\\
x^{6}&0&4&7&7&7&8&9&13&9&12&(1,6)\\
\end{matrix}
\end{eqnarray*}
\end{example}

\begin{claim}\label{rightcheck}
Suppose $y\in \mathbf{D}_{\ell,m}$ is rightable at $r$ and $x=\mathtt{right}(y)$.  Then we have:
\begin{enumerate}
\item $x \in \mathbf{D}_{\ell,m} \setminus \mathbf{T}_{\ell,m}$
\item $x$ is leftable at $r$ and $\mathtt{left}(x)=y$.
\item $degr(x)=degr(y)$
\item $area(x)=area(y)+1$
\end{enumerate}
\end{claim}

\begin{proof}  We write $y=[c_0,c_1,\ldots,c_{\ell}]$ and $x=[a_0,a_1,\ldots,a_{\ell}]$

\begin{enumerate}

\item 
We must first check that all $a_{i}-a_{i+1}\geq -m$ for $i \in [0,\ell)$ (otherwise $x$ would contain negative step coordinates).  If  $i \notin \{r,r+1\}$ this follows from the fact that $c_{j}-c_{j+1}\geq m$ for all $j \in [0,\ell)$.  Next if $i=r$ then $a_{i}-a_{i+1}=c_{r}-c_{\ell}-1\geq -m$ by the assumption that $point(y)=r$.  Finally we must check $a_i-a_{i+1} \geq -m$ for $i=r+1$.  Since we are assuming $i<\ell$ we have $r< \ell-1$.  If $r>0$, then by minimality of $r$ we have $c_{\ell}-c_{r-1} \geq m$ which in conjunction with $c_{r-1}-c_{r} \geq -m $ and $c_{r}-c_{r+1} \geq -m$ gives $c_{\ell}-c_{r+1} \geq -m$ so $a_{r+1}-a_{r+2}=c_{\ell}+1-c_{r+1}> -m$.  If $r=0$ then since $c_0-c_1=0-c_1 \geq -m$ we have $a_{r+1}-a_{r+2}=a_1-a_2=c_{\ell}+1-c_{1} > -m$. 

Next we must check that $x \notin \mathbf{T}_{\ell,m}$ or equivalently that $a_1 \neq 0$.  If $point(y)=0$ then $a_1=c_{\ell}+1>0$.  If $point(y)\geq 2$ and $a_1=0$ then $c_1=0=c_0$ so, since $c_1-c_2 \geq -m$ also  $c_0-c_2 \geq -m$ which would imply $y$ is not rightable. If $point(y)= 1$ and $a_1=0$ then $c_1=0=c_0$ and so since $c_1-c_{\ell} >-m$ we also have $c_0-c_{\ell} >-m$ which contradicts $point(y)=1$.  

\item
We have $a_i-a_{i+2}=c_i+c_{i+2} <-m$ for $0 \leq i \leq r-2$ by the fact that $y$ is rightable at $r$.  Further $a_{r-1}-a_{r+1}=c_{r-1}-c_{\ell}-1<c_{r-1}-c_{\ell} \leq -m$ by minimality of $r$.   If $r=\ell-1$ we have by convention that $a_r-a_{r+2} =a_r \geq 0 \geq -m$.  If $r < \ell-1$ then $a_r-a_{r+2}=c_r-c_{r+1} \geq-m$.  Thus $pair(x)=r$. Further $a_{r+1}=c_{\ell}+1>0$.  Finally if $r=\ell-1$ then  $a_{r+1}-a_{\ell}=0 \leq m+1$. Otherwise $a_{r+1}-a_{\ell}=c_{\ell}+1-c_{\ell-1} \leq m+1$.  Therefore we see $x$ is leftable at $r$ and so $\mathtt{left}(x)=y$.

\item

Using definition \ref{alpha} and applying claim \ref{split} we see that:

\begin{eqnarray*}
degr(y)-degr(x) =\\
\sum_{1 \leq i<j\leq \ell}\alpha(c_i,c_j) - \sum_{2 \leq j \leq \ell} \alpha^0(c_j)
- \sum_{1 \leq i<j\leq \ell}\alpha(a_i,a_j) + \sum_{2 \leq j \leq \ell} \alpha^0(a_j)=\\
\sum_{1 \leq j \leq r} \alpha(c_j,c_\ell)- \alpha(c_j,c_{\ell}+1)+\sum_{r < j < \ell} \alpha(c_j,c_{\ell}) - \alpha(c_{\ell}+1,c_j) \\
-\alpha^0(c_{\ell})+\alpha^0(c_{\ell}+1)
\end{eqnarray*}

First suppose $r>0$.  For $j<r$ the fact that $y$ is rightable at $r$ implies  that $c_j-c_{\ell} \leq -m$ so $\alpha(c_j,c_\ell)=m= \alpha(c_j,c_{\ell}+1)$. Now $c_r-c_{\ell} >-m$ and  $c_r-c_{\ell} \leq 0$ since  $c_{r}-c_{r-1} \leq m$ and $c_{r-1}-c_{\ell} \leq -m$  so $\alpha(c_r,c_\ell)- \alpha(c_r,c_{\ell}+1)=-1$.  For $j>r$ we have $\alpha(c_j,c_{\ell}) = \alpha(c_{\ell}+1,c_j)$ by definition of $\alpha$.  Finally, we have  $\alpha^0(c_{\ell}+1)-\alpha^0(c_{\ell})=1$ since $r>0$ implies $c_0-c_{\ell} \leq -m$. i.e., $c_{\ell} \geq m$.

If $r=0$ then first sum is empty and the second sum completely cancels and we are left with  $\alpha^0(c_{\ell}+1)-\alpha^0(c_{\ell})$, but both terms in this difference are $0$ since $r=0$ implies that $c_0-c_{\ell} >-m$ so $c_{\ell}<m$.  Therefore, in any case we have $degr(y)=degr(x)$.  
\item Finally, it is trivial to check $area(x)=area(y)+1$.
\end{enumerate}
\end{proof}

\begin{claim}\label{leftcheck}
Suppose $x\in \mathbf{D}_{\ell,m} \setminus \mathbf{T}_{\ell,m}$ is leftable at $r$ and $y=\mathtt{left}(x)$.  Then we have:
\begin{enumerate}
\item $y \in \mathbf{D}_{\ell,m}$
\item $y$ is rightable at $r$ and $\mathtt{right}(y)=x$.
\item $degr(y)=degr(x)$
\item $area(y)=area(x)-1$
\end{enumerate}
\end{claim}

\begin{proof}  We write $x=[a_0,a_1,\ldots,a_{\ell}]$and $y=[c_0,c_1,\ldots,c_{\ell}]$.

\begin{enumerate}

\item 
We must first check that $c_{i}-c_{i+1}\geq -m$ for $i \in [0,\ell)$.  If  $i \notin \{r,\ell-1\}$ this follows from the fact that $a_{j}-a_{j+1}\geq m$ for all $j \in [0,\ell)$. Further if $r=\ell-1$ we need only still check that $c_{\ell-1}-c_{\ell} \geq -m$ but we have $c_{\ell-1}-c_{\ell}= a_{\ell-1}-a_{\ell} +1 > -m$ so we may suppose $r < \ell-1$.  Now $c_{r}-c_{r+1}=a_{r}-a_{r+2} \geq -m$ by the assumption that $pair(x)=r$.  Next, $c_{\ell-1}-c_{\ell}=a_{\ell}-a_{r+1}+1 \geq -m$ since $a_{r+1}-a_{\ell} \leq m+1$ by the fact $x$ is leftable at $r$.  

We must also check that $c_{\ell}=a_{r+1}-1 \geq 0$.  Suppose that $a_{r+1}=0$. If $r=0$ this implies $x \in \mathbf{T}_{\ell,m}$ which is contrary to assumption and if $r>0$ it implies $a_{r-1}-a_{r+1} \geq -m$ contradicting $pair(x)=r$.

\item
We have $c_r-c_{\ell}=a_r-a_{r+1}+1 >-m$.  If for some $i<r$ we have $c_i-c_{\ell}>-m$ that means that $a_i-a_{r+1}\geq -m$ or that $(a_i-a_{i+1}+m)+\cdots+(a_r-a_{r+1}+m) \geq m(r-i)$ which implies there must be some $i' \in [i,r-1]$ with $(a_{i'}-a_{i'+1}+m)+(a_{i+1}'-a_{i'+2}+m) \geq m$, which is to say $a_{i'}-a_{i'+2} \geq -m$, contradicting the fact that $pair(x)=r$.  It follows that $point(y)=r$.  Further for all $i \leq r-2$ we have $c_i-c_{i+2}=a_i-a_{i+2} < -m$ by minimality of $r$. Therefore $y$ is rightable at $r$ and $\mathtt{right}(y)=x$.

\item

This follows from the above and claim \ref{rightcheck}

\item Finally, it is trivial to check $area(y)=area(x)-1$.
\end{enumerate}
\end{proof}

\section{Bounded Partitions and $0-1$ Matrices}

In this section we define two sets that are in bijective correspondence with maximal $(\ell,m)$-Dyck paths of degree less than $(\ell-1)m$.  These constructions will be essential to the next section and be used throughout the rest of the paper.  

\begin{definition}
Let $\mathbf{P}_{\ell-1}$ be the set of all partitions $\lambda=(\lambda_1,\ldots,\lambda_r)$ with longest part at most $\ell-1$ that is $\lambda_1 \leq \ell-1$. 
\end{definition}

\begin{definition}
 For $\lambda \in \mathbf{P}_{\ell-1}$ define the height of $\lambda$, or $h(\lambda)$, to be the number of rows of the partition.  Define the $(\ell-1)$-width of $\lambda$ or $w^{\ell-1}(\lambda)$ to be the sum of the length of all of the $(\ell-1)$-hook rows of $\lambda$.  The $(\ell-1)$-hook rows (or just hook rows, when $\ell$ is clear from context) of $\lambda=(\lambda_1,\ldots,\lambda_r)$ are determined as follows: The first row of the partition is a hook row, and, if row $i$  is a hook row so is row $i+\ell-\lambda_i$. Finally the size of $\lambda$ or $|\lambda|$ is the sum of the lengths of all rows.  
\end{definition}

\begin{definition}
Define $\mathbf{P}_{\ell-1}^{<d}$ to be the subset of $\mathbf{P}_{\ell-1}$ of size less than $d$.
\end{definition}

\begin{definition}
Given a matrix, $M$, we define $\mathtt{rowRead}(M)$ to be the list attained by concatenating the rows of $M$ (read left to right) beginning with the top row and moving down.  Similarly,we define $\mathtt{colRead}(M)$ to be the list attained by concatenating the columns of $M$ (read top to bottom) beginning with the leftmost column and moving right.  
\end{definition}

\begin{definition}\label{matdef}
We define $\mathbf{M}_{m \times \ell}$ to be the set of $m \times \ell$ matrices $M$, with the following properties:

\begin{enumerate}

\item The entries of $M$ belong to the set $\{0,1,*\}$.

\item Any $*$ appearing in $\mathtt{rowRead}(M)$ lies to the right of any $0$ or $1$ appearing in $\mathtt{rowRead}(M)$.

\item The first entry of $\mathtt{rowRead}(M)$ is $0$.

\item The rightmost entry of $\mathtt{rowRead}(M)$ which is not a $*$ is a $1$.  

\item There is no appearance of a $0$ that lies directly above a $1$ in $M$.

\end{enumerate} 
\end{definition}

\begin{definition}
We say that a $0$ and a $1$ are in contact in $M$ if either they appear in the same row with the $0$ to the left of the $1$ or else if the $1$ appears in the row below the $0$ and weakly to its left (equivalently, strictly to its left since no $0$ may lie above a $1$).  We define $|M|$ to be the total number of $(0,1)$ pairs that are in contact.
\end{definition}

\begin{definition}
We define $h(M)$ to be the total number of $0$s in $M$ and $w(M)$ to be the total number of $1$s in $M$.
\end{definition}

\begin{example}
We have $M=\begin{bmatrix}
0 & 1 & 1 & 1  & 1 & 1 & 1 & 0 & 1\\
0 & 1 & 1 & 1 & 0 & 1 & 1 & 0 & 1\\
0 & 0 & 1 & 1  & 0 & 1 & 0 & 0 & 1\\
0 & 0 & 1 & 0 & 0 & 1 & * & * & *\\
* & * & * & *  & * & * & * & * & *\\
\end{bmatrix} \in \mathbf{M}_{5 \times 9}$ and we have $h(M)=14$ and $w(M)=19$ and $|M|=51$.
\end{example}

\begin{definition}
We say  $M \in \mathbf{M}_{m \times \ell}^{<d}$ if $M \in \mathbf{M}_{m \times \ell}$ and $|M|<d$. 
\end{definition}

\begin{claim}\label{g}
There is a bijection $g: \mathbf{P}_{\ell-1}^{<(\ell-1)m} \rightarrow \mathbf{M}_{m \times \ell}^{<(\ell-1)m}$ such that if $g(\lambda)=M$ then:
\begin{itemize}

\item $h(\lambda)=h(M)$.

\item $w^{\ell-1}(\lambda)=w(M)$.

\item $|\lambda|=|M|$

\end{itemize}

\end{claim}

\begin{proof}

Let $\lambda \in \mathbf{P}_{\ell-1}^{<(\ell-1)m}$.  Consider the Young diagram associated to the partition $\lambda$.  Beginning in the top right corner and moving to the bottom left corner trace the (lower right) boundary of $\lambda$.  As you do this, create a sequence composed of the symbols $\{0,1,|\}$ by starting with a $|$ and then recording a $0$ for each downward step and a $1$ for each leftward step until you arrive at the (top right corner of the) next $(\ell-1)$-hook row.  A this point record a $|$ and continue the process.  The process terminates when you reach the bottom left corner of $\lambda$.  You now have a sequence of $0$s and $1$s divided up into $n$ parts, $s_1,\ldots,s_n$ by the $|$ symbols where $n$ is the number of hook rows of $\lambda$.   We will denote the number of $0$s in $s_k$ by $z_k$ and the number of $1$s in $s_k$ by $u_k$. We denote the indices of the hook rows of $\lambda$ by $h_1,\ldots,h_n$.

\begin{example}
If $\ell=9$ and $\lambda=(8,7,5,5,4,3,3,2,1,1)$ then $\lambda \in \mathbf{P}_{\ell-1}$. The hook rows are rows $h_1=1$, $h_2=2$, $h_3=4$, and $h_4=8$ and are each marked by a $\bullet$ below.  We have $h(\lambda)=10$ and $w^{\ell-1}(\lambda)=22$ and $|\lambda|=39$.   \\
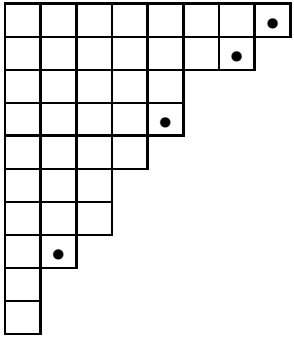
\begin{eqnarray*}
\begin{Young}
&&&&&&&${\bullet}$\cr
&&&&&&${\bullet}$\cr
&&&&\cr
&&&&${\bullet}$\cr
&&&\cr
&&\cr
&&\cr
&$\bullet$\cr
\cr
\cr
\end{Young}
\end{eqnarray*}
Tracing the boundary of $\lambda$ gives the sequence $|01|0110|0101001|01001$. We have $z_1=1$, $z_2=2$, $z_3=4$, and $z_4=3$ while $u_1=1$, $u_2=2$, $u_3=3$, and $u_4=2$. 
\end{example}

Now, start with an empty $m \times \ell$ matrix. We now perform $n$ steps.  In step $1$  place the sequence $s_n$ in row $n$ of $M$ flushed to the left.  (This is possible since $s_n$ has at most $\ell$ terms by construction and since $n \leq m$.  Indeed, if $n>m$ then the first $m$ hook rows of $\lambda$ contribute at least $\ell-1$ to $|\lambda|$ contradicting $|\lambda|<(\ell-1)m$.)  Then, for each $j$ such that $M_{n,j}$ contains a $1$, place a $1$ in $M_{i,j}$ for each $i<n$.   Now suppose that you have completed steps $1,2,\ldots,n-k$.  To perform step $n-k+1$ consider the empty positions in row $k$ and fill these (left to right) with the terms found in the sequence $s_{k}$ until there are no more empty positions or until you run out of  terms. Then,  for each $j$ such that $M_{k,j}$ contains a $1$, place a $1$ in $M_{i,j}$   (if there is not already one) for each $i<n$.  When all $n$ steps are complete, fill the remaining positions with $*$s.

\begin{example} Again let us take $\lambda=(8,7,5,5,4,3,3,2,1,1)$ which gave us the sequence  $|01|0110|0101001|01001$.  Following the steps outlined above creates the following matrix. \\
\small{
\begin{eqnarray*}
\begin{bmatrix}
 &  &  &   &  &  &  &  & \\
 &  &  &  &  & &  &  & \\
 &  &  &  &  &  &  &  & \\
0&1&0&0&1&\phantom{1}  &  &  & \\
 \phantom{1}&   \phantom{1} & \phantom{1}   & \phantom{1}   & \phantom{1}  &  \phantom{1} &  \phantom{1} &  \phantom{1} &  \phantom{1}\\
\end{bmatrix} \longrightarrow
\begin{bmatrix}
 & 1 &  &  & 1 &  &  &  & \\
 & 1 &  & & 1 & &  &  & \\
 & 1 &  & & 1  &  &  &  & \\
0& 1 & 0 & 0  & 1 &  &  &  & \\
 \phantom{1}&   \phantom{1} & \phantom{1}   & \phantom{1}   & \phantom{1}  &  \phantom{1} &  \phantom{1} &  \phantom{1} &  \phantom{1}\\
\end{bmatrix} \longrightarrow
\end{eqnarray*}
\begin{eqnarray*}
\begin{bmatrix}
 & 1 &  &  & 1 &  &  &  & \\
 & 1 &  & & 1 & &  &  & \\
 0 & 1 & 1  & 0 & 1  & 1 & 0 & 0 & 1\\
0& 1 & 0 & 0  & 1 &  &  &  & \\
 \phantom{1}&   \phantom{1} & \phantom{1}   & \phantom{1}   & \phantom{1}  &  \phantom{1} &  \phantom{1} &  \phantom{1} &  \phantom{1}\\
\end{bmatrix}\longrightarrow
\begin{bmatrix}
 & 1 &  1 &  & 1 &  1 &  &  & 1 \\
 & 1 & 1 & & 1 & 1&  &  & 1 \\
 0 & 1 & 1  & 0 & 1  & 1 & 0 & 0 & 1\\
0& 1 & 0 & 0  & 1 &  &  &  & \\
 \phantom{1}&   \phantom{1} & \phantom{1}   & \phantom{1}   & \phantom{1}  &  \phantom{1} &  \phantom{1} &  \phantom{1} &  \phantom{1}\\
\end{bmatrix}\longrightarrow
\end{eqnarray*}
\begin{eqnarray*}
\begin{bmatrix}
 & 1 &  1 &  & 1 &  1 &  &  & 1 \\
0 & 1 & 1 & 1 & 1 & 1& 1 & 0 & 1 \\
 0 & 1 & 1  & 0 & 1  & 1 & 0 & 0 & 1\\
0& 1 & 0 & 0  & 1 &  &  &  & \\
 \phantom{1}&   \phantom{1} & \phantom{1}   & \phantom{1}   & \phantom{1}  &  \phantom{1} &  \phantom{1} &  \phantom{1} &  \phantom{1}\\
\end{bmatrix}\longrightarrow
\begin{bmatrix}
 & 1 &  1 & 1 & 1 &  1 & 1 &  & 1 \\
0 & 1 & 1 & 1 & 1 & 1& 1 & 0 & 1 \\
 0 & 1 & 1  & 0 & 1  & 1 & 0 & 0 & 1\\
0& 1 & 0 & 0  & 1 &  &  &  & \\
 \phantom{1}&   \phantom{1} & \phantom{1}   & \phantom{1}   & \phantom{1}  &  \phantom{1} &  \phantom{1} &  \phantom{1} &  \phantom{1}\\
\end{bmatrix}\longrightarrow
\end{eqnarray*}
\begin{eqnarray*}
\begin{bmatrix}
 0 & 1 &  1 & 1 & 1 &  1 & 1 & 1 & 1 \\
0 & 1 & 1 & 1 & 1 & 1& 1 & 0 & 1 \\
 0 & 1 & 1  & 0 & 1  & 1 & 0 & 0 & 1\\
0& 1 & 0 & 0  & 1 &  &  &  & \\
 \phantom{1}&   \phantom{1} & \phantom{1}   & \phantom{1}   & \phantom{1}  &  \phantom{1} &  \phantom{1} &  \phantom{1} &  \phantom{1}\\ 
\end{bmatrix}\longrightarrow
\begin{bmatrix}
 0 & 1 &  1 & 1 & 1 &  1 & 1 & 1 & 1 \\
0 & 1 & 1 & 1 & 1 & 1& 1 & 0 & 1 \\
 0 & 1 & 1  & 0 & 1  & 1 & 0 & 0 & 1\\
0& 1 & 0 & 0  & 1 & * & * & * &* \\
*&*&*&*&*&*&*&*&*\\
\end{bmatrix}\phantom{\longrightarrow}
\end{eqnarray*}
}

\end{example}

It is clear that in this process a $0$ is never placed above a $1$.  Suppose now steps $1,\ldots,n-k+1$ have been performed and the following additional properties are true for each $k' \geq k$.

\begin{enumerate}

\item There are no empty positions in row $k'$ and ignoring each $1$ in row $k'$ which lies above another $1$, the entries row $k'$ read left to right gives the sequence $s_{k'}$.

\item For each $i \in [1,z_{k'}]$ the $i^{th}$ $0$ of row $k'$ of $M$ is in contact with $\lambda_{h_{k'}+i-1}$ entries equal to $1$.   

\item The leftmost entry in row $k'$ is $0$.

\end{enumerate}

We now show that if these properties hold for all $k' \geq k$ after step $n-k+1$ then they hold for all $k' \geq k-1$ after step $n-k+2$.

\begin{enumerate}

\item  If $k=n$ then the number of empty positions in row $k-1$ of $M$ after step $1$ is equal to $\ell-u_k$.  Moreover, we have $\lambda_{h_k}+z_{k-1}+u_{k-1}=\ell$ and that $u_k=\lambda_{h_k}$ so the number of terms in $s_{k-1}$ is equal to $z_{k-1}+u_{k-1}=\ell-u_k$ as well.  

If $k<n$ then the number of empty positions in row $k-1$ of $M$ after step $n-k+1$ is equal to $z_k$ by the fact that (1) holds for $k'=k$ after step $n-k+1$.  Moreover, by the fact row $h_k$ of $\lambda$ is a hook row  we have $z_k+h_k=\ell$ and  by the fact row $h_{k-1}$ of $\lambda$ is a hook row  we have $z_{k-1}+u_{k-1}+h_k=\ell$ and so the number of terms in $s_{k-1}$ is also equal to $z_k$.  

Therefore, after step $n-k+1$ the number of empty positions in row $k-1$ of $M$ is equal to the length of $s_{k-1}$ which implies that after step $n-k+2$ is completed, (1) holds for all $k' \geq k-1$.

\item Now suppose we have completed step $n-k+2$.  Consider the $i^{th}$ $0$ in row $k-1$.  This $0$ is in contact with three types of $1$s. First, it is in contact with $1$s that appear in row $k-1$ and do not have another $1$ below them.   This number is equal to the number of $1$s following the $i^{th}$ $0$ in $s_{k-1}$ which by considering the Young diagram of $\lambda$ can be seen to be  $\lambda_{h_{k-1}+(i-1)}-\lambda_{h_k}$.  The second type of $1$s are those that appear in row $k-1$ and have a $1$ beneath them.  The third type of $1$s are those that appear in row $k$.   The number of $1$s of the second and third type that our $0$ is in contact with is equal to the total number of $1$s in row $k$.  Since (3) holds for $k'=k$ after step $n-k+1$ there is a $0$ in the first position of row $k$ which is in contact with all $1$s in that row and since (2) holds for $k'=k$ after step $n-k+1$ there are $\lambda_{h_k}$ such $1$s.   Thus it  follows that the $i^{th}$ $0$ of row $k-1$ is in contact with $\lambda_{h_{k-1}+(i-1)}$ $1$s.  This proves that after step $n-k+2$ is completed, (2) is true for all $k'\geq k-1$.

\item Finally, since $s_{k-1}$ must begin with a $0$ and row $k-1$ and column $1$ must be empty before step $n-k+2$ by the fact (3) holds for $k'=k$ after step $n-k+1$, we see that this position contains a $0$ after step $n-k+2$.  Thus after step $n-k+2$ is completed, (3) is true for all $k'\geq k-1$.

\end{enumerate}

We define the result of this process to be $g(\lambda)$.  The discussion above makes it clear that $g(\lambda) \in \mathbf{M}_{m \times \ell}$. Moreover, number of $0$s appearing in $g(\lambda)$ is clearly equal to the number of rows of $\lambda$ so we have $h(g(\lambda))=h(\lambda)$.  Further, $w^{\ell-1}(\lambda)=\lambda_{h_1}+\cdots+\lambda_{h_n}$.  But the discussion above implies that for each $i \in [1,n]$ the $0$ in the leftmost column of row $i$ and  is in contact with $\lambda_{h_i}$ many $1$s, so this is exactly the number of $1$s in row $i$.  It follows that $w(g(\lambda))=w^{\ell-1}(\lambda)$.    Finally, since for each $i \in [1,h(\lambda)]$, the $i^{th}$ $0$ in $g(\lambda)$ is in contact with $\lambda_i$ many $1$s, the total number of $(0,1)$ pairs that are in contact, or $|g(\lambda)|$ is equal to the sum of the lengths of the rows of $\lambda$ or $|\lambda|$.  This last statement implies that in fact $g(\lambda) \in \mathbf{M}_{m \times \ell}^{<(\ell-1)m}$

Now suppose that $M \in \mathbf{M}_{m \times \ell }^{<(\ell-1)m}$.  Let $\prec$ denote row reading order (i.e., the order used to compute $\mathtt{rowread}(M)$) on the positions of $M$.   Define $\lambda=(\lambda_1,\ldots,\lambda_{h(M)})$ where the $i^{th}$  (under the order of $\prec$) $0$ of $M$ is in contact with precisely $\lambda_i$ $1$s. Suppose that there is some $i$ such that $\lambda_i<\lambda_{i+1}$.  Then there is some $j$ such that the $i+1^{st}$ $0$ of $M$ is in contact with a $1$ in column $j$ but $i^{th}$ $0$ of $M$ is not in contact with a $1$ in column $j$.  But this is only possible if a $0$ appears above a $1$ in column $j$ which is a contradiction.  Therefore $\lambda$ is a partition.  Since any $0$ can be in contact with at most $\ell-1$ $1$s we have $\lambda \in \mathbf{P}_{\ell-1}$ and so by construction in $\lambda \in \mathbf{P}_{\ell-1}^{<(\ell-1)m}$.

  Property (2) shown earlier implies that if $g(\mu)=M$ then $\mu=\lambda$ so $g$ is injective.   To show that $g$ is surjective it suffices to show that $g(\lambda)=M$.
 Suppose, rather, that $g(\lambda)=N \neq M$.  Let $(i,j)$ be maximal position in $M$ under $\prec$ such that $M_{i,j} \neq N_{i,j}$.  Suppose that $M_{i,j}=0$ and $N_{i,j}=1$.  Let $(i',j')$ be maximal such that $(i',j')\prec (i,j)$ and $N_{i',j'}=0$.  For any $k$ such that the $0$ in $M_{i,j}$ is in contact with some $1$ in column $k$ of $M$, the $0$ in $N_{i',j'}$ must be in contact with a $1$ from column $k$ of $N$.  In addition, the $0$ in $N_{i',j'}$ is also in contact with the $1$ in $N_{i,j}$.  Therefore (interchanging $M$ and $N$ in the previous argument if necessary), there is some $r$ such that the $r^{th}$ to last $0$ in $M$ is in contact with a different number of $1$s than the $r^{th}$ to last $0$ in $N$ is in contact with.  By definition the former number is $\lambda_{h(\lambda)-r+1}$ and by property (2) the latter number is also  $\lambda_{h(\lambda)-r+1}$.  Thus $M \neq N$ is impossible.

\end{proof}

The following result will be used later:

\begin{claim} \label{fitin}

If $x\in  \mathbf{T}_{\ell}^{<(\ell-1)m}$ and $x=[a_0,a_1,\ldots, a_{\ell}]$ then $a_j \leq m$ for all $j \in [0,\ell]$. 

\end{claim}

\begin{proof}

Suppose there is some $k$ such that $a_k>m$. Let $x$ be represented by an array $A$ composed of columns of varying height, all flushed to the top of the array.  $A$ has $\ell$ columns and column $j$ read from top to bottom is equal to $a_j$ $1$s followed by $m$ $0$s.  

\begin{example}
If $a =[0,0,2,7,3,5,4,6,6,3] \in \mathbf{T}_{9,5}$ then  
\begin{eqnarray*} 
A=\begin{bmatrix}
0 & 1 & 1 & 1  & 1 & 1 & 1 & 1 & 1\\
0 & 1 & 1 & 1 & 1 & 1 & 1 & 1 & 1\\
0 & 0 & 1 & 1  & 1 & 1 & 1 & 1 & 1\\
0 & 0 & 1 & 0 & 1 & 1 & 1 & 1 & 0\\
0 & 0 & 1 & 0  &  1 & 0 & 1 & 1 & 0 \\
  & 0 & 1 & 0  &  0 & 0 & 1 & 1 & 0 \\
  & 0 & 1 &  0  &  0 & 0 & 0 & 0 & 0\\
 &    & 0 &  0 &   0 & 0 &  0& 0 & 0 \\
  &    & 0 &    &   0 & 0 &  0& 0 & \\
 &     & 0 &    &   0 &   &  0& 0 & \\
  &    & 0 &    &    &   &  0& 0 & \\
 &    & 0 &    &    &   &  &  & \\
\end{bmatrix}
\end{eqnarray*}
and we could take $k$ to be $3$, $7$, or $8$.
\end{example}

Recall that using definition  \ref{alpha} we have 
\begin{eqnarray*}
degr(x)=\sum_{1 \leq i <  j \leq \ell} \alpha(a_i,a_j)- \sum_{2 \leq j \leq \ell} \alpha_0(a_j) 
\end{eqnarray*}
Note that if $a_i >a_j$ then $\alpha(a_i,a_j)$ returns the number of $0$s in column $j$ of $A$ in contact with a $1$ in column $i$ of $A$ (and there are no $0$s in column $i$ in contact with a $1$ in column $j$.)  On the other hand, if $a_i \leq a_j$ then $\alpha(a_i,a_j)$ returns the number of $0$s in column $i$ of $A$ in contact with a $1$ in column $j$ of $A$ (and there are no $0$s in column $j$ in contact with a $1$ in column $i$.).  Therefore, in either case $\alpha(a_i,a_j)$ counts the number of $(0,1)$ pairs in contact with one term appearing in column $i$ and the other in column $j$.  
Meanwhile the value of $\alpha^0(a_j)$ is the number of $1$s that appear in column $j$ below row $m$.  Therefore $degr(x)$ is equal to the total number of $(0,1)$ pairs in contact in $A$ minus the number $1$s that appear below row $m$ of $A$.  

Since $a_1=0$ the first column of $A$  contains $m$ $0$s and so for each $1$ in the first $m$ rows of $A$ there is a $0$ in the first column of $A$ that is in contact with it.  Moreover, every $0$ in the first $m$ rows of $A$ and not in column $1$ is in contact with a $1$ in column $k$.  So far we have counted $m (\ell-1)$ $(0,1)$ pairs in contact. Now for any $1$ that lies in some row $i>m$ and column $j$ (for instance row $6$ and column $7$ in the example), let $j'<j$ be maximal such that $A_{i,j'}$ does not contain a $1$ (continuing the example we would have $j'=5$).  Column $1$ contains no $1$ so such maximal $j'$ exists.  Since $a_{j'}-a_{j'+1} \geq -m$ and $A_{i,j'+1}$ contains a $1$, it follows that $A_{i,j'}$ contains a $0$ (that is in contact with the $1$ in $A_{i,j}$).  Since no $(0,1)$ pair with the $0$ below row $m$ has been counted yet, we have now counted $m (\ell-1)+u$  pairs in contact where $u$ is the number of $1$s below row $m$.  It follows that $degr(a) \geq   (m (\ell-1)+u)-u$ which contradicts the assumption thereby proving the claim.

\end{proof}

\begin{claim}\label{f}

There is a bijection $f: \mathbf{M}_{m \times \ell}^{<(\ell-1)m} \rightarrow \mathbf{T}_{\ell,m}^{<(\ell-1) m}$ such that if $f(M)=x$ then:

\begin{itemize}

\item $w(M)=area(x)$.

\item $|M|=degr(x)$

\end{itemize}

\end{claim}

\begin{proof}

Given $M \in \mathbf{M}_{m \times \ell}^{<(\ell-1)m}$ define $f(M)=[a_0,a_1,\ldots,a_{\ell}]$ where $a_0=0$ and for $j>0$, $a_j$ is equal to the number of $1$s in column $j$ of $M$.  Since the $M$ has only $m$ rows the maximum value of $a_j$ is $m$ so $a_j-a_{j+1} \geq -m$ for $0 \leq j <\ell$ and we see that  $f(M) \in \mathbf{D}_{\ell,m}$.  Since column $1$ of $M$ has no $1$s we also have that $a_1=0$ so that in fact $f(M) \in \mathbf{T}_{\ell,m}$.  Clearly $w(M)=area(f(M))$.  Moreover since all $a_j \leq m$ we have $\alpha_0(a_j)=0$ for all $j \in [2,\ell]$ and it follows (see the proof of claim \ref{fitin}) that $degr(f(M))$ is the number of pairs of $0$s and $1$ in contact in the array $A$ associated to $x$ (as in the proof of \ref{fitin}).  Since $A$ is idenctical to $M$ except that the former may contain some $0$s (but no $1$s) below row $m$, this number is the  same as the number of $(0,1)$ pairs in contact in $M$ or $|M|$. This implies that $degr(f(M))=|M|$ and therefore also that $f(M) \in \mathbf{T}_{\ell,m}^{<(\ell-1)m}$

Given $x=[a_0,a_1,\ldots,a_{\ell}] \in \mathbf{T}_{\ell,m}^{<(\ell-1)m}$ start with an empty $m \times \ell$ matrix and for each $j \in [1,\ell]$ place a $1$ in the top $a_j$ positions in column $j$. Note that this is possible since by claim \ref{fitin} $a_j \leq m$ for all $j$. Now suppose that $(i,j)$ is maximal (in the row reading order, $\prec$) such that position $(i,j)$ of the result contains a $1$.  For all pairs $(i',j') \prec (i,j)$ that do not contain a $1$, place a $0$ in position $(i',j')$.  Finally, place a $*$ in all remaining empty positions.  If $M$ is the resulting matrix we define $f^{-1}(x)=M$.  Since $a_1=0$ the top left entry of $M$ is $0$.  The rest of the properties of $\mathbf{M}_{m \times \ell}$ specified in definition \ref{matdef} follow easily from the construction.  Clearly we have $f(f^{-1}(x))=x$ so we see from the above that $|f^{-1}(x)|=degr(x)$ so that $f^{-1}(x) \in \mathbf{M}_{m \times \ell}^{<(\ell-1)m}$

Finally since also $f^{-1}(f(M))=M$ the claim is proved.
\end{proof}

\begin{claim}\label{i}

There is an involution $ \iota$ on  $\mathbf{M}
_{m\times \ell}$ such that if $\iota(M)=N$ then

\begin{itemize}

\item $w(M)=h(N)$

\item $h(M)=w(N)$

\item $|M|=|N|$

\end{itemize}

\end{claim}

\begin{proof}

  Given $M \in \mathbf{M}_{m \times \ell}$ start by changing all entries of $M$ that were $1$ to $0$ and all entries that were $0$ to $1$. Now, let $(\hat{i},\hat{j})$ be maximal in the row reading order $\prec$  such that $M_{\hat{i},\hat{j}} \neq *$. Consider the submatrix with top left corner in position $(1,1)$ and bottom right corner $(\hat{i},\hat{j})$.  Take this submatrix and rotate it through $180$ degrees.  Then consider the submatrix with top left corner in position $(1,\hat{j}+1)$ and bottom right corner in position $(\hat{i}-1,\ell)$.  Rotate this submatrix through $180$ degrees as well.

\begin{example} In the top left we start with the matrix $M$.  Then we interchange all $0$s and $1$s to get the matrix on the top right.  Then we rotate the first submatrix mentioned in the proof to get the matrix on the bottom left.  Finally, we rotate the second submatrix mentioned in the proof to get the matrix on the bottom right.\\
\begin{eqnarray*}
\begin{bmatrix}
 0 & 1 &  1 & 1 & 1 &  1 & 1 & 1 & 1 \\
0 & 1 & 1 & 1 & 1 & 0& 1 & 0 & 1 \\
 0 & 1 & 1  & 0 & 1  & 0 & 1 & 0 & 0\\
0& 0 & 1 & 0  & 1 & * & * & * &* \\
*&*&*&*&*&*&*&*&*\\ 
\end{bmatrix} \rightarrow \begin{bmatrix}
 1 & 0 &  0 & 0 & 0 &  0 & 0 & 0 & 0 \\
1 & 0 & 0 & 0 & 0 & 1& 0 & 1 & 0 \\
 1 & 0 & 0  & 1 & 0  & 1 & 0 & 1 & 1\\
1& 1 & 0 & 1  & 0 & * & * & * &* \\
*&*&*&*&*&*&*&*&*\\
\end{bmatrix} \rightarrow
\end{eqnarray*}
\begin{eqnarray*}
 \begin{bmatrix}
 0 & 1 & 0 & 1 & 1 &  0 & 0 & 0 & 0 \\
0 & 1 & 0 & 0 & 1 & 1& 0 & 1 & 0 \\
 0 & 0 & 0 & 0 & 1  & 1 & 0 & 1 & 1\\
0 & 0 & 0 &  0  & 1 & * & * & * &* \\
*&*&*&*&*&*&*&*&*\\
\end{bmatrix}\rightarrow \begin{bmatrix}
 0 & 1 & 0 & 1 & 1 &  1 & 1 & 0 & 1 \\
0 & 1 & 0 & 0 & 1 & 0& 1 & 0 & 1 \\
 0 & 0 & 0 & 0 & 1  & 0 & 0 & 0 & 0\\
0 & 0 & 0 &  0  & 1 & * & * & * &* \\
*&*&*&*&*&*&*&*&*\\
\end{bmatrix} \phantom{\rightarrow}
\end{eqnarray*}

\end{example}
\noindent First we show that $\iota(M) \in \mathbf{M}_{m \times \ell}$ by checking the conditions of definition \ref{matdef}:

\begin{enumerate}

\item Clearly the entries of $\iota(M)$ belong to the set $\{0,1,*\}$.

\item Any $*$ appearing in $\mathtt{rowRead}(\iota(M))$ lies to the right of any $0$ or $1$ appearing in $\mathtt{rowRead}(\iota(M))$ since $\iota$ does not affect which entries are equal to $*$.

\item The first entry of $\mathtt{rowRead}(\iota(M))$ is $0$. This follows from the fact that the rightmost entry of $\mathtt{rowRead}(M)$ which is not a $*$ is a $1$.  

\item The rightmost entry of $\mathtt{rowRead}(\iota(M))$ which is not a $*$ is a $1$.  This follows from the fact that   the first entry of $\mathtt{rowRead}(M)$ is $0$.

\item There is no appearance of a $0$ which lies directly above a $1$ in $\iota(M)$. This follows from the fact that it is true for $M$.
\end{enumerate}

Next it is obvious from construction that we have $w(M)=h(\iota(M))$ and that $h(M)=w(\iota(M))$.   

Next, the map $\iota$ induces an involution on the set of positions $\{(i,j): M_{i,j} \neq *\}$ in the obvious way (rotating the positions in each of the submatrices mentioned in the construction through 180 degrees) which we denote by $\hat{\iota}$.  We claim that for any position $p$ and $q$ such that $M$ has a $0$ in position $p$ and a $1$ in position $q$ that are in contact we have that $N=\iota(M)$ has a $0$ in position $\hat{\iota}(q)$ and a $1$ in position $\hat{\iota}(p)$ that are in contact. If $p$ and $q$ are in the same submatrix this is easy to see.  If $p$ is in the left submatrix and $q$ in the right then we must have that $p$ and $q$ lie in the same row. It follows that $\hat{\iota}(q)$ lies in the row above $\hat{\iota}(p)$ (and to its right).  Thus, in $N$, the $0$ in the former in is in contact with the $1$ in the latter.  If $p$ lies in the right submatrix and $q$ lies in the left submatrix then it must be that $p$ lies one row above $q$.  It follows that $\hat{\iota}(q)$ lies in the same row as $\hat{\iota}(p)$ (and to its left) so, in $N$, the $0$ in the former is in contact with the $1$ in the latter.  Therefore we have $|M| \leq |\iota(M)|$ and since it is clear that $\iota^2(M)=M$ we have that $|M|=|\iota(M)|$.

\end{proof}

\section{The Projection Theorem and its Proof}

\begin{definition}
Given $x \in \mathbf{D}_{\ell,m}$, define $\mathtt{lowest}(x)=\mathtt{right}^{i}(x)$ where $i$ is maximal such that $\mathtt{right}^i(x)$ is defined (and where $\mathtt{right}^i(x)=\mathtt{right} \circ \cdots \circ \mathtt{right}(x)$ with $i$ factors).  
\end{definition}

\begin{theorem}\label{project}
Let $\lambda=[p_0,\ldots,p_{\ell-2}] \in \mathbf{P}_{\ell-1}$ be the partition with $p_i$ parts of size $\ell-1-i$ and let $\mu=[p_1,\ldots,p_{\ell-2}] \in \mathbf{P}_{\ell-2}$ be the partition with $p_i$ parts of size $\ell-1-i$. Further suppose that $|\lambda| < (\ell-2)m$.   Then if $\mathtt{lowest}(f(g(\lambda)))=[0,a_1,\ldots,a_{\ell}]$ we have $\mathtt{lowest}(f(g(\mu)))=[a_1-(m-p_0),\ldots,a_{\ell}-(m-p_0)]$.  In particular, $a_1=m-p_0$.
\end{theorem}

\begin{proof}

 Let $N_0=g(\lambda) \in \mathbf{M}_{m \times \ell}^{<(\ell-2)m}$.  Let $N_0'$ denote the $m \times (\ell-1)$ matrix with every entry equal to $*$.  Form the array $N_0N_0'$ by appending the matrix $N_0'$ to the bottom of $N_0$ and flushed to the right (so that the shape of $N_0N_0'$ is that of a $2m \times \ell$ matrix except with positions $(m+1,1),\ldots,(2m,1)$ missing.  If $N_kN_k'$ is defined and columns $2$ through $\ell$ of $N_k$ contain at least one entry that is equal to $0$ or $*$, define $N_{k+1}N_{k+1}'=\mathtt{Right}(N_kN_k')$ where $\mathtt{Right}(NN')$ is defined as follows:

\begin{itemize}

\item If the rightmost column of $NN'$ has less than $m$ $1$s define $Point(NN')=0$. Otherwise let $j$ be minimal such that the difference between the number of $1$s  in the rightmost column of $NN'$ and the number of $1$s in the $j^{th}$ column of $NN'$ is less than $m$   and  define $Point(NN')=j$.
\item Remove the rightmost column of $NN'$, denote $C_{\ell}$, and slide all entries in a column $j \in [Point(NN')+1,\ell-1]$ one position to the right.   Then append a $1$ to the top of $C_{\ell}$ and remove its bottom entry and use the result to fill the evacuated positions (top to bottom).

\end{itemize}

\begin{example}\label{mats}
In the examples below we have $m=3$ and $\ell=5$.  The first arrow represents removing column $C_{\ell}$ and shifting the columns to the right of $Point(NN')$ to the right.  The second arrow represents shifting the entries of $C_{\ell}$ down and the third represents placing the resulting entries back into the matrix.\\
In this example we have $Point(NN')=1$.
\begin{eqnarray*}
NN'=\begin{bmatrix}
1&1&1&1& \textcolor{red}1 \\
1&1&1&1& \textcolor{red}1 \\
0&1&1&0& \textcolor{red}1 \\
&1&0&0& \textcolor{red}0 \\
&1&*&*& \textcolor{red}* \\
&*&*&*& \textcolor{red}* \\
\end{bmatrix} \rightarrow
\begin{bmatrix}
1& &1&1&1 \\
1& &1&1&1 \\
0& &1&1&0 \\
& &1&0&0 \\
& &1&*&* \\
& &*&*&* \\
\end{bmatrix},
\begin{bmatrix}
  \textcolor{red}1 \\
 \textcolor{red}1 \\
 \textcolor{red}1\\
 \textcolor{red}0\\
 \textcolor{red} *\\
 \textcolor{red}* \\
\end{bmatrix} \rightarrow
\begin{bmatrix}
 \textcolor{red} 1 \\
 \textcolor{red}1 \\
 \textcolor{red}1\\
 \textcolor{red}1\\
 \textcolor{red} 0\\
 \textcolor{red}* \\
\end{bmatrix} \rightarrow
\begin{bmatrix}
1&  \textcolor{red}1&1&1&1 \\
1&  \textcolor{red}1&1&1&1 \\
0&  \textcolor{red}1&1&1&0 \\
&  \textcolor{red}1&1&0&0 \\
&  \textcolor{red}0&1&*&* \\
&  \textcolor{red}*&*&*&* \\
\end{bmatrix}
\end{eqnarray*}\\
In this example we have $Point(NN')=0$.
\begin{eqnarray*}
NN'=\begin{bmatrix}
1&1&1&1& \textcolor{red}1 \\
1&1&1&1& \textcolor{red}1 \\
0&1&1&0& \textcolor{red}0 \\
&1&0&0& \textcolor{red}0 \\
&1&*&*& \textcolor{red}* \\
&*&*&*& \textcolor{red}* \\
\end{bmatrix} \rightarrow
\begin{bmatrix}
\phantom{0}& 1&1&1&1 \\
& 1&1&1&1 \\
& 0&1&1&0 \\
& &1&0&0 \\
& &1&*&* \\
& &*&*&* \\
\end{bmatrix},
\begin{bmatrix}
  \textcolor{red}1 \\
 \textcolor{red}1 \\
 \textcolor{red}0\\
 \textcolor{red}0\\
 \textcolor{red} *\\
 \textcolor{red}* \\
\end{bmatrix} \rightarrow
\begin{bmatrix}
 \textcolor{red}1 \\
 \textcolor{red}1 \\
 \textcolor{red}1\\
 \textcolor{red}0\\
 \textcolor{red} 0\\
 \textcolor{red}*\\
\end{bmatrix} \rightarrow
\begin{bmatrix}
\textcolor{red}{1}& 1&1&1&1 \\
\textcolor{red}{1}& 1&1&1&1 \\
\textcolor{red}{1}& 0&1&1&0 \\
& \textcolor{red}{0}&1&0&0 \\
& \textcolor{red}{0}&1&*&* \\
& \textcolor{red}{*}&*&*&* \\
\end{bmatrix}
\end{eqnarray*}
\end{example}

\begin{example}\label{N0}
Let $m=3$ and $\ell=5$ and suppose that $\lambda=(4,2,1)=[1,0,1,1]$. Applying $\mathtt{Right}$ starting from $N_0N_0'$ gives the sequence:
\begin{eqnarray*}
N_0N_0'=\begin{bmatrix}
0 & 1 &1 & 1 & 1\\
0 & 1 & 0 & 1 &* \\
* & *  & * & * & *\\
 & *  & * & * & *\\
 & *  & * & * & *\\
 & *  & * & * & *\\
\end{bmatrix} ,\,\,\,\,\,\,\,
N_1N_1'=\begin{bmatrix}
1 & 0 &1 & 1 & 1\\
1 & 0 &1 & 0 & 1 \\
* & * & * & * & *\\
 & *  & * & * & *\\
 & *  & * & * & *\\
 & *  & * & * & *\\
\end{bmatrix}, \ldots \\
\ldots, \,\,\, N_9N_9'=\begin{bmatrix}
1 & 1 &1 & 1 & 1\\
1 & 1 &1 & 1 & 0 \\
1 & 1 &1 & 1 & 0\\
 &  1 & 0 & 1 & *\\
 & *  & * & * & *\\
 & *  & * & * & *\\
\end{bmatrix},\,\,\,\,\,\,\,
N_{10}N_{10}'=\begin{bmatrix}
1 & 1 &1 & 1 & 1\\
1 & 1 &1 & 1 & 1 \\
0 & 1 &1 & 1 & 1\\
 &  0 & 1 & 0 & 1\\
 & *  & * & * & *\\
 & *  & * & * & *\\
\end{bmatrix}\\
\end{eqnarray*}
\end{example}

Define $\mathcal{F}(NN')=[0,a_1,\ldots,a_{\ell}]$ where $a_j$ is the number of $1$s in the $j^{th}$ column of $NN'$.  
Define $\mathcal{G}(NN')$ to be $(\rho_1,\ldots,\rho_z)$ where $\rho_i$ is the number of $1$s that the $i^{th}$ $0$ of $NN'$ (under the row reading order, $\prec$) is in contact with.  Note that $\mathcal{F}(N_0N_0')=f(N_0)$ whereas $\mathcal{G}(N_0N_0')=g^{-1}(N_0)$.

Let $r$ be maximal such that $N_rN_r'$ is defined ($r=10$ in  example \ref{N0} since the first $m$ rows of columns $2$ to $\ell$ of $N_{10}N_{10}'$ only contain $1$s).  We make the following claims:

\begin{enumerate}

\item There is no $0$ over a $1$ in $N_kN_k'$

\item $\mathcal{F}(N_kN_k')=\mathtt{right}^k(f(g(\lambda))$ for $0 \leq k \leq r$.

\item $\mathcal{G}(N_kN_k')=\lambda$ for $0 \leq k \leq r$.

\item $\mathtt{rowRead}(N_kN_k')$ is composed of $h(\lambda)+w^{\ell-1}(\lambda)+k$ entries that are $0$ or $1$ ending on $1$ and then followed by all $*$s.

\end{enumerate}

It is evident that these claims are true for $k=0$.  Suppose they are true for some $k \leq r$.  We will show that this implies they are also true for $k+1$ if $k+1 \leq r$.  

First suppose that $Point(N_kN_k')>1$.  
Let $A$ be the array of $\ell$ columns of varying height all flushed to the top whose $j^{th}$ column is composed of as many $1$s as are in the $j^{th}$ column of $N_kN_k'$ followed by $m$ $0$s.  As in the proof of claim \ref{fitin} we have that $degr(\mathcal{F}(N_kN_k'))$ is given my the number of $(0,1)$ pairs in $A$ in contact minus the number of $1$s below row $m$ in $A$.  Let $(i_1,j_1)$ be minimal in the ordering $\prec$ such that $A_{i_1,j_1}=0$.  Let $(i_2,j_2)$ be minimal in the ordering $\prec$ such that $j_2 \neq j_1$ and $A_{i_2,j_2}=0$. 

If $k<r$ then by definition of $r$ there is a $0$ somewhere in the first $m$ rows of columns $2$ to $\ell$ of $N_kN_k'$.    If $k=r$ then the previous is true of $N_{k-1}N_{k-1}'$ but not of $N_kN_k'$.  In light of the definition of $\mathtt{Right}$ this is only possible if  $N_{k-1}$ has a $0$ in its bottom right corner which then implies that $N_k'$ has a $0$ in its top left corner.  It follows that $A$ has a $0$ in columns $2$ to $\ell$ in a position weakly before (in the $\prec$ order) $(m+1,2)$.  Moreover, the shape of $N_kN_k'$ implies that the top $0$ in column $1$ of $A$ is in row $m+1$ or higher.  It follows that $(i_1,j_1) \prec (i_2,j_2) \preceq (m+1,2)$.  

Consider the set of positions $S= \{(i,j): j \notin \{j_1,j_2\}, \, (i_1,j_1) \preceq (i,j) \prec (i_1+m,j_1)$.  By the assumption that $Point(N_kN_k') >1 $ there are at least $m$ more $1$s in column $\ell$ of $A$ than in column $1$. Since $(i_1,j_1)$ weakly precedes (in $\prec$) the top $0$ in clumn $1$ of $A$ it follows that every position in $S$ lies in a row ending in $1$.  Therefore, every $0$ in $S$ is in contact with a $1$ in column $\ell$.  Moreover, by the construction of $S$, for every $1$ in $S$ there is a $0$ in column $j_1$ in contact with it.  Since the minimality of $(i_1,j_1)$ implies there are no empty positions in $S$, so far we have counted $(\ell-2)m$ $(0,1)$ pairs in contact and none of these pairs feature a $0$ from column $j_2$.

Now consider a fixed column $j \notin \{j_1,j_2\}$.  Suppose the lowest $1$ in column $j$ lies in $(i,j)$.   Since the lowest $0$ in column $j_2$ lies in $(i_2+m-1,j_2)$ if we have $(i_2+m-1,j_2) \prec (i,j)$ then every $0$ in column $j_2$ (there are $m$) lies in contact with a $1$ in column $j$ (there are at most $m$ $1$s in column $j$ below row $m$ due to the shape of $N_kN_k'$).  If  $(i,j) \prec (i_2+m-1,j_2)$ then for every $1$ in column $j$ below row $m$ there is a $0$ in column $j_2$ in contact with it (since $(i_2,j_2) \preceq (m+1,2)$ and every $1$ in $A$ below row $m$ is weakly after $(m+1,2)$).  Thus the number of $(0,1)$ pairs with a $0$ in column $j_2$ and a $1$ in column $j$ is at least the number of $1$s in column $j$ below row $m$.  It follows that the number of $(0,1)$ pairs in contact in $A$ that feature a $0$ in column $j_2$ is at least the number of $1$s below row $m$ in $A$.  All this implies that $degr(\mathcal{F}(N_kN_k')) \geq (\ell-2)m$ which is a contradiction so we may assume that $Point(N_kN_k') \in \{0,1\}$.

 Now suppose that $k<r$ and $Point(N_kN_k')=0$.  We prove each of the four claims listed earlier for $k+1$:

\begin{enumerate}
\item Columns $3$ through $\ell$ of $N_{k+1}N_{k+1}'$ were columns in $N_kN_k'$ so there is no $0$ over a $1$ these columns.  Column $1$ of  $N_{k+1}N_{k+1}'$ is formed by appending a $1$ to the top $m-1$ entries of column $\ell$ of $N_kN_k'$ so there is no $0$ over a $1$ in column $1$ of $N_{k+1}N_{k+1}'$. Column $2$ of $N_{k+1}N_{k+1}'$ is formed from column $1$ of  $N_kN_k'$ followed by the entries in rows $m$ through $2m-1$ of column $\ell$ of $N_kN_k'$.  However since $Point(N_kN_k')=0$ no entry on or below row $m$ of column $\ell$ of $N_kN_k'$ is equal to $1$ so column $2$  of $N_{k+1}N_{k+1}'$ also has no $0$ over a $1$.

\item As mentioned above, no entry on or below row $m$ in column $\ell$ of $N_kN_k'$ is equal to $1$ so all $1$s in column $\ell$ of $N_kN_k'$ are added to the leftmost column of $N_{k+1}N_{k+1}'$ (along with one additional $1$) and none to column $2$.  This along with the fact that $point(\mathcal{F}(N_kN_k'))=Point(N_kN_k')=0$ implies that $\mathcal{F}(N_{k+1}N_{k+1}')=right(\mathcal{F}(N_kN_k'))$.

\item It follows from positional considerations that the $i^{th}$ $0$ of $N_kN_k'$ is in contact with the same number of $1$s as is the $i^{th}$ $0$ of $N_{k+1}N_{k+1}'$ as long as the $i^{th}$ $0$ of $N_kN_k'$ does not occur in row $m-1$ and column $\ell$.  If this were the case it could potentially be in contact with one less $1$ in $N_{k+1}N_{k+1}'$ than in $N_kN_k'$.  However, this would require the entry in row $m$ and column $\ell$ of $N_kN_k'$ to be $1$ which is impossible (since the $i^{th}$ $0$ is above this position).
It follows that $\mathcal{G}(N_{k+1}N_{k+1}')=\mathcal{G}(N_kN_k')$.

\item The final property is obvious as $\mathtt{rowRead}(N_{k+1}N_{k+1}')$ is attained from $\mathtt{rowRead}(N_{k}N_{k}')$ by appending a $1$ to the beginning and deleting the last entry (which must be a $*$ as it cannot be a $1$ because $Point(N_{k}N_{k}')=0$ implies there is a $0$ in column $\ell$ of  $N_{k}N_{k}'$ and it cannot be a $0$ because this would contradict the fact that property (4) holds for $k$).   

\end{enumerate}

Now suppose that $k<r$ and $Point(N_kN_k')=1$. We prove each of the four claims for $k+1$: 

\begin{enumerate}

\item  All columns of $N_{k+1}N_{k+1}'$ appeared as columns in $N_kN_k'$ except column $2$ of $N_{k+1}N_{k+1}'$ which is column $\ell$ $N_kN_k'$ with a $1$ appended to the top and the bottom entry removed.  It follows there is no $0$ over a $1$ in $N_{k+1}N_{k+1}'$.  

\item $Point(N_kN_k')=1$ which means that $point(\mathcal{F}(N_kN_k'))=1$ and also ensures that there is no $1$ in row $2m$ and column $\ell$ of $N_kN_k'$.  This means that all of the $1$s in column $\ell$ of $N_kN_k'$ are placed into column $2$ of $N_{k+1}N_{k+1}'$ along with one additional $1$. From this and the definitions, it is easy to see that $\mathcal{F}(N_{k+1}N_{k+1}')=right(\mathcal{F}(N_kN_k'))$.  

\item We claim that the $i^{th}$ $0$ of $N_kN_k'$ is in contact with the same number of $1$s as the $i^{th}$ $0$ of $N_{k+1}N_{k+1}'$.   Since $Point(N_kN_k')=1$, column $\ell$ of $N_kN_k'$ begins with at least $m$ $1$s so that only $1$s are placed into the first $m$ rows of column $2$ of  $N_{k+1}N_{k+1}'$.  This implies the following:

\begin{itemize}
\item  If the $i^{th}$ $0$ of $N_kN_k'$ appears in position $(p,1)$ then the $i^{th}$ $0$ of $N_{k+1}N_{k+1}'$  also appears in $(p,1)$.  Since exactly one $1$ is removed from row $p$ and one $1$ is added to row $p$ the number of $1$s in row $p$ of $N_kN_k'$ and $N_{k+1}N_{k+1}'$ is the same.  But these numbers are the number of $1$s that the $i^{th}$ $0$ of $N_kN_k'$ is in contact and the number of $1$s that the $i^{th}$ $0$ of $N_{k+1}N_{k+1}'$ is in contact with, respectively.
\item If $j \in [2,\ell-1]$ and the $i^{th}$ $0$ of $N_kN_k'$ appears in position $(p,j)$ then the $i^{th}$ $0$ of $N_{k+1}N_{k+1}'$   appears in position  $(p,j+1)$.  Further, the number of $1$s in $N_kN_k'$ in the interval $[(p,j),(p+1,j)]$ under $\prec$ is the same as the number of $1$s in $N_{k+1}N_{k+1}'$ in the interval $[(p,j+1),(p+1,j+1)]$ under $\prec$.  

\item If the $i^{th}$ $0$ of $N_kN_k'$ appears in position $(p,\ell)$ then $p>m$ and the $i^{th}$ $0$ of $N_{k+1}N_{k+1}'$   appears in position  $(p+1,2)$.  In this case the $i^{th}$ $0$ of $N_kN_k'$ is in contact with exactly the number of $1$s that appear in row $p+1$ of of $N_kN_k'$.  Since $p>m$ there is no entry in position $(p+1,1)$ so the $i^{th}$ $0$ of $N_{k+1}N_{k+1}'$ is in contact with exactly the number of $1$s that appear in row $p+1$ of $N_{k+1}N_{k+1}'$.  Since neither position $(p,\ell)$ nor $(p+1,\ell)$ of $N_kN_k'$ contain a $1$ these numbers are the same.  
\end{itemize}
Thus the number of $1$s that the $i^{th}$ $0$ of $N_kN_k'$ is in contact with is the same as the number of $1$s that the $i^{th}$ $0$ of $N_{+1}N_{k+1}'$ is in contact with.  It follows that $\mathcal{G}(N_{k+1}N_{k+1}')=\mathcal{G}(N_kN_k')$.

 \item This is true for $N_{k+1}N_{k+1}'$  because it is true for $N_kN_k'$  and because row $2m$ and column $\ell$ of $N_kN_k'$ contains a $*$ (otherwise it would have to be a $1$ and would force $Point(N_kN_k') \geq 2$).

\end{enumerate}

By definition of $r$ we have that $N_r$ contains only $1$s in columns $2$ to $\ell$. Since $r$ is minimal, this is not true of $N_{r-1}$.  The only way both of these statements are possible is if $N_{r-1}N_{r-1}'$ has a $0$ in row $m$ and column $\ell$.  It follows that row $m+1$ of column $2$ of $N_rN_r'$  contains a $0$, that is, the leftmost column of $N_r'$ is all $0$s and $*$s which implies that $N_r' \in \mathbf{M}_{m \times (\ell-1)}$.  Now since $\mathcal{G}(N_rN_r')=\lambda=[p_0,p_1,\ldots,p_{\ell-2}]$,  column $1$ of $N_rN_r'$ must contain $p_0$ $0$s (since no $0$ in $N_r'$ can be in contact with more than $\ell-2$ $1$s).

If we let $\mathbf{1}_{m \times \ell}$ denote the $m \times \ell$ matrix of all $1$s and $\mathbf{1}_{m\times\ell}N_r'$ the result of appending $N_r'$ below it (flushed to the right) then this implies that $\mathcal{G}(\mathbf{1}_{m\times\ell}N_r')=[0,p_1,\ldots,p_{\ell-2}]$. The latter is equivalent to the statement that $g^{-1}(N_r')=[p_1,\ldots,p_{\ell-2}]$ or that $g(\mu)=N_r'$ (where $g$ is being applied to $\mu$ considered as an $\ell-2$ bounded partition).

 Now $|\mu| < (\ell-2)m-p_0(\ell-1) \leq (\ell-2)(m-p_0)$, so that $|N_r'| < (\ell-2)(m-p_0)$.  Suppose that $N_r'$ has more than $m-p_0$ $1$s in some column $j \in [1, \ell-2]$  or that column $\ell-1$ has $m-p_0$ $1$s.  Since $N_r' \in \mathbf{M}_{m \times (\ell-1)}$ it follows that column $1$ of $N_r'$ contains at least $m-p_0$ non-$*$ entries which must all be $0$s. Each of these $0$s is in contact with all the $1$s in its row.   Moreover, all $0$s in the first $m-p_0$ rows of columns $2$ to $\ell-1$ of $N_r'$ are in contact with a $1$ in column $j$.  Thus the number of $(0,1)$ pairs in contact in $N_r'$ is at least $(\ell-2)(m-p_0)$ so we would have $|N_r'|\geq (\ell-2)(m-p_0)$ since $N_r'$ can have no $1$s below row $m$ as it has only $m$ rows.  Thus the assumption of the second sentence of this paragraph is impossible and it follows that the bottom $p_0$ rows of $N_r'$ are all composed of all $*$s.

Now let us apply $\mathtt{Right}^{p_0(\ell-1)}$ to $N_r' \mathbf{X}_{m \times (\ell-2)}$, the array formed by appending the $m \times (\ell-2)$ matrix of all $*$s to the bottom of $N_r'$ and flushed to the right.  Before each of these $p_0(\ell-1)$ applications of $\mathtt{Right}$  the rightmost entry in row $m$ is still $*$, from which it follows that the action of $\mathtt{Right}$ is just to move each entry one position forward in the order $\prec$  (and add a $1$ in the top left position).  Therefore $\mathtt{Right}^{p_0(\ell-1)}(N_r' \mathbf{X}_{m \times (\ell-2)})$ is just the matrix whose first $p_0$ rows are all $1$s and whose next $m-p_0$ rows are the top $m-p_0$ rows of $N_r'$ and whose last $m$ rows are all $*$s.

Now suppose that $\mathtt{right}^r(f(g(\lambda)))=[0,a_1,\ldots,a_{\ell}]$ which in turn  means that $\mathcal{F}(N_rN_r')=[0,a_1,\ldots,a_{\ell}]$.  Since column $1$ of $N_r$ contains $m-p_0$ $1$s this says that $a_1=m-p_0$. Now for $j \in [2, \ell]$ column $j$ of $N_rN_r'$ contains $m$ more $1$s than column $j-1$ of $N_r'$ and column $j-1$ of $\mathtt{Right}^{p_0(\ell-1)}(N_r' \mathbf{X}_{m \times (\ell-2)})$ contains $p_0$ more $1$s than column $j-1$ of $N_r'$. Since column $j$ of $N_rN_r'$ contains $a_j$ $1$s, column $j-1$ of $\mathtt{Right}^{p_0(\ell-1)}(N_r' \mathbf{X}_{m \times (\ell-2)})$  contains $a_j-m+p_0$ $1$s.  In other words we have $\mathcal{F}(\mathtt{Right}^{p_0(\ell-1)}(N_r' \mathbf{X}_{m \times (\ell-2)}))=[0,a_2-(m-p_0),\ldots,a_{\ell}-(m-p_0)]$.  However, it is not hard to see that 
$\mathcal{F}(\mathtt{Right}^{p_0(\ell-1)}(N_r' \mathbf{X}_{m \times (\ell-2)}))=\mathtt{right}^{p_0(\ell-1)}\mathcal{F}((N_r' \mathbf{X}_{m \times (\ell-2)}))$ and since we have $g(\mu)=N_r'$ this in turn  equals $\mathtt{right}^{p_0(\ell-1)}f(g(\mu)))$.

Let $x^q=\mathtt{right}^{r+q}(f(g(\lambda)))$ and let $y^q=\mathtt{right}^{p_0(\ell-1)+q}f(g(\mu)))$.  We claim that for any $q \geq 0$ such that $x^q$ and $y^q$ are defined if $x^q=[0,a_1^q,\ldots,a_{\ell}^q]$ then $a_1^q=m-p_0$, and $a_j^q \geq m$ for $j \in [2,\ell]$, and $y^q=[0,a_2^q-(m-p_0),\ldots,a_{\ell}^q-(m-p_0)]$.  Moreover, if $q>0$ then $a_2^q>m$.  

We have already established this for $q=0$.  Suppose it is true for some $q \geq 0$ and $x^q$ and $y^q$ are rightable.  Then $a_{\ell}^q \geq m$ so $point(x^q)>0$.  It follows from this that $point(y^q)=point(x^q)-1$ and that if $x_{q+1}=[0,a_1^{q+1},\ldots,a_{\ell}^{q+1}]$ then $a_1^{q+1}=m-p_0$ and  $y^q=[0,a_2^{q+1}-(m-p_0),\ldots,a_{\ell}^{q+1}-(m-p_0)]$, and since  $a_j^q \geq m$ for $j \in [2,\ell]$ that  $a_j^{q+1} \geq m$ for $j \in [2,\ell]$.   If $point(x^q)=1$ it also follows that $a_2^{q+1}>m$ since $a_{\ell}^q \geq m$.  Additionally if $point(x^q)>1$ and $q>0$  then $a_2^{q+1}=a_2^{q}$ so it follows that $a_2^{q+1}>m$ since $a_2^q >m$.  It is only left to show that $a_2^{q+1}>m$ if $point(x^q)>1$ and $q=0$.  But the fact that $Point(N_rN_r') \in \{0,1\}$ shows that  $point(x^0)>1$ is impossible.

 Now let $q$ be minimal such that one or both of $x^q$ and $y^q$ is not rightable.   Suppose that $point(x^q)=s$.  Then $s>0$ and $x^q$ is rightable if and only if $s< \ell$ and $a_i^q-a_{i+2}^q < -m$ for $i \in [0,s-2]$.  Meanwhile $y^q$ is rightable if and only if $s< \ell$ and $(a_i^q-(m-p_0))-(a_{i+2}^q-(m-p_0)) < -m$ for $i \in [1,s-2]$. But $a_0^q-a_2^q=-a_2^q \leq -m$ with strict inequailty if $q>0$.  Thus unless $q=0$ the statements are equivalent.  But if $q=0$ then $s=1$ and the sets $[0,s-2]$ and $[1,s-2]$ are empty so both $x^q$ and $y^q$ are rightable.  It follows that there is some $q$ such that $x^q=\mathtt{lowest}(f(g(\lambda)))=[0,a_1^q,\ldots,a_{\ell}^q]$ where $a_1^q=m-p_0$ and $y^q=\mathtt{lowest}(f(g(\mu)))=[0,a_2^q-(m-p_0),\ldots,a_{\ell}^q-(m-p_0)]$.  This completes the proof of the theorem.  
\end{proof}

\begin{example}
Let $m=3$ and $\ell=5$ and $\lambda=(4,2,1)=[1,0,1,1]$ (as in example \ref{mats}) and $\mu=(2,1)=[0,1,1]$.   For $q \geq -r$ set $x^{q}=\mathtt{right}^{r+q}(f(g(\lambda)))$
and similarly for each  $q \geq -p_0(\ell-1)$ set $y^{q}=\mathtt{right}^{p_0(\ell-1)+q}(f(g(\mu)))$.  We have $r=10$ and $p_0(\ell-1)=4$  so that $f(g(\lambda))=x^{-10}$ and $f(g(\mu))=y^{-4}$ (shown in green).  $x^{0}=\mathtt{right}^{r}(f(g(\lambda)))$
and $y^{0}=\mathtt{right}^{p_0(\ell-1)}(f(g(\mu)))$ are shown in red.  
Since $x^{19}$ and $y^{19}$ are not rightable, $\mathtt{lowest}(f(g(\lambda)))=x^{19}$ and $\mathtt{lowest}(f(g(\mu)))=y^{19}$ (shown in blue).  

\begin{eqnarray*}
\setcounter{MaxMatrixCols}{23}
\setlength\arraycolsep{2.7pt}\begin{matrix}
\textcolor{green}{x^{-10}}&=&[&\textcolor{green}0&\textcolor{green}0&\textcolor{green}2&\textcolor{green}1&\textcolor{green}2&\textcolor{green}1&]&\phantom{12}&&&&&&&&\\
x^{-9}&=&[&0&2&0&2&1&2&]&\phantom{12}&&&&&&&&\\
x^{-8}&=&[&0&3&2&0&2&1&]&\phantom{12}&&&&&&&&\\
x^{-7}&=&[&0&2&3&2&0&2&]&\phantom{12}&&&&&&&&\\
x^{-6}&=&[&0&3&2&3&2&0&]&\phantom{12}&&&&&&&&\\
x^{-5}&=&[&0&1&3&2&3&2&]&\phantom{12}&&&&&&&&\\
x^{-4}&=&[&0&3&1&3&2&3&]&\phantom{12}&[&\textcolor{green}0&\textcolor{green}0&\textcolor{green}1&\textcolor{green}0&\textcolor{green}1&]&=&\textcolor{green}{y^{-4}}\\
x^{-3}&=&[&0&3&4&1&3&2&]&\phantom{12}&[&0&2&0&1&0&]&=&y^{-3}\\
x^{-2}&=&[&0&3&3&4&1&3&]&\phantom{12}&[&0&1&2&0&1&]&=&y^{-2}\\
x^{-1}&=&[&0&3&4&3&4&1&]&\phantom{12}&[&0&2&1&2&0&]&=&y^{-1}\\
\textcolor{red}{x^0}&=&[&\textcolor{red}0&\textcolor{red}2&\textcolor{red}3&\textcolor{red}4&\textcolor{red}3&\textcolor{red}4&]&\phantom{12}&[&\textcolor{red}0&\textcolor{red}1&\textcolor{red}2&\textcolor{red}1&\textcolor{red}2&]&=&\textcolor{red}{y^0}\\
x^1&=&[&0&2&5&3&4&3&]&\phantom{12}&[&0&3&1&2&1&]&=&y^1\\
x^2&=&[&0&2&4&5&3&4&]&\phantom{12}&[&0&2&3&1&2&]&=&y^2\\
x^3&=&[&0&2&5&4&5&3&]&\phantom{12}&[&0&3&2&3&1&]&=&y^3\\
&\phantom{12}&\phantom{12}&\phantom{12}&\phantom{12}&\phantom{12}
&\phantom{12}&\phantom{12}&\phantom{12}&\phantom{12}&\phantom{12}
&\phantom{12}&\phantom{12}&\phantom{12}&\phantom{12}&\phantom{12}
&\phantom{12}&\phantom{12}&\phantom{12}&\phantom{12}\\
\vdots&&&\vdots&\vdots&\vdots&\vdots&\vdots&\vdots&&\phantom{12}&&\vdots&\vdots&\vdots&\vdots&\vdots&&&\vdots\\
&&&&&&&&&&\phantom{12}&&&\\
x^{15}&=&[&0&2&5&7&8&9&]&\phantom{12}&[&0&3&5&6&7&]&=&y^{15}\\
x^{16}&=&[&0&2&5&7&10&8&]&\phantom{12}&[&0&3&5&8&6&]&=&y^{16}\\
x^{17}&=&[&0&2&5&7&9&10&]&\phantom{12}&[&0&3&5&7&8&]&=&y^{17}\\
x^{18}&=&[&0&2&5&7&9&11&]&\phantom{12}&[&0&3&5&7&9&]&=&y^{18}\\
\textcolor{blue}{x^{19}}&=&[&\textcolor{blue}0&\textcolor{blue}2&\textcolor{blue}5&\textcolor{blue}7&\textcolor{blue}9&\textcolor{blue}{12}&]&\phantom{12}&[&\textcolor{blue}0&\textcolor{blue}3&\textcolor{blue}5&\textcolor{blue}7&\textcolor{blue}{10}&]&=&\textcolor{blue}{y^{19}}\\

\end{matrix}
\end{eqnarray*}
\end{example}

\section{Strings and their Extensions}

\begin{definition}
We say that $x \in \mathbf{D}_{\ell,m}$ is disconnected if there exists $i \geq 0$ such that $\mathtt{left}^i(x) \in \mathbf{D}_{\ell,m}$ and is unleftable.  We say $x \in \mathbf{D}_{\ell,m}$ is connected if it is not disconnected. 
\end{definition}

\begin{definition}
Let $\lambda \in \mathbf{P}_{\ell-1}^{<(\ell-1)m}$. Define $\mathtt{string}(\lambda)=\{v_0,\ldots,v_b\}$ where we define  $v_0=f(g(\lambda))$ and where $v_b=\mathtt{lowest}(v_0)$ and for $0 < i < b$ $v_i=\mathtt{right}^i(v_0)$.  (Note that the definition depends on $m$ since the definition of $\mathtt{lowest}$ depends on $m$.)
\end{definition}

\begin{claim}\label{connected}
Let $d < (\ell-1)m$ and let $\mathbf{C}_{\ell,m}^d \subseteq \mathbf{D}_{\ell,m}$ denote the subset of elements which are connected and have degree $d$.  
\begin{eqnarray*}
\mathbf{C}_{\ell,m}^d=\bigcup_{\lambda \in \mathbf{P}_{(\ell-1)}^d} \mathtt{string}(\lambda)
\end{eqnarray*}
where the union is disjoint and $\lambda \in \mathbf{P}_{(\ell-1)}^d$ if and only if $\lambda \in \mathbf{P}_{(\ell-1)}$ and $|\lambda|=d$.
\end{claim}
\begin{proof}
Assume $x \in \mathtt{string}(\lambda) \cap \mathtt{string}(\lambda^*)$ where $\lambda \neq \lambda^*$. Thus we have  $\mathtt{left}^i(x)=t$ and $\mathtt{left}^{i^*}(x)=t^*$ for some $i \neq i^*$ and $t \neq t^* \in \mathbf{T}_{\ell,m}^d$ by injectivity of $f \circ g$ and the fact that $\mathtt{left}$ and $\mathtt{right}$ are inverses.  WLOG $i<i^*$ so $t^*=\mathtt{left}^{i^*-i}(t)$.  Writing $t=[0,a_1,\ldots,a_{\ell}]$ we have $a_1=0$ since $t \in \mathbf{T}_{\ell,m}^d$ which implies $a_2 \leq m$ so that $0-a_2 \geq -m$ and $pair(t)=0$.   Since $a_1=0$ this implies $\mathtt{left}(t)$ contains a negative position coordinate and therefore  so does $t^*=\mathtt{left}^{i^*-i}(t)$, which contradicts $t^* \in \mathbf{D}_{\ell,m}^d$.

Now let $x \in \mathbf{C}_{\ell,m}^d$. Suppose there is no $i \geq 0$ such that $\mathtt{left}^i(x) \in \mathbf{T}_{\ell,m}^d$.  Since it is clear we cannot have $\mathtt{left}^i(x) \in \mathbf{D}_{\ell,m}^d$ for all $i$ (eventually some position coordinate must become negative) there must exist $x' \in \mathbf{D}_{\ell,m}^d \setminus \mathbf{T}_{\ell,m}^d$ with $\mathtt{left}(x') \notin \mathbf{D}_{\ell,m}^d$.  But this contradicts claim \ref{leftcheck}.  Therefore such $i$ exists and $x \in \mathtt{string}(g^{-1}(f^{-1}( \mathtt{left}^i(x) )))$.

\end{proof}

\begin{claim}\label{takeoff}
Suppose $x=[0,a_1,\ldots,a_{\ell}] \in \mathbf{N}_{\ell,m}^d$ where $d < m(\ell-2)$.  Then $x'=[0,a_2-a_1,\ldots,a_{\ell}-a_1] \in \mathbf{N}_{\ell-1,m}^{d-(m-a_1)(\ell-1)}$
\end{claim}

\begin{proof}
Since $x \in \mathbf{N}_{\ell,m}^d$ there is $q \geq 0$ such that $z=\mathtt{left}^q(x)=[0,e_1,\ldots,e_{\ell}] \in \mathbf{D}_{\ell,m}^d$ is unleftable.  Suppose there is some $r \in [0,q]$ such that writing  $y=\mathtt{left}^r(x)=[0,c_1,\ldots,c_{\ell}]$ then either $c_2 \leq m$ or $c_j<m$ for some $j \in (2,\ell]$ and let us assume we have chosen the maximal such $r$.  Suppose that $r<q$ and let $\mathtt{left}(y)=[0,\overline{c_1},\overline{c_2},\ldots,\overline{c_{\ell}}]$.  If $c_2 \leq m$  then $pair(y)=0$ and so $\overline{c_{\ell}}=c_1-1 <m$ since $0-c_1 \geq -m$.  If $c_j<m$ for some $j \in (2,\ell]$  then either $\overline{c_j}=c_j$ or $\overline{c_{j-1}}=c_j$ or $\overline{c_{\ell}}=c_j-1$ so one of these is less than $m$.  Therefore either $\overline{c_2} \leq m$ or else $\overline{c_j}<m$ for some $j \in (2,\ell]$ which  contradicts the maximality of $r$.

Therefore we must have that $r=q$ so either $e_2 \leq m$ or $e_j<m$ for some $j \in (2,\ell]$.  Now $pair(z)>0$ since otherwise the fact that $z$ is unleftable would imply that $e_1-e_{\ell}>m+1$ which is impossible since $0-e_1 \geq -m$.     Therefore $i=pair(z)$ is the minimal $i>0$ such that $e_i-e_{i+2} \geq -m$ and  we have $e_{i+1}-e_{\ell}>m+1$ since $z$ is unleftable. Therefore the hypotheses of corollary \ref{bound3} apply to $z$ so we have $degr(z) \geq (\ell-2)m$ which is a contradiction since $degr(x)=degr(z)$.

 It follows that for all $r \in [0,q]$ such that $y=\mathtt{left}^r(x)=[0,c_1,\ldots,c_{\ell}]$ we have $c_2>m$ and $c_j \geq m$ for all $j \in (2, \ell]$.  Therefore corollary \ref{corfitsin} implies  that for all such $y$ we  have that the element $y'=[0,c_2-c_1,\ldots,c_{\ell}-c_1] \in \mathbf{D}_{\ell-1,m}^{d-(m-c_1)(\ell-1)}$. Moreover, since $c_2>m$ it follows that $pair(y)>0$ for all such $y$.  Since $x=[0,a_1,\ldots,a_{\ell}]$ and  $x'=[0,a_2-a_1,\ldots,a_{\ell}-a_1]$ this implies (by induction) that if $\mathtt{left}^r(x)=[0,c_1,\ldots,c_{\ell}]$ then  $\mathtt{left}^r(x')=[0,c_2-a_1,\ldots,c_{\ell}-a_1]$.    When $r=q$ this implies that  $\mathtt{left}^q(x')=z'$ where $z'=[0,e_2-a_1,\ldots,e_{\ell}-a_1]$.  It follows from the facts that $pair(z)>0$ and that $z=[0,e_1,\ldots,e_{\ell}]$ is unleftable that $z'$ is also unleftable which implies that    $x' \in \mathbf{N}_{\ell-1,m}^{d-(m-a_1)(\ell-1)}$.

\end{proof}

\begin{claim}\label{addon}
Suppose that $x=[0,a_2,\ldots,a_{\ell}] \in \mathbf{N}_{\ell-1,m}^d$ and suppose $a_1 \leq m$ is such that $(m-a_1)(\ell-1)+d < m(\ell-2)$.  Then $x'=[0,a_1,a_2+a_1,\ldots,a_{\ell}+a_1] \in  \mathbf{N}_{\ell,m}^{d+(m-a_1)(\ell-1)}$ 
\end{claim}

\begin{proof}

Since $x \in \mathbf{N}_{\ell,m}^d$ there is $q \geq 0$ such that $z=\mathtt{left}^q(x)=[0,e_2,\ldots,e_{\ell}] \in \mathbf{D}_{\ell,m}^d$ is unleftable.  Suppose there is some $r \in [0,q]$ such that writing  $y=\mathtt{left}^r(x)=[0,c_2,\ldots,c_{\ell}]$ then either $c_2 \leq m-a_1$ or $c_j<m-a_1$ for some $j \in (2,\ell]$ and let us assume we have chosen the maximal such $r$.  Suppose that $r<q$ and let $\mathtt{left}(y)=[0,\overline{c_2},\ldots,\overline{c_{\ell}}]$.  If $c_2 \leq m-a_1$  then either $\overline{c_2}=c_2 \leq m-a_1$ or else $\overline{c_{\ell}}=c_2-1<m-a_1$.  If $c_j<m$ for some $j \in (2,\ell]$  then either $\overline{c_j}=c_j$ or $\overline{c_{j-1}}=c_j$ or $\overline{c_{\ell}}=c_j-1$ so one of these is less than $m-a_1$.  Therefore either $\overline{c_2} \leq m-a_1$ or else $\overline{c_j}<m-a_1$ for some $j \in (2,\ell]$ which  contradicts the maximality of $r$.  

It follows that we must have that $r=q$ which means that either $e_2 \leq m-a_1$ or $e_j<m-a_1$ for some $j \in (2,\ell]$.  Since $z$ is unleftable, if $i$ is minimal such that $e_i-e_{i+2} \geq-m$, then $e_{i+1}-e_{\ell}>m+1$ (note that $i>0$ by definition).  Further corollary \ref{corfitsout} implies that the element $z'=[0,a_1,e_2+a_1,\ldots,e_{\ell}+a_1] \in \mathbf{D}_{\ell,m}^{\leq d+(m-a_1)(\ell-1)}$ and the previous sentence implies that if $i>0$ is minimal such that $(e_i+a_1)-(e_{i+2}+a_1) \geq-m$ then $(e_{i+1}+a_1)-(e_{\ell}+a_1)>m+1$. Since the first sentence of the paragraph implies that either $e_2+a_1 \leq m$ or $e_j+a_1<m$ for some $j \in (2,\ell]$ the hypotheses of corollary \ref{bound3} apply to $z'$.  It follows that we have $degr(z') \geq (\ell-2)m$ which means that $d+(m-a_1)(\ell-1) \geq (\ell-2)m$ which contradicts an assumption of the original claim.

It follows that for any $r \in [0,q]$, if we write $y=\mathtt{left}^r(x)=[0,c_2,\ldots,c_{\ell}]$ then $c_j \geq m-a_1$ for all $j \in (2,\ell]$ and  $c_2 > m-a_1$, the latter of which implies $pair([0,a_1,c_2+a_1,\ldots,c_{\ell}+a_1])>0$. Now since $x=[0,a_2,\ldots, a_{\ell}]$ and  $x'=[0,a_1,a_2+a_1,\ldots,a_{\ell}+a_1]$ this implies (by induction) that if $\mathtt{left}^r(x)=[0,c_2,\ldots,c_{\ell}$] we must have $\mathtt{left}^r(x')=[0,a_1,c_2+a_1,\ldots,c_{\ell}+a_1]$. When $r=q$ this means that we have $\mathtt{left}^q(x')=z'$ where  $z'=[0,a_1,e_2+a_1,\ldots,e_{\ell}+a_1]$ and that $pair(z')>0$, which in light of the fact that $z$ is unleftable implies $z'$ is unleftable.   When $r=0$ the first sentence implies that for all $j \in [2,\ell]$, $a_j \geq m-a_1$, i.e., $a_j+a_1 \geq m$ so that  we may apply corollary \ref{corfitsin} to $x'=[0,a_1,a_2+a_1,\ldots,a_{\ell}+a_1]$ to see that $degr(x')=degr(x)+(m-a_1)(\ell-1)$.  Since corollary \ref{corfitsout} implies $x' \in \mathbf{D}_{\ell,m}$, it follows that $x' \in  \mathbf{N}_{\ell,m}^{d+(m-a_1)(\ell-1)}$
\end{proof}

\begin{definition}
Fix $\ell$, $m$, and $d$ and set $M=m(\ell+1)(\ell)/2$.  Suppose that for each $\lambda \in \mathbf{P}_{\ell-1}^d$ such that $\mathtt{string}(\lambda)=(v_0,\ldots,v_s)$ we have either that $area(v_s)=M-|\lambda|-h(\lambda)$ or else  $area(v_s)<M-|\lambda|-h(\lambda)$ and there is a way to chose an extension $\mathtt{ext}(\lambda)=(w_{1},\ldots,w_e)$ with $w_i \in \mathbf{N}_{\ell,m}^d$  such $area(w_1)=area(v_s)+1$, $area(w_{i+1})=area(w_i)+1$, and $area(w_e)=M-|\lambda|-h(\lambda)$. Further, suppose this can be done in such a way that:
\begin{eqnarray*}
\mathbf{N}_{\ell,m}^d=\bigcup_{\lambda \in P_{\ell-1}^d} \mathtt{ext}(\lambda)
\end{eqnarray*}
where the union is disjoint.  In this case we define $\mathtt{extendable}(\ell,m,d)=\mathtt{true}$ and equal to $\mathtt{false}$ otherwise.
\end{definition}

\begin{example}
Let $\ell=4$, $m=3$, and $d=5$ so that $M=30$.  The five elements of  $\mathbf{P}_4^5$ are shown below along with their diagramatic representations above them.  Below each is an interval of the form $[P,Q/R]$ where $P=w^{\ell-1}(\lambda)=area(f(g(\lambda)))$ and $Q=area(\mathtt{lowest}(f(g(\lambda))))$ and $R=M-d-h(\lambda)$. Further below that are all elements (appearing in black) of $\mathtt{string}(\lambda)$ arranged by value of $area$ (the numerical value is shown on the left).  The red entries appearing at the bottom of a string comprise the elements of $\mathtt{ext}(\lambda)$.  Since the union of the six red elements turns out to be precisely $\mathbf{N}_{4,3}^5$ we see that $\mathtt{extendable}(4,3,5)=\mathtt{true}$.

 \begin{align*}
&&\begin{Young}
&&\cr
&\cr
\end{Young}
&&
\begin{Young}
&&\cr
\cr
\cr
\end{Young}
&&
\begin{Young}
&\cr
&\cr
\cr
\end{Young}
&&
\begin{Young}
&\cr
\cr
\cr
\cr
\end{Young}
&&
\begin{Young}
\cr
\cr
\cr
\cr
\cr
\end{Young}\\
&&\lambda=[1,1,0] &&\lambda=[1,0,2]&&\lambda=[0,2,1]&&\lambda=[0,1,3]&&\lambda=[0,0,5]\\
  &&   [5,23]        &&   [4,22]       &&       [3,20/22]    &&  [3,19/21]    &&  [2,18/20] \\
  &&                  &&                &&                 &&                &&  \\
2 &&                  &&                &&                 &&                && [0,0,0,2,0]\\
3 &&                  &&                &&  [0,0,2,0,1] && [0,0,1,2,0] && [0,1,0,0,2]\\
4 &&                  && [0,0,1,2,1] &&  [0,2,0,2,0] && [0,1,0,1,2] &&[0,3,1,0,0] \\
5 && [0,0,2,2,1]   && [0,2,0,1,2] &&  [0,1,2,0,2] && [0,3,1,0,1] &&[0,1,3,1,0] \\
6 && [0,2,0,2,2]   && [0,3,2,0,1] &&  [0,3,1,2,0]  && [0,2,1,1,3] &&[0,1,1,3,1]\\
\vdots  &&  \vdots          &&  \vdots       &&  \vdots         &&  \vdots         &&    \vdots     \\
17&& [0,2,4,6,5] && [0,2,5,4,6] && [0,3,3,6,5] && [0,3,5,3,6] && [0,3,4,7,3]\\
18&& [0,2,4,6,6] && [0,2,5,7,4] && [0,3,6,3,6] && [0,3,5,7,3] && [0,3,4,4,7]\\
19&& [0,2,4,7,6] && [0,2,5,5,7] && [0,3,6,7,3] && [0,3,4,5,7] &&  \textcolor{red}{ [0,3,5,8,3]}\\
20&& [0,2,4,7,7] && [0,2,5,8,5] && [0,3,4,6,7] &&  \textcolor{red}{[0,3,6,8,3]}&&  \textcolor{red}{ [0,3,4,5,8]} \\
21&& [0,2,4,7,8] && [0,2,5,6,8] && \textcolor{red} {[0,3,6,9,3]} &&  \textcolor{red} {[0,3,4,6,8] }         && \\
22&& [0,2,4,7,9] && [0,2,5,6,9] && \textcolor{red}  {[0,3,4,6,9]  }          &&                 && \\
23&& [0,2,4,7,10] &&                &&                 &&                && \\
\end{align*}

\end{example}

Before stating the next claim it will be convenient to translate Theorem \ref{project} and the last two claims into step coordinates:

\begin{corollary}\label{corproject}
Let $\lambda=[p_0,\ldots,p_{\ell-2}] \in \mathbf{P}_{\ell-1}$ be the partition with $p_i$ parts of size $\ell-1-i$ and let $\mu=[p_1,\ldots,p_{\ell-2}] \in \mathbf{P}_{\ell-2}$ be the partition with $p_i$ parts of size $\ell-1-i$. Further suppose that $|\lambda| < (\ell-2)m$.   Suppose that $x=\mathtt{lowest}(f(g(\mu)))$ and $x'=\mathtt{lowest}(f(g(\mu)))$.  Then if $x=(x_1,\ldots,x_{\ell-1},x_{\ell})$ we have $x'=(p_0,x_1,\ldots,x_{\ell-1},-)$.  
\end{corollary}
\begin{proof}
Translate Theorem 1 into step coordinates.  
\end{proof}

\begin{corollary}\label{cortakeoff}
Suppose $x=(x_0,x_1,\ldots,x_{\ell-1},x_{\ell}) \in \mathbf{N}_{\ell,m}^d$ where $d < (\ell-2)m$.  Then $x'=(x_1,\ldots,x_{\ell-1},-) \in \mathbf{N}_{\ell-1,m}^{d-x_0(\ell-1)}$
\end{corollary}
\begin{proof}
Translate claim \ref{takeoff} into step coordinates.  
\end{proof}

\begin{corollary}\label{coraddon}
Suppose that $x=(x_1,\ldots,x_{\ell-1},x_{\ell}) \in \mathbf{N}_{\ell-1,m}^d$ and suppose $x_0$ is such that $x_0(\ell-1)+d < (\ell-2)m$.  Then $x'=(x_0,x_1,\ldots,x_{\ell-1},-)  \in  \mathbf{N}_{\ell,m}^{d+x_0(\ell-1)}$ 
\end{corollary}
\begin{proof}
Translate claim \ref{addon} into step coordinates.  
\end{proof}

\begin{claim}\label{ext}
Fix $m$ and $\ell^*$ and some $d^*<(\ell^*-1)m$.  If for all $d \leq d^*$ we have $\mathtt{extendable}(\ell^*,m,d)=\mathtt{true}$ then for any $\ell>\ell^*$ we have  $\mathtt{extendable}(\ell,m,d)=\mathtt{true}$ for all $d \leq d^*$.
\end{claim}
\begin{proof}

It suffices to assume that $\ell=\ell^*+1$. Suppose we have already constructed $\mathtt{ext}(\mu)$ for all $\mu \in \mathbf{P}_{\ell^*-1}^{\leq d^*}$.    Now fix $d \leq d^*$ and suppose that $\lambda=[p_0,\ldots,p_{\ell-2}] \in \mathbf{P}_{\ell-1}^d$ and let $\mu=[p_1,\ldots,p_{\ell-2}] \in \mathbf{P}_{\ell^*-1}^{d-p_0(\ell-1)}$ .  

Write  $x=\mathtt{lowest}(f(g((\mu)))$ and   $x'=\mathtt{lowest}(f(g(\lambda)))$.  It follows from corollary \ref{corproject} that if $x=(x_1,\ldots,x_{\ell-1},x_{\ell})$ we have $x'=(p_0,x_1,\ldots,x_{\ell-1},-)$.

  If $\mathtt{ext}(\mu)=\{w_{1},\ldots,w_e\}$ then define $\mathtt{ext}(\lambda)=\{w_{1}',\ldots,w_e'\}$  where for $1 \leq i \leq e$ we have that if $w_i=(y_1,\ldots,y_{\ell-1},y_{\ell})$ then  $w_i'=(p_0,y_1,\ldots,y_{\ell-1},-)$.  Note the following:

\begin{enumerate}
\item We have $area(w_{1}')- area(x')= area(w_{1} )- area(x)=1$
\item We have $area(w_{i+1}')-area(w_i')=area(w_{i+1})-area(w_i)=1$
\item We have \begin{eqnarray*}area(w_e')=area(w_e)+(m-p_0) \ell=\\
m(\ell^*+1)\ell^*/2-|\mu|-h(\mu) +(m-p_0) \ell=  \\
m\ell(\ell-1)/2-(|\lambda|-p_0(\ell-1))-(h(\lambda)-p_0) +(m-p_0) \ell=  \\
m\ell(\ell-1)/2-|\lambda|-h(\lambda)+p_0 \ell +(m-p_0) \ell=  \\
m(\ell+1)(\ell)/2-|\lambda|-h(\lambda)\end{eqnarray*}.
\item $w_i' \in  \mathbf{N}_{\ell,m}^{d}$.
\end{enumerate}
The last of these follows from corollary \ref{coraddon} since $p_0(\ell-1)+degr(w_i)= p_0(\ell-1) + |\mu| =|\lambda|<m(\ell-2)$.  Therefore  taking $\mathtt{ext}(\lambda)=(w_1',\ldots,w_e')$ gives a valid extension.  It remains to show:

\begin{eqnarray*}
\mathbf{N}_{\ell,m}^d=\bigcup_{\lambda \in P_{\ell-1}^d} \mathtt{ext}(\lambda)
\end{eqnarray*}
where the union is disjoint.

Now let $y'=(y_0,y_1,\ldots,y_{\ell-1},y_{\ell})  \in  \mathbf{N}_{\ell,m}^d$ be arbitrary. We must show $y'$ appears in the union above. Claim \ref{takeoff} implies that the element $y=(y_1,\ldots,y_{\ell-1},-) \in  \mathbf{N}_{(\ell-1),m}^{d-y_0(\ell-1)}$.  Therefore $y \in \mathtt{ext}(\mu)$ for some $\mu \in \mathbf{P}_{\ell-2}^{d-y_0(\ell-1)}$ since we are assuming the extensions for all such $\mu$ have already been (validly) constructed.  It follows by the construction in the proof that if $\mu=[p_1,\ldots,p_{\ell-1}]$ then $y' \in \mathtt{ext}(\lambda)$ where $\lambda=[y_0,p_1,\ldots,p_{\ell-1}]$.  Clearly $\lambda \in \mathbf{P}_{\ell-1}^d$  so $y'$ appears in the union above.

Finally we must show the union is disjoint.  Suppose that $\lambda \neq  \rho \in \mathbf{P}_{\ell-1}^d$ and $y' \in \mathtt{ext}(\lambda)$, and $y' \in \mathtt{ext}(\rho)$. If  $y'=(y_0,y_1,\ldots,y_{\ell-1},y_{\ell}) $ and $\lambda=[p_0,\ldots,p_{\ell-1}]$ and $\rho=[q_0,\ldots,q_{\ell-1}]$ we must have $p_0=y_0=q_0$ by construction.  Moreover we have $y=(y_1,\ldots,y_{\ell-1},-) \in \mathtt{ext}([p_1,\ldots,p_{\ell-1}])$ and also $y \in \mathtt{ext}([q_1,\ldots,q_{\ell-1}])$ which implies $p_i=q_i$ for $1 \leq i \leq \ell-1$ since otherwise we would have two distinct partitions in $\mathbf{P}_{\ell-2}^{d-y_0(\ell-1)}$ whose extensions intersect.  

\end{proof}

\section{The case of $\ell=2$}

We now specialize to the case that $\ell=2$.  We begin with the following claim:

\begin{claim}\label{maxarea}
Suppose that $x \in \mathbf{D}_{2,m}$.  Then $area(x)+2degr(x) \leq 3m$.  Moreover, for any $d \in [0,m]$ there is some $x \in \mathbf{D}_{2,m}$ with $degr(x)=d$ and $area(x)+2degr(x) = 3m$.
\end{claim}
\begin{proof}
To show the first half, suppose that $x=[0,a_1,a_2]$.  We have that:
\begin{eqnarray*}
degr(x)=\alpha(a_1,a_2)-\alpha_0(a_2)\\
=max(a_2-a_1,a_1-a_2-1)-max(0,a_2-m)\\
\end{eqnarray*}
Thus if :
\begin{itemize}
\item $a_2 \geq m$ then $degr(x)=m-a_1$ so $area(x)+2degr(x)=a_2-a_1+2m \leq 3m$ since  $a_1-a_2 \geq -m$.
\item $a_2 < m$ and $a_2 \geq a_1$ then  $degr(x)=a_2-a_1-(0)$ so that $area(x)+2degr(x)=3a_2-a_1 < 3m$ since $a_2 < m$.
\item $a_2<a_1$ then $degr(x)=a_1-a_2-1-(0)$ since $a_2 <a_1 \leq m$ so that $area(x)+2degr(x)=3a_1-a_2-2 < 3m$ since $a_1 \leq m$.
\end{itemize}
To see the second part, simply take $x=[0,m-d,2m-d]$.
\end{proof}

\begin{claim}\label{equalstring}
If $d<m$ then $\mathbf{D}_{2,m}^d= \mathtt{string}([d])$ where $[d]$ is the partition composed of $d$ parts of size $1$ (i.e., the unique element of $\mathbf{P}_{1}^d$).
\end{claim}

\begin{proof}
First we show that $\mathbf{C}_{2,m}^d=\mathbf{D}_{2,m}^d$.  If this were not true there would have to be some disconnected element $x \in \mathbf{D}_{2,m}^d$ and therefore some unleftable element $y=[0,a_1,a_2] \in \mathbf{D}_{2,m}^d$.  If $pair(y)=0$ this would imply that $a_1-a_2>m+1$ which is impossible as $a_1 \leq m$.   If $pair(y)=1$ this would imply that $a_2-a_2>m+1$ which is also impossible.  

Now, since $d<m=(\ell-1)m$ we may apply claim \ref{connected} to see
\begin{eqnarray*}
\mathbf{D}_{2,m}^d=\mathbf{C}_{2,m}^d=\bigcup_{\lambda \in \mathbf{P}_{1}^d} \mathtt{string}(\lambda)
\end{eqnarray*}
but $[d]$ is the only element in $\mathbf{P}_{1}^d$ so the claim follows.

\end{proof}

\begin{claim}\label{stringsym}
Writing $M=3m$ we have
\end{claim}
\begin{eqnarray*}
\sum_{x \in \mathbf{D}_{2,m}^d} q^{area(x)}t^{M-d-area(x)}=\\ q^dt^{M-2d}+q^{d+1}t^{M-2d-1}+\cdots+q^{M-2d-1}t^{d+1}+q^{M-2d}t^{d}
\end{eqnarray*}
\begin{proof}

If $d<m$ let $b = \mathtt{lowest}(f(g([d])))$.  By  claim \ref{equalstring} we know that $b$ is the unique element of maximal area in $\mathbf{D}_{2,m}^d$ and by claim \ref{maxarea} that we know that such an element has area  of $M-2d$.  Since $area(f(g([d])))=d$ the claim follows for $d<m$.
On the other hand, suppose $x=[0,a_1,a_2] \in \mathbf{D}_{2,m}^d$ for some $d \geq m$.  If:
\begin{itemize}
\item $a_2 \geq m$ then $degr(x)=m-a_1$ which forces $a_1=0$ which in turn forces $a_2 \leq m$.  Thus we must have $x=[0,0,m]$.
\item $a_2 < m$ and $a_2 \geq a_1$ then  $degr(x)=a_2-a_1-(0)<m$ which is a contradiction.
\item $a_2<a_1$ then $degr(x)=a_1-a_2-1-(0)<m$ which is a contradiction.
\end{itemize}
Therefore $\mathbf{D}_{2,m}^m=\{[0,0,m]\}$ and $\mathbf{D}_{2,m}^d=\emptyset$ for $d>m$ which shows the claim is true for all $d \geq m$ since $area([0,0,m])=m$.  

\end{proof}

\section{The Conjecture and Evidence of its Truth}

We now consider rational Dyck paths $x$ of slope $n/(\ell+1)$ where $\ell+1$ and $n$ are relatively prime (the parameters $(\ell+1,n)$ correspond to the paramters $(s,r)$ appearing in the introduction). They are equivalent to sequences $x=(x_0,x_1,\ldots,x_{\ell})$ such that  $x_0+\ldots+x_i < n(i+1)/(\ell+1)$ for $0 \leq i \leq \ell -1$ and $x_0+\ldots+x_{\ell}=n$.  We define $\mathbf{D}_{n/(\ell+1)}$ to be the set of such sequences.  Note that if $n=(\ell+1)m+1$ for some $m$ then $\mathbf{D}_{n/(\ell+1)}=\mathbf{D}_{\ell,m}$ (more precisely, there is a bijection $\mathbf{D}_{\ell,m} \rightarrow \mathbf{D}_{n/(\ell+1)}$ given by adding $1$ to the last step coordinate).  We make the following definitions:

\begin{eqnarray*}
M=\sum_{i=0}^{\ell-1} \lfloor \frac{n}{\ell+1}(i+1) \rfloor\\
area(x)=M-(x_0(\ell)+x_1(\ell-1)+\cdots+x_{\ell-1}(1)+x_{\ell}(0))\\
\end{eqnarray*}

Note that if $n=(\ell+1)m+1$ for some $m$ then $M=m(\ell)(\ell+1)/2$ and the definition of $area(x)$ coincides with that given earlier.  Next, we define the following:

\begin{eqnarray*}
\beta_{ij}(x)=(x_i+\cdots+x_j)-\frac{n}{\ell+1}(j-i+1)\\
\gamma_{ij}(x)=min(x_{i-1},\lfloor | \beta_{ij}(x) | \rfloor) \text{   if   } \beta_{ij}(x)<0 \\
\gamma_{ij}(x)=min(x_{i},\lfloor | \beta_{ij}(x) | \rfloor) \text{   if   } \beta_{ij}(x)>0\\
degr(x)=\sum_{1 \leq i \leq j < \ell}\gamma_{ij}(x)\\
\end{eqnarray*}
Note that if $n=(\ell+1)m+1$  then $\gamma_{ij}(x)=\delta_{ij}^+(x)+\delta_{ij}^-(x)$ and so the definition of  $degr(x)$ agrees with that given earlier.  Next, we define:

\begin{eqnarray*}
under(x)=\min\limits_{i \in[0,\ell)} \left( \frac{n}{\ell+1}(i+1)-(x_0+\cdots+x_i) \right)\\
\mathbf{T}_{n/(\ell+1)}=\{x: x \in \mathbf{D}_{n/(\ell+1)}\text{,   }under(x)=1/(\ell+1)\}\\
\mathbf{D}_{n/(\ell+1)}^d=\{x: x \in \mathbf{D}_{n/(\ell+1)}\text{,   }degr(x)=d\}\\
\mathbf{T}_{n/(\ell+1)}^d=\{x: x \in \mathbf{T}_{n/(\ell+1)}\text{,   }degr(x)=d\}\\
\end{eqnarray*}

Note that if $n=(\ell+1)m+1$ and $x=(x_0,x_1,\ldots,x_{\ell}) \in  \mathbf{D}_{n/(\ell+1)}= \mathbf{D}_{\ell,m}$ then $under(x)=1/(\ell+1)$ is equivalent to the statement that $x_0=m$ which is equivalent to the statement that $x \in \mathbf{T}_{\ell,m}$.  Thus $\mathbf{T}_{n/(\ell+1)}=\mathbf{T}_{\ell,m}$.  Similarly, $\mathbf{T}_{n/(\ell+1)}^d=\mathbf{T}_{\ell,m}^d$.

\greybox{
\begin{conjecture}\label{con}
Define the degree $M-d$ part of the rational $q,t$-Catalan polynomial of slope $n/(\ell+1)$ to be:
\begin{eqnarray*}
\mathcal{C}_{n/(\ell+1)}^d:= \sum_{x \in \mathbf{D}_{n/(\ell+1)}^d} q^{area(x)}t^{M-d-area(x)}
\end{eqnarray*}
Further, let us define:
\begin{eqnarray*}
sym(a,b)= \frac{q^{b+1} t^a-q^a t^{b+1}}{q-t}
\end{eqnarray*}
Then we have that:
\begin{eqnarray*}
\mathcal{C}_{n/(\ell+1)}^d= \sum_{x \in \mathbf{T}_{n/(\ell+1)}^d}sym\left(area(x),M-d-area(x)\right)
\end{eqnarray*}
\end{conjecture}
}

\begin{remark}
This function first appears as definition 21 in section 7 of \cite{LW09}. More precisely, we have:
\begin{eqnarray*}
\sum_{d \geq 0} \mathcal{C}_{n/(\ell+1)}^d=C_{n,\ell+1,1}(q, t)
\end{eqnarray*}
where the left hand side is our notation and the right hand side is the notation of \cite{LW09}.  More precisely our statistic $area$ corresponds to their statistic $\text{area}^c$ and their $h^+_{r/s}$ (sometimes referred to as $dinv$) is equal to the value of our $M-degr-area$.  
\end{remark}

\begin{remark}
In the case that $n=(\ell+1)m+1$ we write $\mathcal{C}_{n/(\ell+1)}^d=\mathcal{C}_{\ell,m}^d$.  
\end{remark}

\begin{lemma}\label{1}
Set $d^*=20$ and fix $m \geq 1$.  Let $\ell^*$ be minimal such that $m(\ell^*-1)>d^*$.  Then for all $d \leq d^*$ and all $\lambda \in \mathbf{P}_{\ell^*-1}^{d}$ we have:
\begin{eqnarray*}
area(\mathtt{lowest}(f \circ g(\lambda))) \geq \frac{(\ell^*+1)(\ell^*)}{2}m-|\lambda|-h(\lambda)
\end{eqnarray*}
where $\mathtt{lowest}$ is defined with respect to the choice of $m$.
\end{lemma}
\begin{proof}
First suppose that $m>20$.  In this case we have $\ell^*=2$.   According to claim \ref{stringsym} the maximal area of an element of $\mathbf{D}_{2,m}^d$ is $3m-2d$.  According to claim \ref{equalstring} the element of maximum area in $\mathbf{D}_{2,m}^d$ is $\mathtt{lowest}(f \circ g([d])))$ so we conclude this element has area of $3m-2d$.  On the other hand, $3m-|[d]|-h([d])=3m-2d$ as well, so the claim holds.   Therefore it remains only to check the lemma in the case that $m \in [1,20]$.  But in this case the parameters $m$ and $d$ are both bounded and $\ell^*$ is a function of $m$.  Thus there are only a finite number of values of $(\ell^*,m,d)$ to check.  This is done in appendix \ref{comp}.  
\end{proof}

\begin{lemma}\label{3}
Set $d^*=20$ and fix $m \geq 1$.  Let $\ell^*$ be minimal such that $(\ell^*-1)m>d^*$.  Then for all $d \leq d^*$ and $\ell \leq \ell^*$ we have:
\begin{eqnarray}\label{toshow}
\mathcal{C}_{\ell,m}^d= \sum_{x \in \mathbf{T}_{\ell,m}^d}sym\left(area(x),\frac{(\ell+1)\ell}{2}m-d-area(x)\right)
\end{eqnarray}

\end{lemma}
\begin{proof}
First suppose that $m>20$.  In this case we have $\ell^*=2$.  We need to consider the cases of $\ell=1$ and $\ell=2$.  If $\ell=2$ then $M=m(\ell+1)(\ell)/2=3m$. For $d \leq d^*$  we have
\begin{eqnarray*}
\mathcal{C}_{2,m}^d= \sum_{x \in \mathbf{D}_{2,m}^d} q^{area(x)}t^{M-d-area(x)}=\\
q^dt^{M-2d}+q^{d+1}t^{M-2d-1}+\cdots+q^{M-2d-1}t^{d+1}+q^{M-2d}t^{d}=\\
\sum_{x \in \mathbf{T}_{2,m}^d}sym\left(area(x),M-d-area(x)\right)
\end{eqnarray*}
The first equality is the definition.  The second equality is claim \ref{stringsym}.  The third equality follows from the fact that $\mathbf{T}_{2,m}^d=\{[0,0,d]\}$ for $d < m$ and the fact that $area([0,0,d])=d$.  Therefore equation \ref{toshow} holds for $\ell=2$ for all $d \leq d^*$  for $m>20$.  

On the other hand, if $\ell=1$ then $M=m$ and we have $\mathbf{D}_{1,m}=\{[0,0],[0,1],\ldots,[0,m]\}$. Since $degr([0,i])=0$ and $area([0,i])=i$ it follows that $\mathcal{C}_{1,m}^d=0$ if $d>0$ and
\begin{eqnarray*}
\mathcal{C}_{1,m}^0= q^0t^m+q^1t^{m-1}+\cdots+q^{m-1}t^1+q^mt^0=sym(0,m)
\end{eqnarray*}
Since $\mathbf{T}_{1,m}=\{[0,0]\}$ and $area([0,0])=0=degr([0,0])$ it follows that   equation \ref{toshow} holds for  $\ell=2$ for all $d \leq d^*$  for $m>20$.  

This proves the lemma in the case that $m>20$.  Therefore it remains only to check the lemma in the case that $m \in [1,20]$.  But in this case the parameters $m$ and $d$ are both bounded and $\ell^*$ is also finite by construction.  Thus there are only a finite number of values of $(\ell,m,d)$ to check.  This is done in appendix \ref{comp}.  
\end{proof}

\begin{theorem}
Suppose that $n=(\ell+1)m+1$ for some $m \geq 1$. Then Conjecture \ref{con} holds for all $d \leq 20$.
\end{theorem}

\begin{proof}
Let $d^*=20$ and choose $\ell^*$ minimal such that $(\ell^*-1)m>d^*$. First let us define:
\begin{eqnarray*}
str(a,b,c)=q^a t^{c-a}+ q^{a+1} t^{c-a-1}+\cdots+ q^{b-1} t^{c-b+1} + q^b t^{c-b}
\end{eqnarray*}

\begin{remark} \label{rem}
Note that if $a < b$ then $str(a,b,a+b)=sym(a,b)$.  Moreover, note that if for some index set $I$ we have $a_i < c/2 < b_i$ for $i \in I$ then for any permutation, $\sigma$, of the index set we have:
\begin{eqnarray*}
\sum_{i \in I} str(a_i,b_i,c)= \sum_{i \in I} str(a_i,b_{\sigma_i},c)
\end{eqnarray*}
\end{remark}

\begin{remark} \label{mer}
If  $\lambda \in \mathbf{P}_{\ell^-1}^d$for some $d < (\ell-1)m$ then we have $2h(\lambda)  < \frac{(\ell+1)\ell}{2}m-|\lambda|$ and $2w^{\ell-1}(\lambda)  < \frac{(\ell+1)\ell}{2}m-|\lambda|$.  This follows from the fact that
\begin{eqnarray*}
3|\lambda|=3d< 3(\ell-1)m = \\
((\ell-2) +(\ell-1) +(\ell))m \leq \\
 (1+2+\cdots+\ell^*)m = \\
 \frac{(\ell+1)\ell}{2}m \\
\end{eqnarray*}
along with the fact that  $h(\lambda) \leq |\lambda|$  and  $w^{\ell-1}(\lambda) \leq |\lambda|$.
\end{remark}

Now fix some $d \leq d^*$. We wish to evaluate $\mathtt{extendable}(\ell^*,m,d)$.

Note that since $d \leq d^* < (\ell^*-1)m$ it follows by claims \ref{g} and \ref{f} that the function $f \circ g$ is  a bijection $\mathbf{P}_{\ell^*-1}^d \rightarrow \mathbf{T}_{\ell^*,m}^d$ such that if $f \circ g(\lambda)=x$ then $area(x)=w^{\ell^*-1}(\lambda)$.  Hence by Lemma \ref{3} we have:
\begin{eqnarray*}
\mathcal{C}_{\ell^*,m}^d= \sum_{\lambda \in \mathbf{P}_{\ell^*-1}^d} sym\left(w^{\ell^*-1}(\lambda),\frac{(\ell^*+1)\ell^*}{2}-|\lambda|-w^{\ell^*-1}(\lambda) \right)
\end{eqnarray*}

By remark \ref{mer} we have $w^{\ell^*-1}(\lambda)  < \frac{(\ell^*+1)\ell^*}{2}m-|\lambda|-w^{\ell^*-1}(\lambda) $ so the first part of remark \ref{rem} implies that we may instead write:

\begin{eqnarray*}
\mathcal{C}_{\ell^*,m}^d= \sum_{\lambda \in \mathbf{P}_{\ell^*-1}^d} str\left(w^{\ell^*-1}(\lambda),\frac{(\ell^*+1)\ell^*}{2}-|\lambda|-w^{\ell^*-1}(\lambda), \frac{(\ell^*+1)\ell^*}{2}-|\lambda|\right)
\end{eqnarray*}
Since $d < (\ell^*-1)m$ we may combine claims \ref{g} and \ref{i} and consider the involution $g^{-1}\circ \iota \circ g$ on $\mathbf{P}_{\ell^*-1}^d$ which has the property that if $g^{-1}( \iota ( g (\lambda)))=\mu$ then $w^{\ell^*-1}(\lambda)=h(\mu)$.  From this and the inequalities of remark \ref{mer} and the second part of remark \ref{rem} it follows that:
\begin{eqnarray*}
\mathcal{C}_{\ell^*,m}^d= \sum_{\lambda \in \mathbf{P}_{\ell^*-1}^d} str\left(w^{\ell^*-1}(\lambda),\frac{(\ell^*+1)\ell^*}{2}-|\lambda|-h(\lambda), \frac{(\ell^*+1)\ell^*}{2}-|\lambda|\right)
\end{eqnarray*}
This implies that there is a way to write
\begin{eqnarray*}
\mathbf{D}_{\ell^*,m}^d= \bigcup_{\lambda \in \mathbf{P}_{\ell^*-1}^d} string(\lambda)
\end{eqnarray*}
 where for each $\lambda$ we have that if $string(\lambda)=(u_0,\ldots,u_e)$ then $area(u_0)=w^{\ell-1}(\lambda)$, while $area(u_{i+1})=area(u_i)+1$ and $area(u_e)=\frac{(\ell^*+1)\ell^*}{2}-|\lambda|-h(\lambda)$.
Now by claim \ref{connected} we have
\begin{eqnarray*}
\mathbf{C}_{\ell^*,m}^d= \bigcup_{\lambda \in \mathbf{P}_{\ell^*-1}^d} \mathtt{string}(\lambda)
\end{eqnarray*}
 and since Lemma \ref{1} states that $area(\mathtt{lowest}(f \circ g(\lambda))) \geq \frac{(\ell^*+1)(\ell^*)}{2}-|\lambda|-h(\lambda)$ for all $\lambda \in \mathbf{P}_{\ell^*-1}^d$  we may assume that $string(\lambda)=(v_0,\ldots,v_s,w_{s+1},\ldots,w_e)$ where $\mathtt{string}(\lambda)=(v_0,\ldots,v_s)$ and all $w_i \in \mathbf{N}_{\ell^*,m}^d$.  
Thus $\mathtt{extendable}(\ell^*,m,d)=\mathtt{true}$.  

Since we have shown that $\mathtt{extendable}(\ell^*,m,d)=\mathtt{true}$ for any $d \leq d^*$ we see from claim \ref{ext} that for all $\ell \geq \ell^*$ and $d \leq d^*$ we have $\mathtt{extendable}(\ell,m,d)=\mathtt{true}$.  From this it follows that for any $\ell \geq \ell^*$ and $d \leq d^*$ we have:

\begin{eqnarray*}
\mathcal{C}_{\ell,m}^d= \sum_{\lambda \in \mathbf{P}_{\ell-1}^d} str\left(w^{\ell-1}(\lambda),\frac{(\ell+1)\ell}{2}-|\lambda|-h(\lambda), \frac{(\ell+1)\ell}{2}-|\lambda|\right)   \\
\end{eqnarray*}
Now $d < (\ell^*-1)m \leq (\ell-1)m$ so $g^{-1}\circ \iota \circ g$ is an involution on $\mathbf{P}_{\ell-1}^d$ such that if $g^{-1}( \iota ( g (\lambda)))=\mu$ then $w^{\ell-1}(\lambda)=h(\mu)$ so again applying the second part of remark \ref{rem} in light of the inequalities of remark \ref{mer} we have:
\begin{eqnarray*}
\mathcal{C}_{\ell,m}^d= \sum_{\lambda \in \mathbf{P}_{\ell-1}^d} str\left(w^{\ell-1}(\lambda),\frac{(\ell+1)\ell}{2}-|\lambda|-w^{\ell-1}(\lambda), \frac{(\ell+1)\ell}{2}-|\lambda|\right)\\
= \sum_{\lambda \in \mathbf{P}_{\ell-1}^d} sym\left(w^{\ell-1}(\lambda),\frac{(\ell+1)\ell}{2}-|\lambda|-w^{\ell-1}(\lambda)\right)\\
\end{eqnarray*}
Finally  using the bijection $f \circ g: \mathbf{P}_{\ell-1}^d \rightarrow \mathbf{T}_{\ell,m}^d$ we have:
\begin{eqnarray*}
\mathcal{C}_{\ell,m}^d= \sum_{x \in \mathbf{T}_{\ell,m}^d} sym\left(area(x),\frac{(\ell+1)\ell}{2}-d-area(x)\right)\\
\end{eqnarray*}
We have shown this equation holds for all $\ell \geq \ell^*$ and $d \leq d^*$.  Since this equation also holds for all $\ell \leq \ell^*$ and $d \leq d^*$ by Lemma \ref{3} this completes the proof.

\end{proof}

\begin{example}\label{conex}
The elements of $\mathbf{D}_{13/8}^{19}$ are shown on the next page diagramatically.   For each $x \in \mathbf{D}_{13/8}^{19}$ the value of $area(x)$ is written inside the diagram of $x$.  This is equal to the number of unshaded boxes.   The  maximum possible area for any path is given by  $M=1+3+4+6+8+9+11=42$.  By the definition of $\mathcal{C}_{13/8}^{19}$ we may compute that:
\begin{eqnarray*}
\mathcal{C}_{13/8}^{19}=q^8t^{15}+3q^9t^{14}+6q^{10}t^{13}+8q^{11}t^{12}+8q^{12}t^{11}+6q^{13}t^{10}+3q^{14}t^9+q^{15}t^8
\end{eqnarray*}
The elements of $\mathbf{T}_{13/8}^{19}$ are those whose paths come within $1/8$ of the bounding diagonal.  These are specified above by a red circle where this occurs.  Therefore the conjecture states that:
\begin{align*}
\mathcal{C}_{13/8}^{19}&=\\
q^8t^{15}+q^9t^{14}+q^{10}t^{13}+q^{11}t^{12}&+q^{12}t^{11}+q^{13}t^{10}+q^{14}t^9+q^{15}t^8\\
+2q^9t^{14}+2q^{10}t^{13}+2q^{11}t^{12}&+2q^{12}t^{11}+2q^{13}t^{10}+2q^{14}t^9\\
+4q^{10}t^{13}+4q^{11}t^{12}&+4q^{12}t^{11}+4q^{13}t^{10}\\
+4q^{11}t^{12}&+4q^{12}t^{11}\\
+&4(0)\\
-2q^{11}t^{12}&-2q^{12}t^{11}\\
-q^{10}t^{13}-q^{11}t^{12}&-q^{12}t^{11}-q^{13}t^{10}\\
\end{align*}
Comparing the two expressions shows that the conjecture is correct in this case.

\newpage
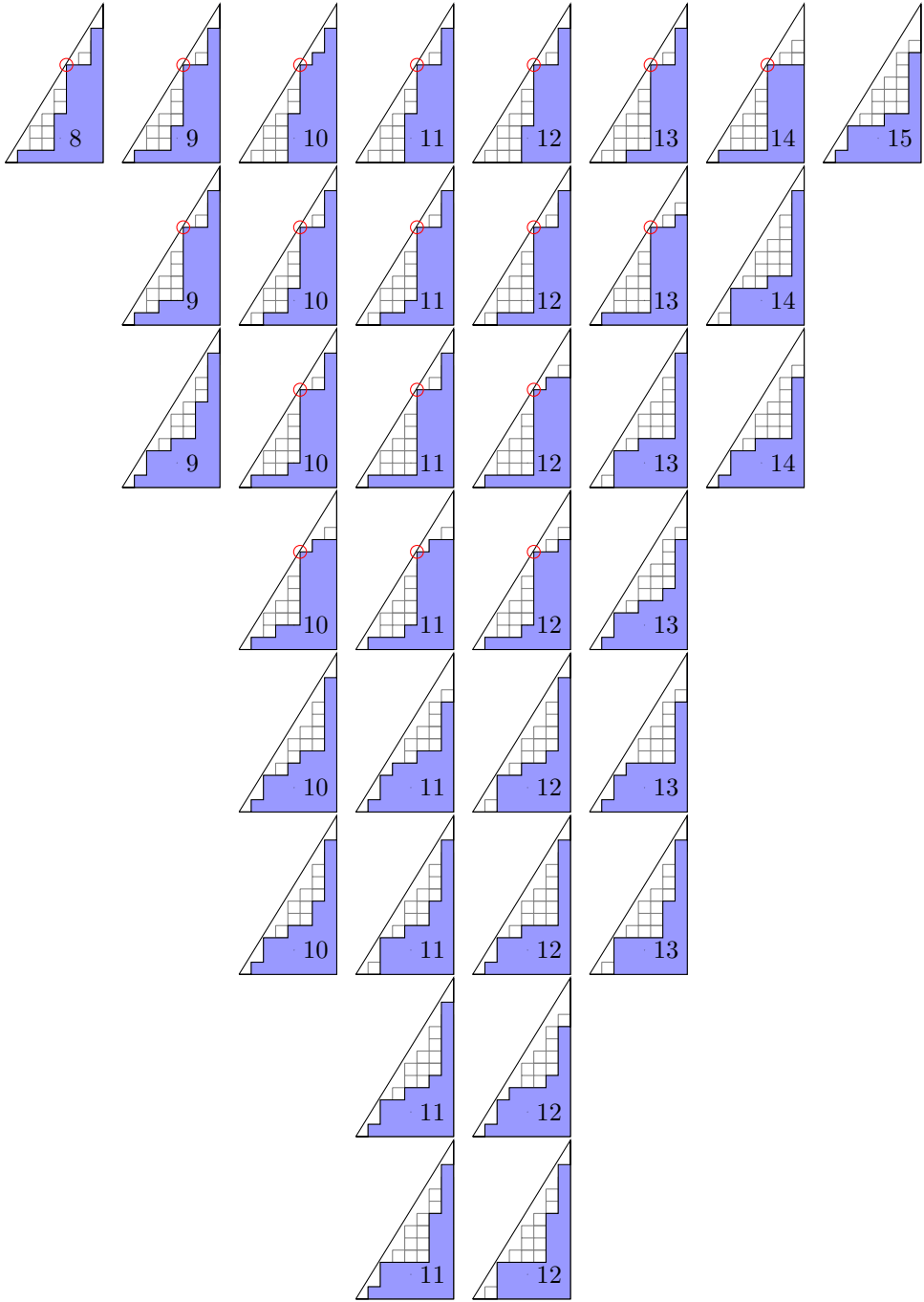
\begin{multicols}{8}
\noindent \begin{tikzpicture}[scale=0.17] 
 \draw (0, 0)--(8, 13)--(8,0)--(0,0); 
 \draw[step=1cm,gray] (1,0) grid (8,1); 
 \draw[step=1cm,gray] (2,1) grid (8,3); 
 \draw[step=1cm,gray] (3,3) grid (8,4); 
 \draw[step=1cm,gray] (4,4) grid (8,6); 
 \draw[step=1cm,gray] (5,6) grid (8,7); 
 \draw[step=1cm,gray] (6,7) grid (8,9); 
 \draw[step=1cm,gray] (7,9) grid (8,10); 
 \filldraw[fill=blue!40!white, draw=black](0,0)--(0,0)--(1,0)--(1,1)--(2,1)--(2,1)--(3,1)--(3,1)--(4,1)--(4,4)--(5,4)--(5,8)--(6,8)--(6,8)--(7,8)--(7,11)--(8,11)--(8,13)--(8,0)--(0,0);
 \filldraw[black] (4.5,2) circle (0.1pt) node[anchor=west]{$8$}; 
  \draw[red] (5,8) circle (15pt); 
\end{tikzpicture}
\vfill\null
\columnbreak
\noindent
\begin{tikzpicture}[scale=0.17] 
 \draw (0, 0)--(8, 13)--(8,0)--(0,0); 
 \draw[step=1cm,gray] (1,0) grid (8,1); 
 \draw[step=1cm,gray] (2,1) grid (8,3); 
 \draw[step=1cm,gray] (3,3) grid (8,4); 
 \draw[step=1cm,gray] (4,4) grid (8,6); 
 \draw[step=1cm,gray] (5,6) grid (8,7); 
 \draw[step=1cm,gray] (6,7) grid (8,9); 
 \draw[step=1cm,gray] (7,9) grid (8,10); 
 \filldraw[fill=blue!40!white, draw=black](0,0)--(0,0)--(1,0)--(1,1)--(2,1)--(2,1)--(3,1)--(3,1)--(4,1)--(4,3)--(5,3)--(5,8)--(6,8)--(6,8)--(7,8)--(7,11)--(8,11)--(8,13)--(8,0)--(0,0);
 \filldraw[black] (4.5,2) circle (0.1pt) node[anchor=west]{9}; 
  \draw[red] (5,8) circle (15pt); 
\end{tikzpicture}
\begin{tikzpicture}[scale=0.17] 
 \draw (0, 0)--(8, 13)--(8,0)--(0,0); 
 \draw[step=1cm,gray] (1,0) grid (8,1); 
 \draw[step=1cm,gray] (2,1) grid (8,3); 
 \draw[step=1cm,gray] (3,3) grid (8,4); 
 \draw[step=1cm,gray] (4,4) grid (8,6); 
 \draw[step=1cm,gray] (5,6) grid (8,7); 
 \draw[step=1cm,gray] (6,7) grid (8,9); 
 \draw[step=1cm,gray] (7,9) grid (8,10); 
 \filldraw[fill=blue!40!white, draw=black](0,0)--(0,0)--(1,0)--(1,1)--(2,1)--(2,1)--(3,1)--(3,2)--(4,2)--(4,2)--(5,2)--(5,8)--(6,8)--(6,8)--(7,8)--(7,11)--(8,11)--(8,13)--(8,0)--(0,0);
 \filldraw[black] (4.5,2) circle (0.1pt) node[anchor=west]{9}; 
  \draw[red] (5,8) circle (15pt); 
\end{tikzpicture}
\begin{tikzpicture}[scale=0.17] 
 \draw (0, 0)--(8, 13)--(8,0)--(0,0); 
 \draw[step=1cm,gray] (1,0) grid (8,1); 
 \draw[step=1cm,gray] (2,1) grid (8,3); 
 \draw[step=1cm,gray] (3,3) grid (8,4); 
 \draw[step=1cm,gray] (4,4) grid (8,6); 
 \draw[step=1cm,gray] (5,6) grid (8,7); 
 \draw[step=1cm,gray] (6,7) grid (8,9); 
 \draw[step=1cm,gray] (7,9) grid (8,10); 
 \filldraw[fill=blue!40!white, draw=black](0,0)--(0,0)--(1,0)--(1,1)--(2,1)--(2,3)--(3,3)--(3,3)--(4,3)--(4,4)--(5,4)--(5,4)--(6,4)--(6,7)--(7,7)--(7,11)--(8,11)--(8,13)--(8,0)--(0,0);
 \filldraw[black] (4.5,2) circle (0.1pt) node[anchor=west]{9}; 
  \end{tikzpicture}
\vfill\null
\columnbreak
\noindent
\begin{tikzpicture}[scale=0.17] 
 \draw (0, 0)--(8, 13)--(8,0)--(0,0); 
 \draw[step=1cm,gray] (1,0) grid (8,1); 
 \draw[step=1cm,gray] (2,1) grid (8,3); 
 \draw[step=1cm,gray] (3,3) grid (8,4); 
 \draw[step=1cm,gray] (4,4) grid (8,6); 
 \draw[step=1cm,gray] (5,6) grid (8,7); 
 \draw[step=1cm,gray] (6,7) grid (8,9); 
 \draw[step=1cm,gray] (7,9) grid (8,10); 
 \filldraw[fill=blue!40!white, draw=black](0,0)--(0,0)--(1,0)--(1,0)--(2,0)--(2,0)--(3,0)--(3,0)--(4,0)--(4,4)--(5,4)--(5,8)--(6,8)--(6,9)--(7,9)--(7,11)--(8,11)--(8,13)--(8,0)--(0,0);
 \filldraw[black] (4.5,2) circle (0.1pt) node[anchor=west]{10}; 
  \draw[red] (5,8) circle (15pt); 
\end{tikzpicture}
\begin{tikzpicture}[scale=0.17] 
 \draw (0, 0)--(8, 13)--(8,0)--(0,0); 
 \draw[step=1cm,gray] (1,0) grid (8,1); 
 \draw[step=1cm,gray] (2,1) grid (8,3); 
 \draw[step=1cm,gray] (3,3) grid (8,4); 
 \draw[step=1cm,gray] (4,4) grid (8,6); 
 \draw[step=1cm,gray] (5,6) grid (8,7); 
 \draw[step=1cm,gray] (6,7) grid (8,9); 
 \draw[step=1cm,gray] (7,9) grid (8,10); 
 \filldraw[fill=blue!40!white, draw=black](0,0)--(0,0)--(1,0)--(1,0)--(2,0)--(2,1)--(3,1)--(3,1)--(4,1)--(4,3)--(5,3)--(5,8)--(6,8)--(6,8)--(7,8)--(7,11)--(8,11)--(8,13)--(8,0)--(0,0);
 \filldraw[black] (4.5,2) circle (0.1pt) node[anchor=west]{10}; 
  \draw[red] (5,8) circle (15pt); 
\end{tikzpicture}
\begin{tikzpicture}[scale=0.17] 
 \draw (0, 0)--(8, 13)--(8,0)--(0,0); 
 \draw[step=1cm,gray] (1,0) grid (8,1); 
 \draw[step=1cm,gray] (2,1) grid (8,3); 
 \draw[step=1cm,gray] (3,3) grid (8,4); 
 \draw[step=1cm,gray] (4,4) grid (8,6); 
 \draw[step=1cm,gray] (5,6) grid (8,7); 
 \draw[step=1cm,gray] (6,7) grid (8,9); 
 \draw[step=1cm,gray] (7,9) grid (8,10); 
 \filldraw[fill=blue!40!white, draw=black](0,0)--(0,0)--(1,0)--(1,1)--(2,1)--(2,1)--(3,1)--(3,1)--(4,1)--(4,2)--(5,2)--(5,8)--(6,8)--(6,8)--(7,8)--(7,11)--(8,11)--(8,13)--(8,0)--(0,0);
 \filldraw[black] (4.5,2) circle (0.1pt) node[anchor=west]{10}; 
  \draw[red] (5,8) circle (15pt); 
\end{tikzpicture}
\begin{tikzpicture}[scale=0.17] 
 \draw (0, 0)--(8, 13)--(8,0)--(0,0); 
 \draw[step=1cm,gray] (1,0) grid (8,1); 
 \draw[step=1cm,gray] (2,1) grid (8,3); 
 \draw[step=1cm,gray] (3,3) grid (8,4); 
 \draw[step=1cm,gray] (4,4) grid (8,6); 
 \draw[step=1cm,gray] (5,6) grid (8,7); 
 \draw[step=1cm,gray] (6,7) grid (8,9); 
 \draw[step=1cm,gray] (7,9) grid (8,10); 
 \filldraw[fill=blue!40!white, draw=black](0,0)--(0,0)--(1,0)--(1,1)--(2,1)--(2,1)--(3,1)--(3,2)--(4,2)--(4,2)--(5,2)--(5,8)--(6,8)--(6,9)--(7,9)--(7,9)--(8,9)--(8,13)--(8,0)--(0,0);
 \filldraw[black] (4.5,2) circle (0.1pt) node[anchor=west]{10}; 
  \draw[red] (5,8) circle (15pt); 
\end{tikzpicture}
\begin{tikzpicture}[scale=0.17] 
 \draw (0, 0)--(8, 13)--(8,0)--(0,0); 
 \draw[step=1cm,gray] (1,0) grid (8,1); 
 \draw[step=1cm,gray] (2,1) grid (8,3); 
 \draw[step=1cm,gray] (3,3) grid (8,4); 
 \draw[step=1cm,gray] (4,4) grid (8,6); 
 \draw[step=1cm,gray] (5,6) grid (8,7); 
 \draw[step=1cm,gray] (6,7) grid (8,9); 
 \draw[step=1cm,gray] (7,9) grid (8,10); 
 \filldraw[fill=blue!40!white, draw=black](0,0)--(0,0)--(1,0)--(1,1)--(2,1)--(2,3)--(3,3)--(3,3)--(4,3)--(4,4)--(5,4)--(5,5)--(6,5)--(6,5)--(7,5)--(7,11)--(8,11)--(8,13)--(8,0)--(0,0);
 \filldraw[black] (4.5,2) circle (0.1pt) node[anchor=west]{10}; 
  \end{tikzpicture}
\begin{tikzpicture}[scale=0.17] 
 \draw (0, 0)--(8, 13)--(8,0)--(0,0); 
 \draw[step=1cm,gray] (1,0) grid (8,1); 
 \draw[step=1cm,gray] (2,1) grid (8,3); 
 \draw[step=1cm,gray] (3,3) grid (8,4); 
 \draw[step=1cm,gray] (4,4) grid (8,6); 
 \draw[step=1cm,gray] (5,6) grid (8,7); 
 \draw[step=1cm,gray] (6,7) grid (8,9); 
 \draw[step=1cm,gray] (7,9) grid (8,10); 
 \filldraw[fill=blue!40!white, draw=black](0,0)--(0,0)--(1,0)--(1,1)--(2,1)--(2,3)--(3,3)--(3,3)--(4,3)--(4,4)--(5,4)--(5,4)--(6,4)--(6,6)--(7,6)--(7,11)--(8,11)--(8,13)--(8,0)--(0,0);
 \filldraw[black] (4.5,2) circle (0.1pt) node[anchor=west]{10}; 
  \end{tikzpicture}
\vfill\null
\columnbreak
\noindent
\begin{tikzpicture}[scale=0.17] 
 \draw (0, 0)--(8, 13)--(8,0)--(0,0); 
 \draw[step=1cm,gray] (1,0) grid (8,1); 
 \draw[step=1cm,gray] (2,1) grid (8,3); 
 \draw[step=1cm,gray] (3,3) grid (8,4); 
 \draw[step=1cm,gray] (4,4) grid (8,6); 
 \draw[step=1cm,gray] (5,6) grid (8,7); 
 \draw[step=1cm,gray] (6,7) grid (8,9); 
 \draw[step=1cm,gray] (7,9) grid (8,10); 
 \filldraw[fill=blue!40!white, draw=black](0,0)--(0,0)--(1,0)--(1,0)--(2,0)--(2,0)--(3,0)--(3,0)--(4,0)--(4,4)--(5,4)--(5,8)--(6,8)--(6,8)--(7,8)--(7,11)--(8,11)--(8,13)--(8,0)--(0,0);
 \filldraw[black] (4.5,2) circle (0.1pt) node[anchor=west]{11}; 
  \draw[red] (5,8) circle (15pt); 
\end{tikzpicture}
\begin{tikzpicture}[scale=0.17] 
 \draw (0, 0)--(8, 13)--(8,0)--(0,0); 
 \draw[step=1cm,gray] (1,0) grid (8,1); 
 \draw[step=1cm,gray] (2,1) grid (8,3); 
 \draw[step=1cm,gray] (3,3) grid (8,4); 
 \draw[step=1cm,gray] (4,4) grid (8,6); 
 \draw[step=1cm,gray] (5,6) grid (8,7); 
 \draw[step=1cm,gray] (6,7) grid (8,9); 
 \draw[step=1cm,gray] (7,9) grid (8,10); 
 \filldraw[fill=blue!40!white, draw=black](0,0)--(0,0)--(1,0)--(1,0)--(2,0)--(2,1)--(3,1)--(3,1)--(4,1)--(4,2)--(5,2)--(5,8)--(6,8)--(6,8)--(7,8)--(7,11)--(8,11)--(8,13)--(8,0)--(0,0);
 \filldraw[black] (4.5,2) circle (0.1pt) node[anchor=west]{11}; 
  \draw[red] (5,8) circle (15pt); 
\end{tikzpicture}
\begin{tikzpicture}[scale=0.17] 
 \draw (0, 0)--(8, 13)--(8,0)--(0,0); 
 \draw[step=1cm,gray] (1,0) grid (8,1); 
 \draw[step=1cm,gray] (2,1) grid (8,3); 
 \draw[step=1cm,gray] (3,3) grid (8,4); 
 \draw[step=1cm,gray] (4,4) grid (8,6); 
 \draw[step=1cm,gray] (5,6) grid (8,7); 
 \draw[step=1cm,gray] (6,7) grid (8,9); 
 \draw[step=1cm,gray] (7,9) grid (8,10); 
 \filldraw[fill=blue!40!white, draw=black](0,0)--(0,0)--(1,0)--(1,1)--(2,1)--(2,1)--(3,1)--(3,1)--(4,1)--(4,1)--(5,1)--(5,8)--(6,8)--(6,8)--(7,8)--(7,11)--(8,11)--(8,13)--(8,0)--(0,0);
 \filldraw[black] (4.5,2) circle (0.1pt) node[anchor=west]{11}; 
  \draw[red] (5,8) circle (15pt); 
\end{tikzpicture}
\begin{tikzpicture}[scale=0.17] 
 \draw (0, 0)--(8, 13)--(8,0)--(0,0); 
 \draw[step=1cm,gray] (1,0) grid (8,1); 
 \draw[step=1cm,gray] (2,1) grid (8,3); 
 \draw[step=1cm,gray] (3,3) grid (8,4); 
 \draw[step=1cm,gray] (4,4) grid (8,6); 
 \draw[step=1cm,gray] (5,6) grid (8,7); 
 \draw[step=1cm,gray] (6,7) grid (8,9); 
 \draw[step=1cm,gray] (7,9) grid (8,10); 
 \filldraw[fill=blue!40!white, draw=black](0,0)--(0,0)--(1,0)--(1,1)--(2,1)--(2,1)--(3,1)--(3,1)--(4,1)--(4,2)--(5,2)--(5,8)--(6,8)--(6,9)--(7,9)--(7,9)--(8,9)--(8,13)--(8,0)--(0,0);
 \filldraw[black] (4.5,2) circle (0.1pt) node[anchor=west]{11}; 
  \draw[red] (5,8) circle (15pt); 
\end{tikzpicture}
\begin{tikzpicture}[scale=0.17] 
 \draw (0, 0)--(8, 13)--(8,0)--(0,0); 
 \draw[step=1cm,gray] (1,0) grid (8,1); 
 \draw[step=1cm,gray] (2,1) grid (8,3); 
 \draw[step=1cm,gray] (3,3) grid (8,4); 
 \draw[step=1cm,gray] (4,4) grid (8,6); 
 \draw[step=1cm,gray] (5,6) grid (8,7); 
 \draw[step=1cm,gray] (6,7) grid (8,9); 
 \draw[step=1cm,gray] (7,9) grid (8,10); 
 \filldraw[fill=blue!40!white, draw=black](0,0)--(0,0)--(1,0)--(1,1)--(2,1)--(2,3)--(3,3)--(3,4)--(4,4)--(4,4)--(5,4)--(5,5)--(6,5)--(6,5)--(7,5)--(7,9)--(8,9)--(8,13)--(8,0)--(0,0);
 \filldraw[black] (4.5,2) circle (0.1pt) node[anchor=west]{11}; 
  \end{tikzpicture}
\begin{tikzpicture}[scale=0.17] 
 \draw (0, 0)--(8, 13)--(8,0)--(0,0); 
 \draw[step=1cm,gray] (1,0) grid (8,1); 
 \draw[step=1cm,gray] (2,1) grid (8,3); 
 \draw[step=1cm,gray] (3,3) grid (8,4); 
 \draw[step=1cm,gray] (4,4) grid (8,6); 
 \draw[step=1cm,gray] (5,6) grid (8,7); 
 \draw[step=1cm,gray] (6,7) grid (8,9); 
 \draw[step=1cm,gray] (7,9) grid (8,10); 
 \filldraw[fill=blue!40!white, draw=black](0,0)--(0,0)--(1,0)--(1,0)--(2,0)--(2,3)--(3,3)--(3,3)--(4,3)--(4,4)--(5,4)--(5,4)--(6,4)--(6,6)--(7,6)--(7,11)--(8,11)--(8,13)--(8,0)--(0,0);
 \filldraw[black] (4.5,2) circle (0.1pt) node[anchor=west]{11}; 
  \end{tikzpicture}
\begin{tikzpicture}[scale=0.17] 
 \draw (0, 0)--(8, 13)--(8,0)--(0,0); 
 \draw[step=1cm,gray] (1,0) grid (8,1); 
 \draw[step=1cm,gray] (2,1) grid (8,3); 
 \draw[step=1cm,gray] (3,3) grid (8,4); 
 \draw[step=1cm,gray] (4,4) grid (8,6); 
 \draw[step=1cm,gray] (5,6) grid (8,7); 
 \draw[step=1cm,gray] (6,7) grid (8,9); 
 \draw[step=1cm,gray] (7,9) grid (8,10); 
 \filldraw[fill=blue!40!white, draw=black](0,0)--(0,0)--(1,0)--(1,1)--(2,1)--(2,3)--(3,3)--(3,3)--(4,3)--(4,4)--(5,4)--(5,4)--(6,4)--(6,5)--(7,5)--(7,11)--(8,11)--(8,13)--(8,0)--(0,0);
 \filldraw[black] (4.5,2) circle (0.1pt) node[anchor=west]{11}; 
  \end{tikzpicture}
\begin{tikzpicture}[scale=0.17] 
 \draw (0, 0)--(8, 13)--(8,0)--(0,0); 
 \draw[step=1cm,gray] (1,0) grid (8,1); 
 \draw[step=1cm,gray] (2,1) grid (8,3); 
 \draw[step=1cm,gray] (3,3) grid (8,4); 
 \draw[step=1cm,gray] (4,4) grid (8,6); 
 \draw[step=1cm,gray] (5,6) grid (8,7); 
 \draw[step=1cm,gray] (6,7) grid (8,9); 
 \draw[step=1cm,gray] (7,9) grid (8,10); 
 \filldraw[fill=blue!40!white, draw=black](0,0)--(0,0)--(1,0)--(1,1)--(2,1)--(2,3)--(3,3)--(3,3)--(4,3)--(4,3)--(5,3)--(5,3)--(6,3)--(6,7)--(7,7)--(7,11)--(8,11)--(8,13)--(8,0)--(0,0);
 \filldraw[black] (4.5,2) circle (0.1pt) node[anchor=west]{11}; 
  \end{tikzpicture}
\vfill\null
\columnbreak
\noindent
\begin{tikzpicture}[scale=0.17] 
 \draw (0, 0)--(8, 13)--(8,0)--(0,0); 
 \draw[step=1cm,gray] (1,0) grid (8,1); 
 \draw[step=1cm,gray] (2,1) grid (8,3); 
 \draw[step=1cm,gray] (3,3) grid (8,4); 
 \draw[step=1cm,gray] (4,4) grid (8,6); 
 \draw[step=1cm,gray] (5,6) grid (8,7); 
 \draw[step=1cm,gray] (6,7) grid (8,9); 
 \draw[step=1cm,gray] (7,9) grid (8,10); 
 \filldraw[fill=blue!40!white, draw=black](0,0)--(0,0)--(1,0)--(1,0)--(2,0)--(2,0)--(3,0)--(3,0)--(4,0)--(4,3)--(5,3)--(5,8)--(6,8)--(6,8)--(7,8)--(7,11)--(8,11)--(8,13)--(8,0)--(0,0);
 \filldraw[black] (4.5,2) circle (0.1pt) node[anchor=west]{12}; 
  \draw[red] (5,8) circle (15pt); 
\end{tikzpicture}
\begin{tikzpicture}[scale=0.17] 
 \draw (0, 0)--(8, 13)--(8,0)--(0,0); 
 \draw[step=1cm,gray] (1,0) grid (8,1); 
 \draw[step=1cm,gray] (2,1) grid (8,3); 
 \draw[step=1cm,gray] (3,3) grid (8,4); 
 \draw[step=1cm,gray] (4,4) grid (8,6); 
 \draw[step=1cm,gray] (5,6) grid (8,7); 
 \draw[step=1cm,gray] (6,7) grid (8,9); 
 \draw[step=1cm,gray] (7,9) grid (8,10); 
 \filldraw[fill=blue!40!white, draw=black](0,0)--(0,0)--(1,0)--(1,0)--(2,0)--(2,1)--(3,1)--(3,1)--(4,1)--(4,1)--(5,1)--(5,8)--(6,8)--(6,8)--(7,8)--(7,11)--(8,11)--(8,13)--(8,0)--(0,0);
 \filldraw[black] (4.5,2) circle (0.1pt) node[anchor=west]{12}; 
  \draw[red] (5,8) circle (15pt); 
\end{tikzpicture}
\begin{tikzpicture}[scale=0.17] 
 \draw (0, 0)--(8, 13)--(8,0)--(0,0); 
 \draw[step=1cm,gray] (1,0) grid (8,1); 
 \draw[step=1cm,gray] (2,1) grid (8,3); 
 \draw[step=1cm,gray] (3,3) grid (8,4); 
 \draw[step=1cm,gray] (4,4) grid (8,6); 
 \draw[step=1cm,gray] (5,6) grid (8,7); 
 \draw[step=1cm,gray] (6,7) grid (8,9); 
 \draw[step=1cm,gray] (7,9) grid (8,10); 
 \filldraw[fill=blue!40!white, draw=black](0,0)--(0,0)--(1,0)--(1,1)--(2,1)--(2,1)--(3,1)--(3,1)--(4,1)--(4,1)--(5,1)--(5,8)--(6,8)--(6,9)--(7,9)--(7,9)--(8,9)--(8,13)--(8,0)--(0,0);
 \filldraw[black] (4.5,2) circle (0.1pt) node[anchor=west]{12}; 
  \draw[red] (5,8) circle (15pt); 
\end{tikzpicture}
\begin{tikzpicture}[scale=0.17] 
 \draw (0, 0)--(8, 13)--(8,0)--(0,0); 
 \draw[step=1cm,gray] (1,0) grid (8,1); 
 \draw[step=1cm,gray] (2,1) grid (8,3); 
 \draw[step=1cm,gray] (3,3) grid (8,4); 
 \draw[step=1cm,gray] (4,4) grid (8,6); 
 \draw[step=1cm,gray] (5,6) grid (8,7); 
 \draw[step=1cm,gray] (6,7) grid (8,9); 
 \draw[step=1cm,gray] (7,9) grid (8,10); 
 \filldraw[fill=blue!40!white, draw=black](0,0)--(0,0)--(1,0)--(1,1)--(2,1)--(2,1)--(3,1)--(3,1)--(4,1)--(4,2)--(5,2)--(5,8)--(6,8)--(6,8)--(7,8)--(7,9)--(8,9)--(8,13)--(8,0)--(0,0);
 \filldraw[black] (4.5,2) circle (0.1pt) node[anchor=west]{12}; 
  \draw[red] (5,8) circle (15pt); 
\end{tikzpicture}
\begin{tikzpicture}[scale=0.17] 
 \draw (0, 0)--(8, 13)--(8,0)--(0,0); 
 \draw[step=1cm,gray] (1,0) grid (8,1); 
 \draw[step=1cm,gray] (2,1) grid (8,3); 
 \draw[step=1cm,gray] (3,3) grid (8,4); 
 \draw[step=1cm,gray] (4,4) grid (8,6); 
 \draw[step=1cm,gray] (5,6) grid (8,7); 
 \draw[step=1cm,gray] (6,7) grid (8,9); 
 \draw[step=1cm,gray] (7,9) grid (8,10); 
 \filldraw[fill=blue!40!white, draw=black](0,0)--(0,0)--(1,0)--(1,0)--(2,0)--(2,3)--(3,3)--(3,3)--(4,3)--(4,4)--(5,4)--(5,4)--(6,4)--(6,5)--(7,5)--(7,11)--(8,11)--(8,13)--(8,0)--(0,0);
 \filldraw[black] (4.5,2) circle (0.1pt) node[anchor=west]{12}; 
  \end{tikzpicture}
\begin{tikzpicture}[scale=0.17] 
 \draw (0, 0)--(8, 13)--(8,0)--(0,0); 
 \draw[step=1cm,gray] (1,0) grid (8,1); 
 \draw[step=1cm,gray] (2,1) grid (8,3); 
 \draw[step=1cm,gray] (3,3) grid (8,4); 
 \draw[step=1cm,gray] (4,4) grid (8,6); 
 \draw[step=1cm,gray] (5,6) grid (8,7); 
 \draw[step=1cm,gray] (6,7) grid (8,9); 
 \draw[step=1cm,gray] (7,9) grid (8,10); 
 \filldraw[fill=blue!40!white, draw=black](0,0)--(0,0)--(1,0)--(1,1)--(2,1)--(2,3)--(3,3)--(3,3)--(4,3)--(4,4)--(5,4)--(5,4)--(6,4)--(6,4)--(7,4)--(7,11)--(8,11)--(8,13)--(8,0)--(0,0);
 \filldraw[black] (4.5,2) circle (0.1pt) node[anchor=west]{12}; 
  \end{tikzpicture}
\begin{tikzpicture}[scale=0.17] 
 \draw (0, 0)--(8, 13)--(8,0)--(0,0); 
 \draw[step=1cm,gray] (1,0) grid (8,1); 
 \draw[step=1cm,gray] (2,1) grid (8,3); 
 \draw[step=1cm,gray] (3,3) grid (8,4); 
 \draw[step=1cm,gray] (4,4) grid (8,6); 
 \draw[step=1cm,gray] (5,6) grid (8,7); 
 \draw[step=1cm,gray] (6,7) grid (8,9); 
 \draw[step=1cm,gray] (7,9) grid (8,10); 
 \filldraw[fill=blue!40!white, draw=black](0,0)--(0,0)--(1,0)--(1,1)--(2,1)--(2,3)--(3,3)--(3,4)--(4,4)--(4,4)--(5,4)--(5,4)--(6,4)--(6,5)--(7,5)--(7,9)--(8,9)--(8,13)--(8,0)--(0,0);
 \filldraw[black] (4.5,2) circle (0.1pt) node[anchor=west]{12}; 
  \end{tikzpicture}
\begin{tikzpicture}[scale=0.17] 
 \draw (0, 0)--(8, 13)--(8,0)--(0,0); 
 \draw[step=1cm,gray] (1,0) grid (8,1); 
 \draw[step=1cm,gray] (2,1) grid (8,3); 
 \draw[step=1cm,gray] (3,3) grid (8,4); 
 \draw[step=1cm,gray] (4,4) grid (8,6); 
 \draw[step=1cm,gray] (5,6) grid (8,7); 
 \draw[step=1cm,gray] (6,7) grid (8,9); 
 \draw[step=1cm,gray] (7,9) grid (8,10); 
 \filldraw[fill=blue!40!white, draw=black](0,0)--(0,0)--(1,0)--(1,0)--(2,0)--(2,3)--(3,3)--(3,3)--(4,3)--(4,3)--(5,3)--(5,3)--(6,3)--(6,7)--(7,7)--(7,11)--(8,11)--(8,13)--(8,0)--(0,0);
 \filldraw[black] (4.5,2) circle (0.1pt) node[anchor=west]{12}; 
  \end{tikzpicture}
\vfill\null
\columnbreak
\noindent
\begin{tikzpicture}[scale=0.17] 
 \draw (0, 0)--(8, 13)--(8,0)--(0,0); 
 \draw[step=1cm,gray] (1,0) grid (8,1); 
 \draw[step=1cm,gray] (2,1) grid (8,3); 
 \draw[step=1cm,gray] (3,3) grid (8,4); 
 \draw[step=1cm,gray] (4,4) grid (8,6); 
 \draw[step=1cm,gray] (5,6) grid (8,7); 
 \draw[step=1cm,gray] (6,7) grid (8,9); 
 \draw[step=1cm,gray] (7,9) grid (8,10); 
 \filldraw[fill=blue!40!white, draw=black](0,0)--(0,0)--(1,0)--(1,0)--(2,0)--(2,0)--(3,0)--(3,1)--(4,1)--(4,1)--(5,1)--(5,8)--(6,8)--(6,8)--(7,8)--(7,11)--(8,11)--(8,13)--(8,0)--(0,0);
 \filldraw[black] (4.5,2) circle (0.1pt) node[anchor=west]{13}; 
  \draw[red] (5,8) circle (15pt); 
\end{tikzpicture}
\begin{tikzpicture}[scale=0.17] 
 \draw (0, 0)--(8, 13)--(8,0)--(0,0); 
 \draw[step=1cm,gray] (1,0) grid (8,1); 
 \draw[step=1cm,gray] (2,1) grid (8,3); 
 \draw[step=1cm,gray] (3,3) grid (8,4); 
 \draw[step=1cm,gray] (4,4) grid (8,6); 
 \draw[step=1cm,gray] (5,6) grid (8,7); 
 \draw[step=1cm,gray] (6,7) grid (8,9); 
 \draw[step=1cm,gray] (7,9) grid (8,10); 
 \filldraw[fill=blue!40!white, draw=black](0,0)--(0,0)--(1,0)--(1,1)--(2,1)--(2,1)--(3,1)--(3,1)--(4,1)--(4,1)--(5,1)--(5,8)--(6,8)--(6,8)--(7,8)--(7,9)--(8,9)--(8,13)--(8,0)--(0,0);
 \filldraw[black] (4.5,2) circle (0.1pt) node[anchor=west]{13}; 
  \draw[red] (5,8) circle (15pt); 
\end{tikzpicture}
\begin{tikzpicture}[scale=0.17] 
 \draw (0, 0)--(8, 13)--(8,0)--(0,0); 
 \draw[step=1cm,gray] (1,0) grid (8,1); 
 \draw[step=1cm,gray] (2,1) grid (8,3); 
 \draw[step=1cm,gray] (3,3) grid (8,4); 
 \draw[step=1cm,gray] (4,4) grid (8,6); 
 \draw[step=1cm,gray] (5,6) grid (8,7); 
 \draw[step=1cm,gray] (6,7) grid (8,9); 
 \draw[step=1cm,gray] (7,9) grid (8,10); 
 \filldraw[fill=blue!40!white, draw=black](0,0)--(0,0)--(1,0)--(1,0)--(2,0)--(2,3)--(3,3)--(3,3)--(4,3)--(4,4)--(5,4)--(5,4)--(6,4)--(6,4)--(7,4)--(7,11)--(8,11)--(8,13)--(8,0)--(0,0);
 \filldraw[black] (4.5,2) circle (0.1pt) node[anchor=west]{13}; 
  \end{tikzpicture}
\begin{tikzpicture}[scale=0.17] 
 \draw (0, 0)--(8, 13)--(8,0)--(0,0); 
 \draw[step=1cm,gray] (1,0) grid (8,1); 
 \draw[step=1cm,gray] (2,1) grid (8,3); 
 \draw[step=1cm,gray] (3,3) grid (8,4); 
 \draw[step=1cm,gray] (4,4) grid (8,6); 
 \draw[step=1cm,gray] (5,6) grid (8,7); 
 \draw[step=1cm,gray] (6,7) grid (8,9); 
 \draw[step=1cm,gray] (7,9) grid (8,10); 
 \filldraw[fill=blue!40!white, draw=black](0,0)--(0,0)--(1,0)--(1,1)--(2,1)--(2,3)--(3,3)--(3,3)--(4,3)--(4,4)--(5,4)--(5,4)--(6,4)--(6,5)--(7,5)--(7,9)--(8,9)--(8,13)--(8,0)--(0,0);
 \filldraw[black] (4.5,2) circle (0.1pt) node[anchor=west]{13}; 
  \end{tikzpicture}
\begin{tikzpicture}[scale=0.17] 
 \draw (0, 0)--(8, 13)--(8,0)--(0,0); 
 \draw[step=1cm,gray] (1,0) grid (8,1); 
 \draw[step=1cm,gray] (2,1) grid (8,3); 
 \draw[step=1cm,gray] (3,3) grid (8,4); 
 \draw[step=1cm,gray] (4,4) grid (8,6); 
 \draw[step=1cm,gray] (5,6) grid (8,7); 
 \draw[step=1cm,gray] (6,7) grid (8,9); 
 \draw[step=1cm,gray] (7,9) grid (8,10); 
 \filldraw[fill=blue!40!white, draw=black](0,0)--(0,0)--(1,0)--(1,1)--(2,1)--(2,3)--(3,3)--(3,4)--(4,4)--(4,4)--(5,4)--(5,4)--(6,4)--(6,4)--(7,4)--(7,9)--(8,9)--(8,13)--(8,0)--(0,0);
 \filldraw[black] (4.5,2) circle (0.1pt) node[anchor=west]{13}; 
  \end{tikzpicture}
\begin{tikzpicture}[scale=0.17] 
 \draw (0, 0)--(8, 13)--(8,0)--(0,0); 
 \draw[step=1cm,gray] (1,0) grid (8,1); 
 \draw[step=1cm,gray] (2,1) grid (8,3); 
 \draw[step=1cm,gray] (3,3) grid (8,4); 
 \draw[step=1cm,gray] (4,4) grid (8,6); 
 \draw[step=1cm,gray] (5,6) grid (8,7); 
 \draw[step=1cm,gray] (6,7) grid (8,9); 
 \draw[step=1cm,gray] (7,9) grid (8,10); 
 \filldraw[fill=blue!40!white, draw=black](0,0)--(0,0)--(1,0)--(1,0)--(2,0)--(2,3)--(3,3)--(3,3)--(4,3)--(4,3)--(5,3)--(5,3)--(6,3)--(6,6)--(7,6)--(7,11)--(8,11)--(8,13)--(8,0)--(0,0);
 \filldraw[black] (4.5,2) circle (0.1pt) node[anchor=west]{13}; 
  \end{tikzpicture}
\vfill\null
\columnbreak
\noindent
\begin{tikzpicture}[scale=0.17] 
 \draw (0, 0)--(8, 13)--(8,0)--(0,0); 
 \draw[step=1cm,gray] (1,0) grid (8,1); 
 \draw[step=1cm,gray] (2,1) grid (8,3); 
 \draw[step=1cm,gray] (3,3) grid (8,4); 
 \draw[step=1cm,gray] (4,4) grid (8,6); 
 \draw[step=1cm,gray] (5,6) grid (8,7); 
 \draw[step=1cm,gray] (6,7) grid (8,9); 
 \draw[step=1cm,gray] (7,9) grid (8,10); 
 \filldraw[fill=blue!40!white, draw=black](0,0)--(0,0)--(1,0)--(1,1)--(2,1)--(2,1)--(3,1)--(3,1)--(4,1)--(4,1)--(5,1)--(5,8)--(6,8)--(6,8)--(7,8)--(7,8)--(8,8)--(8,13)--(8,0)--(0,0);
 \filldraw[black] (4.5,2) circle (0.1pt) node[anchor=west]{14}; 
  \draw[red] (5,8) circle (15pt); 
\end{tikzpicture}
\begin{tikzpicture}[scale=0.17] 
 \draw (0, 0)--(8, 13)--(8,0)--(0,0); 
 \draw[step=1cm,gray] (1,0) grid (8,1); 
 \draw[step=1cm,gray] (2,1) grid (8,3); 
 \draw[step=1cm,gray] (3,3) grid (8,4); 
 \draw[step=1cm,gray] (4,4) grid (8,6); 
 \draw[step=1cm,gray] (5,6) grid (8,7); 
 \draw[step=1cm,gray] (6,7) grid (8,9); 
 \draw[step=1cm,gray] (7,9) grid (8,10); 
 \filldraw[fill=blue!40!white, draw=black](0,0)--(0,0)--(1,0)--(1,0)--(2,0)--(2,3)--(3,3)--(3,3)--(4,3)--(4,3)--(5,3)--(5,4)--(6,4)--(6,4)--(7,4)--(7,11)--(8,11)--(8,13)--(8,0)--(0,0);
 \filldraw[black] (4.5,2) circle (0.1pt) node[anchor=west]{14}; 
  \end{tikzpicture}
\begin{tikzpicture}[scale=0.17] 
 \draw (0, 0)--(8, 13)--(8,0)--(0,0); 
 \draw[step=1cm,gray] (1,0) grid (8,1); 
 \draw[step=1cm,gray] (2,1) grid (8,3); 
 \draw[step=1cm,gray] (3,3) grid (8,4); 
 \draw[step=1cm,gray] (4,4) grid (8,6); 
 \draw[step=1cm,gray] (5,6) grid (8,7); 
 \draw[step=1cm,gray] (6,7) grid (8,9); 
 \draw[step=1cm,gray] (7,9) grid (8,10); 
 \filldraw[fill=blue!40!white, draw=black](0,0)--(0,0)--(1,0)--(1,1)--(2,1)--(2,3)--(3,3)--(3,3)--(4,3)--(4,4)--(5,4)--(5,4)--(6,4)--(6,4)--(7,4)--(7,9)--(8,9)--(8,13)--(8,0)--(0,0);
 \filldraw[black] (4.5,2) circle (0.1pt) node[anchor=west]{14}; 
  \end{tikzpicture}
\vfill\null
\columnbreak
\noindent
\begin{tikzpicture}[scale=0.17] 
 \draw (0, 0)--(8, 13)--(8,0)--(0,0); 
 \draw[step=1cm,gray] (1,0) grid (8,1); 
 \draw[step=1cm,gray] (2,1) grid (8,3); 
 \draw[step=1cm,gray] (3,3) grid (8,4); 
 \draw[step=1cm,gray] (4,4) grid (8,6); 
 \draw[step=1cm,gray] (5,6) grid (8,7); 
 \draw[step=1cm,gray] (6,7) grid (8,9); 
 \draw[step=1cm,gray] (7,9) grid (8,10); 
 \filldraw[fill=blue!40!white, draw=black](0,0)--(0,0)--(1,0)--(1,1)--(2,1)--(2,3)--(3,3)--(3,3)--(4,3)--(4,3)--(5,3)--(5,4)--(6,4)--(6,4)--(7,4)--(7,9)--(8,9)--(8,13)--(8,0)--(0,0);
 \filldraw[black] (4.5,2) circle (0.1pt) node[anchor=west]{15}; 
  \end{tikzpicture}
\end{multicols}

\end{example}

\appendix

\section{Completing the proof of Lemmas \ref{1} and \ref{3}} \label{comp}
In this appendix we give the python code needed to carry out the computations needed to complete the proofs of  Lemmas \ref{1} and \ref{3}.  We refer to ``computation 1" as the computation needed to complete Lemma \ref{1} and ``computation 2" as the computation needed to complete Lemma \ref{3}.  The code  should not take more than a few minutes to run on a personal computer.  
\lstset{language=Python}
\lstset{frame=lines}
\lstset{label={lst:code_direct}}
\lstset{basicstyle=\footnotesize}
\begin{lstlisting}
#We use position coordinates to define Dyck paths in this code

import math

#Define the basic statistics
def alpha(a,b,m):
    if a<=b:
        return(min(b-a,m))
    if a>b:
        return(min(a-b-1,m))
   
def alpha0(a,m):
    return(max(0,a-m))

def point(x,m):
    for i in range(len(x)-1,-1,-1):
        if x[i]-x[len(x)-1]>-m:
            pt=i
    return(pt)
       
def pair(x,m):
    pr=len(x)-2
    for i in range(len(x)-3,-1,-1):
        if x[i]-x[i+2]>=-m:
            pr=i
    return(pr)

#Define right and left
def right(x,m):
    pt=point(x,m)
    pr=pair(x,m)
    if pt<=pr+1 and pt<len(x)-1:
        y=x[0:pt+1]
        y+=[x[len(x)-1]+1]
        y+=x[pt+1:len(x)-1]
        return(y)
    if pt>pr+1 or pt==len(x)-1:
        return(x)
   
def left(x,m):
    pr=pair(x,m)
    if x[len(x)-1]-x[pr+1]>=-m-1:
        y=x[0:pr+1]
        y+=x[pr+2:len(x)]
        y+=[x[pr+1]-1]  
        return(y)
    if x[len(x)-1]-x[pr+1]<-m-1:
        return(x)

#Define lowest    
def lowest(x,m):
    Record=[]
    y=None
    while x!=y:
        y=x
        x=right(x,m)
        Record.append(x)
    return(x)

#Define height of A to be height of P such that f(g(P))=A
def height(x):
    greatest=0
    j=0
    for i in range(0,len(x)):
        if x[i]>=greatest:
            greatest=x[i]
            j=i
    return((greatest-1)*(len(x)-1)+j-sum(x))


           
#Generate all relevant Dyck paths: Format [m,d,[a0,a1,...,a_l]]
def generate():
    Dyck_Paths=[]
    #We need to check m=1,2,3,...,20
    for m in range(20,0,-1):
        print('generating Dyck Paths with m = ' + str(m) )
        #Compute lstar and dstar for given m
        lstar=int(math.ceil(20/m+1.001))
        dstar=20
        #Initialize A as all paths for a given m
        A=[[m,0,[0]]]
        #Initialize X as paths with current number of coordinates
        X=[[m,0,[0]]]
        while len(X[0][2])<lstar+1:
            #Initialize Y as paths of with one more coordinate than X
            Y=[]
            for i in range(0,len(X)):
                x=X[i][2]
                last=x[len(x)-1]
                #Create paths with l+1 coordinates from a path with l
                for j in range(0,last+m+1):
                    d=X[i][1]
                    #Compute increase in degr statistic
                    for k in range(1,len(x)):
                        d+=alpha(x[k],j,m)
                    d-=alpha0(j,m)
                    #If degr<=dstar record the path
                    if d<=dstar:
                        Y.append([m,d,x+[j]])
                        A.append([m,d,x+[j]])
            #Reset X to Y
            X=Y
        #Collect all A together in variable Dyck_Paths
        Dyck_Paths+=A
    return(Dyck_Paths)

#Perform computation 1 for single Dyck path of form A=[m,d,[a0,a1,...,a_l]]
def string_okay(A):
    m=A[0]
    d=A[1]
    x=A[2]
    #Compute lstar and dstar for given m
    lstar=int(math.ceil(20/m+1.001))
    dstar=20
    #If l<lstar there is nothing to check
    if len(x)<lstar+1:
        return(True)
    if len(x)==lstar+1:
        #If x is not in T there is nothing to check
        if x[1]>0:
            return(True)
        if x[1]==0:
            #Compute M, h, and b=lowest(x)
            M=m*len(x)*(len(x)-1)/2
            h=height(x)
            b=lowest(x,m)
            #Compare area(b) to its supposed bound
            if sum(b)<=M-h-d:
                return(True)
            if sum(b)>M-h-d:
                print(A)
                return(False)
           
#Check computation 1 for all Dyck Paths in 'dp'        
def all_strings_okay(dp):
    okay=True
    for v in range(0,len(dp)):
        if string_okay(dp[v])==False:
            okay=False
    return(okay)


#Sorting parameter to use later
def sort_rule(X):
    return(X[1]+X[0]/1000)

#Check computation 2 for all Dyck Paths in 'dp'
def conjecture(dp):
    okay=True
    start=0
    end=0
    while end<len(dp):
        #Group together paths
        #if they have the same value of m
        #and they have the same value of l
        while end<len(dp)-1 and dp[end][0]==dp[end+1][0] and\
            len(dp[end][2])==len(dp[end+1][2]):
            end+=1
        end+=1
        #into a variable called 'Dlm'
        Dlm=dp[start:end]
        l=len(Dlm[0][2])-1
        #The case l=0 is trivial
        if l>0:
            m=Dlm[0][0]
            M=m*(l+1)*(l)/2
            #Plus represents all positive monomials appearing in RHS of computation 2
            Plus=[]
            #Minus represents all negative monomials appearing in RHS of computation 2
            Minus=[]
            #All represents all monomials appearing in LHS of computation 2
            All=[]
            #Monomial format is [q degree, sum of q and t degree]
            for i in range(0,len(Dlm)):
                x=Dlm[i][2]
                d=Dlm[i][1]
                a=sum(x)
                #Add a monomial to All for every element of Dlm
                All.append([a,M-d])
                #Check if x is also an element of Tlm
                if x[1]==0:
                    a=sum(x)
                    if a<=M-a-d:
                        #Add monomial string to Plus if coefficients are positive
                        for j in range(a,int(M-a-d+1)):
                            Plus.append([j,M-d])
                    if M-a-d<a:
                        #Add monomial string to Minus if coefficients are negative
                        for j in range(int(M-a-d+1),a):
                            Minus.append([j,M-d])
            #Check that the union of All and Minus gives Plus
            All_Minus=All+Minus
            Plus.sort(key=sort_rule)
            All_Minus.sort(key=sort_rule)
            if Plus==All_Minus:
                print('m=' + str(m) + ', l=' + str(l) + ' is ok')
            if Plus!=All_Minus:
                okay=False
        start=end
    return(okay)



#Generate all relevant Dyck paths
DP=generate()
print(len(DP))


#Perform computation 1
str_check=all_strings_okay(DP)
print(str_check)

#Perform computation 2
con_check=conjecture(DP)
print(con_check)


\end{lstlisting}

\section{Checking Conjecture \ref{con}}
Using this python code the individual may check whether Conjecture \ref{con} holds for any pair of relatively prime numbers.   
\lstset{language=Python}
\lstset{frame=lines}
\lstset{label={lst:code_direct}}
\lstset{basicstyle=\footnotesize}
\begin{lstlisting}
#We use step coordinates to define Dyck paths in this code

import math

#Define the basic statistics
def beta(n,l,i,j,x):
    return(sum(x[i:j+1]) -n*(j-i+1)/(l+1))

#Define the basic statistics

def gamma(n,l,i,j,x):
    b=beta(n,l,i,j,x)
    if b<0:
        g=min(x[i-1],math.floor(-b))
    if b>0:
        g=min(x[i],math.floor(b))
    return(g)

#Generate all relevant Dyck paths: Format [d,(x0,x1,...,x{l-1})]
def generate(l,n):
    #Initialize X as paths with current number of coordinates
    X=[[0,[]]]
    while len(X[0][1])<l:
        #Initialize Y as paths of with one more coordinate than X
        Y=[]
        for i in range(0,len(X)):
            x=X[i][1]
            room=math.floor((len(x)+1)*n/(l+1)-sum(x))
            #Create paths with l+1 coordinates from a path with l
            for j in range(0,int(room)+1):
                d=X[i][0]
                xj=x+[j]
                #Compute increase in degr statistic
                for k in range(1,len(xj)):
                    d+=gamma(n,l,k,len(xj)-1,xj)
                #Append path to Y
                Y.append([d,xj])
        #Reset X to Y
        X=Y
    return(X)

#Sorting parameter to use later
def sort_rule(X):
    return(X[1]+X[0]/1000)

#Test the conjecture for a pair of relatively prime integers p and n.  
def conjecture(p,n):
    #Use l=p-1 to agree with notation in paper
    l=p-1
    #Find the point where the path can come within 1/(l+1) of bounding diagonal
    for q in range(0,l):
        if  round((l+1)*( (q+1)*n/(l+1)-math.floor((q+1)*n/(l+1)) ))==1:
            #Record as [horizontal coordinate, vertical coordinate]
            closest=[q,math.floor((q+1)*n/(l+1))]
    dp=generate(l,n)
    #Compute max possible area, 'M'
    M=0
    for i in range(0,l):
        M+=int(math.floor(i+1)*n/(l+1))
    #Plus represents all positive monomials appearing in RHS of conjecture
    Plus=[]
    #Plus represents all negative monomials appearing in RHS of conjecture
    Minus=[]
    #All represents all monomials appearing in LHS of conjecture
    All=[]
    #Monomial format is [q degree, sum of q and t degree]
    for i in range(0,len(dp)):
        x=dp[i][1]
        d=dp[i][0]
        #Compute area, 'a'
        a=M
        for p in range(0,len(x)):
            a-=sum(x[0:p+1])
        #Add a monomial to All for every element
        All.append([a,M-d])
        #Check if x is in T
        if sum(x[0:closest[0]+1])==closest[1]:
            if a<=M-a-d:
            #Add monomial string to Plus if coefficients are positive
                for j in range(a,int(M-a-d+1)):
                    Plus.append([j,M-d])
            if M-a-d<a:
            #Add monomial string to Minus if coefficients are negative
                for j in range(int(M-a-d+1),a):
                    Minus.append([j,M-d])  
    print(len(All))
    print(len(Plus))
    print(len(Minus))
    All_Minus=All+Minus
    Plus.sort(key=sort_rule)
    All_Minus.sort(key=sort_rule)
    #Check that the union of All and Minus gives Plus
    if Plus==All_Minus:
        okay=True
    if Plus!=All_Minus:
        okay=False
    return(okay)  

#Test the conjecture on your favorite pair of relatively prime integers
print(conjecture(12,17))
\end{lstlisting}

\section*{Funding and/or Conflicts of interests/Competing interests}
The author has no conflicts of interest or competing interests to declare.  The author was  supported by ISF 711/18 during some of this research.


\begin{thebibliography}{11}



\bibitem{GH02}
A.M. Garsia and James Haglund \emph{A proof of the q, t-Catalan positivity conjecture} Discrete Math. 256 (2002), no. 3, 677-717
\bibitem{HA03}
James Haglund \emph{Conjectured statistics for the q, t-Catalan numbers} Adv. Math. 175 (2003), no. 2, 319-334
\bibitem{HHLRU}
 James Haglund, Mark Haiman, Nicholas A. Loehr, J. B. Remmel, and A. Ulyanov   \emph{A combinatorial formula for the character of the diagonal coinvariants} Duke Math. J., 126(2):195-232, 2005.
\bibitem {HL05}
James Haglund and Nicholas A. Loehr \emph{A conjectured combinatorial formula for the Hilbert series for diagonal harmonics} Discrete Math. 298 (2005), no. 1-3, 189-204
\bibitem{HA08}
 James Haglund \emph{The q,t-Catalan numbers and the space of diagonal harmonics}, volume 41 of  \emph{University Lecture Series} American Mathematical Society, Providence, RI, 2008.
\bibitem{LW09}
 Nicholas A. Loehr and Gregory S. Warrington \emph{A continuous family of partition statistics equidistributed with length} J. Combin. Theory Ser. A 116 (2009), no. 2, 379–403.
\bibitem{LLL18}
Kyungyong Lee, Li Li, and Nicholas A. Loehr \emph{A Combinatorial Approach to the Symmetry of q,t-Catalan Numbers}, SIAM J. Discrete Math (SIDMA). 32 (2018) no.1, 191--232.
\bibitem{HLLL20}
Seongjune Han, Kyungyong Lee, Li Li, and Nicholas A. Loehr \emph{Chain Decompositions of q, t-Catalan Numbers via Local Chains}, Ann. Comb. 24, 739–765 (2020).
\bibitem{HLLL22}
Seongjune Han, Kyungyong Lee, Li Li, and Nicholas A. Loehr \emph{Chain Decompositions of q, t-Catalan Numbers: Tail Extensions and Flagpole Partitions}, Ann. Comb. (2022)  https://doi.org/10.1007/s00026-022-00590-7.







 \end{thebibliography}
\end{document}